\documentclass[11pt]{amsart}
\usepackage{amsfonts,amssymb,amscd,amstext,mathrsfs,mathtools}
\usepackage[utf8]{inputenc}
\usepackage{hyperref}
\usepackage{verbatim}
\usepackage{graphics}
\usepackage{graphicx}
\usepackage{t1enc}

\usepackage{times}
\usepackage{enumerate}
\usepackage[up,bf]{caption}
\usepackage{color}

\input xy
\xyoption{all}

\numberwithin{equation}{section}

\newtheorem{theorem}{Theorem}[section] 
\newtheorem{proposition}[theorem]{Proposition} 
 
\newtheorem{lemma}[theorem]{Lemma} 
\newtheorem{corollary}[theorem]{Corollary} 

\theoremstyle{definition} 
\newtheorem{definition}[theorem]{Definition} 
\newtheorem{remark}[theorem]{Remark} 
 
\newtheorem{problem}[theorem]{Problem}

\newcommand\Bcal{\mathcal{B}} 
\newcommand\Ccal{\mathcal{C}}

\newcommand\Pcal{\mathcal{P}}

\newcommand\Ascr{\mathscr{A}} 
 
\newcommand\Cscr{\mathscr{C}}

\newcommand\Oscr{\mathscr{O}}

\newcommand\B{\mathbb{B}} 
\newcommand\C{\mathbb{C}} 
\newcommand\D{\overline{\mathbb D}} 
\newcommand\CP{\mathbb{CP}} 
\renewcommand\D{\mathbb D}

\newcommand\N{\mathbb{N}} 
 
\newcommand\R{\mathbb{R}}

\newcommand\Z{\mathbb{Z}}

\newcommand\igot{\mathfrak{i}}

\renewcommand\igot{\mathfrak{i}}

\renewcommand\imath{\igot}

\newcommand\hra{\hookrightarrow}

\newcommand\wt{\widetilde} 
\newcommand\wh{\widehat} 
\newcommand\di{\partial} 
\newcommand\dibar{\overline\partial}

\newcommand\dist{\mathrm{dist}}

\newcommand\Id{\mathrm{Id}}

\newcommand\supp{\mathrm{supp}}

\newcommand\Hom{\mathrm{Hom}}
\newcommand\Lin{\mathrm{Lin}}

\newcommand\Jscr{\mathscr{J}}
\newcommand\Oscrc{\overline{\mathscr{O}}}
\newcommand\Oscrcl{\Oscrc_{\mathrm{loc}}}
\newcommand\boldA{\mathbf A}

\def\dist{\mathrm{dist}}

\def\Ell1{\mathrm{Ell_1}} 
\def\CEll1{\mathrm{CEll_1}}

\def\Jst{J_{\mathrm{st}}} 

\makeatletter
\@namedef{subjclassname@2020}{\textup{2020} Mathematics Subject Classification}
\makeatother

\begin{document} 

\title{Runge and Mergelyan theorems on families of \\ 
open Riemann surfaces} 

\author{Franc Forstneri{\v c}}

\address{Faculty of Mathematics and Physics, University of Ljubljana, Jadranska 19, 1000 Ljubljana, Slovenia}

\address{Institute of Mathematics, Physics, and Mechanics, Jadranska 19, 1000 Ljubljana, Slovenia}

\email{franc.forstneric@fmf.uni-lj.si}


\date{31 January 2025. This version: 24 May 2026} 


\begin{abstract}
Given a smooth open oriented surface \(X\), endowed with 
a family of complex structures \(\{J_b\}_{b\in B}\) of some 
H\"older class and depending 
continuously or smoothly on the parameter \(b\) in a suitable
topological space \(B\), we construct continuous or smooth families 
\(F_b:X\to Y\), \(b\in B\), of \(J_b\)-holomorphic maps to any 
Oka manifold \(Y\), with approximation on a suitable family 
of compact Runge sets in \(X\). 
Along the way, we prove Runge and Mergelyan approximation 
theorems and Weierstrass interpolation theorem for functions
on such families. We include applications to the construction 
of families of directed holomorphic immersions 
and conformal minimal immersions to Euclidean spaces.
\end{abstract}

\maketitle

\setcounter{tocdepth}{1}
\tableofcontents


%
%

\section{Introduction and main results}\label{sec:intro} 

In this paper, we construct families of holomorphic maps from families 
of open Riemann surfaces to Oka manifolds. Our main results are applied 
to the construction of families of directed holomorphic immersions 
and of conformal minimal immersions from families of open 
Riemann surfaces to Euclidean spaces.
We expect further applications of the techniques developed in the paper. 

A complex manifold $Y$ is said to be an Oka manifold 
if maps from any Stein manifold $X$ 
to $Y$ satisfy all forms of the homotopy principle, called the 
Oka principle in this context. Basically, this means that holomorphic maps 
$X\to Y$ satisfy the same approximation and extension properties 
as holomorphic functions $X\to \C$ in the absence of topological obstructions; see \cite[Theorem 5.4.4]{Forstneric2017E} for a summary
statement. Oka manifolds appear naturally in 
many existence results in complex geometry. 
Every complex homogeneous manifold and, more
generally, every Gromov elliptic manifold is an Oka manifold 
(see Grauert \cite{Grauert1958MA} and Gromov \cite{Gromov1989}).
Further examples and properties of Oka manifolds can be found in \cite{Forstneric2017E,Forstneric2023Indag,ForstnericKusakabe2025,ForstnericLarusson2011,ForstnericLarusson2025MRL,Larusson2004,Larusson2005} 
and in other sources.

It is often desirable to construct families of holomorphic maps 
depending continuously or smoothly on parameters.
For maps from a Stein space $X$ with a fixed complex structure, 
see \cite[Theorem 2.8.4]{Forstneric2017E} for the parametric 
Cartan--Oka--Weil theorem for functions $X\to\C$ and 
\cite[Theorem 5.4.4]{Forstneric2017E} for the parametric Oka principle
for maps from $X$ to any Oka manifold. In those results, the maps depend 
continuously on the parameter in a compact Hausdorff space $B$, 
and they remain unchanged for parameter values in a closed subset 
of $B$ for which they are already holomorphic on all of $X$. 

In the present paper, we consider the more general situation 
where not only the maps but also the complex structures 
on the source manifold depend on a parameter 
in a suitable topological space. 
In this first work on the subject, we limit ourselves to the 
case of a smooth open oriented surface $X$  
endowed with a family $\Jscr=\{J_b\}_{b\in B}$ of complex structures.
Recall that every open Riemann surface is a Stein manifold
according to Behnke and Stein \cite{BehnkeStein1949}.
A compact set $K$ in an open Riemann surface $(X,J)$
is called {\em Runge} if the complement $X\setminus K$ has no
relatively compact connected components.
Our main result, Theorem \ref{th:Oka}, is an Oka principle saying
that for suitable parameter spaces $B$ and regularity conditions
on the family of complex structures $\Jscr=\{J_b\}_{b\in B}$ on $X$, 
every family of continuous maps $f_b: X \to Y$ $(b\in B)$ 
to an Oka manifold $Y$ 
is homotopic to a family of $J_b$-holomorphic maps $F_b:X\to Y$, 
with approximation on a suitable family of compact Runge subsets 
$K_b\subset X$ on which the given maps $f_b$ are already holomorphic.  
Our method also applies to families of products of open 
Riemann surfaces with a fixed Stein manifold; see Theorem \ref{th:Okabis}. 
Furthermore, we obtain Mergelyan-type theorems for families 
of maps to complex manifolds; see Theorems \ref{th:Mergelyan-manifold} 
and \ref{th:Mergelyan-admissible}. 
On the way to the main results, we obtain approximation of functions 
on families of open Riemann surfaces --- 
the Runge Theorem \ref{th:Runge} and the Mergelyan Theorem 
\ref{th:Mergelyan}. Their proofs are simpler 
since one can use partitions of unity,
instead of dealing with homotopies as we must do 
for nonlinear target manifolds, 
and in this case $B$ may be an arbitrary paracompact Hausdorff space. 
In particular, these two results apply to the universal family
of complex structures on a given smooth open orientable surface.

Our main results, combined with techniques of Gromov's convex integration theory, enable the construction of families of holomorphic curves 
with prescribed conformal types having additional properties 
(immersed, directed by a conical subvariety of $\C^n$, etc.), 
and of families of immersed minimal surfaces  
with prescribed conformal types in Euclidean spaces $\R^n$ for $n\ge 3$.
A sample of such applications is given in Section \ref{sec:directed},
where we also indicate several further problems which can likely 
be treated by these methods. 

%
%
We now turn to the detailed presentation.
Let $X$ be a smooth, connected, orientable surface 
with a countable base of topology, 
which will be an open surface in most results. 
We endow $X$ with a smooth Riemannian metric, which is used
to define H\"older spaces of functions or maps from domains in $X$;
see \eqref{eq:alphanorm}. A complex structure on $X$ 
is given by an endomorphism $J$ of its tangent bundle $TX$ 
satisfying $J^2=-\Id$. Thus, $J$ is a section of the smooth vector bundle
$T^*X\otimes TX\to X$ whose fibre over 
$x\in X$ is the space $\Hom(T_xX,T_xX)$ 
of linear maps $T_x X\mapsto T_xX$. 
(Note that this bundle is trivial if $X$ is an open orientable surface.)
A differentiable function $f:X\to\C$ is said to be {\em $J$-holomorphic} 
if the Cauchy--Riemann equation 
$df_x \circ J_x = \imath\, df_x$ holds for every $x\in X$, where
$\imath=\sqrt{-1}$. We say that $J$ is of local H\"older class 
$\Cscr^{(k,\alpha)}$ for some $k\in\Z_+=\{0,1,2,\ldots\}$ and $0<\alpha<1$
if for every relatively compact domain $\Omega\subset X$, the restriction
$J|_\Omega\in \Gamma^{(k,\alpha)}(\Omega,T^*\Omega\otimes T\Omega)$ 
is a section of class $\Cscr^{(k,\alpha)}(\Omega)$ of the restricted vector 
bundle $T^*\Omega\otimes T\Omega\to \Omega$. For such $J$, 
there is an atlas $\{(U_i,\phi_i)\}_i$ of 
open sets $U_i\subset X$ with $\bigcup_{i} U_i=X$ and 
$J$-holomorphic charts $\phi_i :U_i \to \phi_i(U_i)\subset \C$ 
of class $\Cscr^{(k+1,\alpha)}(U_i)$; see Theorem \ref{th:isothermal}.
Since the transition maps $\phi_i\circ\phi_j^{-1}$
are biholomorphic in the standard structure $\Jst$ on $\C$, 
$J$ determines on $X$ the structure of a Riemann surface
$(X,J)$ whose underlying smooth structure is $\Cscr^{(k+1,\alpha)}$ 
compatible with the smooth structure on $X$. 
%
%
In fact, the inverse of a diffeomorphism of local class 
$\Cscr^{(k+1,\alpha)}$ is again of the same class; 
see Norton \cite{Norton1986} and 
Bojarski et al.\ \cite[Theorem 2.1]{BojarskiAll2005}.

Let $l\in \Z_+$ be an integer, and let $B$ be a topological space
which is a manifold of class $\Cscr^l$ if $l\in \N=\{1,2,\ldots\}$.  
A family $\Jscr=\{J_b\}_{b\in B}$ of complex structures on $X$
is said to be of class $\Cscr^{l,(k,\alpha)}$ if 
for any relatively compact domain $\Omega\Subset X$ the map 
$B\ni b\mapsto J_b|_{\Omega}\in \Gamma^{(k,\alpha)}(\Omega,T^*\Omega\otimes T\Omega)$ is of class $\Cscr^l$.
Such a family $\Jscr$ can equivalently be given by 
a family $\{\mu_b\}_{b\in B}$ of maps from 
$X$ to the unit disc $\D=\{\zeta\in\C:|\zeta|<1\}$
of the same smoothness class $\Cscr^{l,(k,\alpha)}$; 
see Section \ref{sec:Beltrami}.
Following Kodaira and Spencer \cite{KodairaSpencer1958}
and Kirillov \cite{Kirillov1976}, the collection $\{(X,J_b)\}_{b\in B}$
is called a {\em family of Riemann surfaces} of class 
$\Cscr^{l,(k,\alpha)}$. Consider the projection $\pi:Z=B\times X\to B$. 
We endow each fibre $X_b=\pi^{-1}(b)$ 
with the complex structure $J_b$. 
A map $f:B\times X\to Y$ to a complex manifold $Y$ 
is said to be {\em $\Jscr$-holomorphic} 
if the map $f_b=f(b,\cdotp):X_b\to Y$ is $J_b$-holomorphic
for every $b\in B$. Assuming that the family $\Jscr$ is of class
$\Cscr^{l,(k,\alpha)}$, the space $Z=B\times X$ admits fibre 
preserving $\Jscr$-holomorphic charts of class $\Cscr^{l,(k+1,\alpha)}$ 
with values in $B\times \C$ (see Theorem \ref{th:Beltrami1}).
If $0<l\le k+1$ then $Z=B\times X$ endowed 
with such an atlas is a {\em mixed manifold} of class $\Cscr^l$
in the sense of Douady \cite{Douady1962} and 
Jurchescu \cite{Jurchescu1979,Jurchescu1988RRMPA},
and a Levi-flat CR manifold of CR-dimension one in the 
sense of the Cauchy--Riemann geometry; 
see \cite{MongodiTomassini2016} and \cite{BaouendiEbenfeltRothschild1999}. 
In Jurchescu's papers, maps which are holomorphic 
on complex leaves of a mixed manifold are called {\em morphic}, 
while in CR geometry they are called CR maps.

Recall that a topological space is said to be paracompact if every 
open cover has an open locally finite refinement. 
A Hausdorff space is paracompact if and only if it admits a 
locally finite partition of unity subordinate to any 
open cover. Every metrisable space is
Hausdorff and paracompact \cite{StoneAH1948,RudinME1969}.

Our first result, which is proved in Section \ref{sec:Runge}, 
extends the classical Runge--Behnke--Stein approximation theorem 
on open Riemann surfaces \cite{Runge1885,BehnkeStein1949}, 
combined with the Weierstrass--Florack interpolation theorem 
\cite{Weierstrass1885,Florack1948},
to families of open Riemann surfaces.

%
%
\begin{theorem}
\label{th:Runge}
Assume that $l\in \Z_+$, $B$ is a paracompact Hausdorff space if $l=0$ 
and a manifold of class $\Cscr^l$ if $l>0$,
$X$ is a smooth open surface, 
$\Jscr=\{J_b\}_{b\in B}$ is a family of complex structures on $X$ 
of class $\Cscr^{l,(k,\alpha)}$ $(k\in \Z_+,\ l\le k+1,\ 0<\alpha<1$), 
$K$ is a compact Runge subset of $X$, 
$A$ is a closed discrete subset of $X$,
$U\subset B\times X$ is an open set containing $B\times (K\cup A)$,  
$f:U\to\C$ is a function of class $\Cscr^{l,0}$ 
such that $f_b=f(b,\cdotp)$ is $J_b$-holomorphic on 
$U_b=\{x\in X:(b,x)\in U\}$ for every $b\in B$, and $r\in \{0,1,\ldots,k+1\}$.
Then, $f$ is of class $\Cscr^{l,(k+1,\alpha)}$ on $U$ and there is a function 
$F:B\times X\to\C$ of class $\Cscr^{l,(k+1,\alpha)}$ satisfying the
following conditions.
\begin{enumerate}[\rm (a)]
\item The function $F_b=F(b,\cdotp):X\to\C$ is $J_b$-holomorphic
for every $b\in B$.
\item F approximates $f$ as closely as desired 
in the fine $\Cscr^{l,(k+1,\alpha)}$ topology on $B\times K$.
\item $F_b-f_b$ vanishes to order $r$ at every point $a\in A$
for every $b\in B$.
\end{enumerate}
\end{theorem}

The reason for assuming $l\le k+1$ will become 
evident in the proof of Lemma \ref{lem:Xholomorphic}.
As explained above, the complex structures $J_b$ in the theorem 
are compatible with one another only to order $k+1$, which necessitates 
the assumption $r\le k+1$ in condition (c). 
See also Corollary \ref{cor:Runge} for approximation on 
variable families of compact $J_b$-convex subsets $K_b\subset X$,
$b\in B$.

%
%
\begin{remark}\label{em:Cartanmanifold}
If $B$ is a manifold of class $\Cscr^l$ then 
Theorems \ref{th:Runge} and \ref{th:Beltrami1}
imply that the manifold $B\times X$, endowed with the
family of complex structures $\{J_b\}_{b\in B}$ 
as in the theorem, is a {\em Cartan manifold} 
of class $\Cscr^{l}$ in the sense of Jurchescu 
\cite[Sect.\ 6]{Jurchescu1988RRMPA}; see also his papers
\cite{Jurchescu1979,Jurchescu1983,Jurchescu1984,Jurchescu1988AUB}
as well as \cite{FlondorPascu1983,Vajaitu2001}. 
Cartan manifolds are analogues of Stein manifolds in 
the category of mixed manifolds. In the cited papers, 
Cartan manifolds are assumed to be of class $\Cscr^\infty$ 
in order to avoid problems which appear under the finite 
regularity assumptions, such as those in the present paper.
\end{remark}

A more sophisticated approximation theorem was proved 
by Mergelyan \cite{Mergelyan1951} in 1951. Its original version  
says that a continuous function on a compact Runge set $K\subset \C$ 
which is holomorphic in the interior $\mathring K$ of $K$ 
is a uniform limit on $K$ of holomorphic polynomials. 
In view of Runge's theorem, 
the main new point is that every function in the algebra
$\Ascr(K)$ of continuous functions on $K$ which are holomorphic 
in $\mathring K$ can be approximated 
by functions holomorphic on open neighbourhoods
of $K$. A compact set $K$ in a Riemann surface $X$ for which this
condition holds is said to have the {\em Mergelyan property}.
For results on this subject we refer to the surveys 
\cite[Sect.\ 2]{FornaessForstnericWold2020} and \cite{Gaier1987}.
The following Mergelyan theorem for families of 
complex structures on a smooth surface is proved in 
Section \ref{sec:Runge}.  

%
%
\begin{theorem}
\label{th:Mergelyan}
Assume that $X$ is a smooth oriented surface without boundary, 
$B$ is a paracompact Hausdorff space, 
$\Jscr=\{J_b\}_{b\in B}$ is a continuous family of 
complex structures on $X$ of H\"older class 
$\Cscr^\alpha$ for some $0<\alpha<1$, 
$K$ is compact set in $X$ such that, for some $c>0$ and a 
Riemannian distance function on $X$, each relatively compact 
connected component of $X\setminus K$ has diameter at least $c$,
$A$ is a finite subset of $\mathring K$, 
and $f: B\times K \to\C$ is a continuous function  
such that $f_b=f(b,\cdotp):K\to \C$ 
is $J_b$-holomorphic on $\mathring K$ for every $b\in B$.
Given a continuous function $\epsilon:B\to (0,+\infty)$, 
there is a continuous $\Jscr$-holomorphic function $F$ on a 
neighbourhood $U\subset B\times X$ of $B\times K$ 
such that for every $b\in B$ we have 
$\sup_{x\in K}|F_b(x)-f_b(x)|<\epsilon(b)$ and 
$F_b-f_b$ vanishes to order $1$ in every point $a\in A$.
\end{theorem}

The condition on the set $K$ in the above theorem ensures
that $K$ has the Mergelyan property with respect to any complex
structure on $X$; see Bishop \cite{Bishop1958PJM}.
See also Corollary \ref{cor:Mergelyan} for Mergelyan approximation
on certain variable families of holomorphically convex sets $K_b\subset X$,
$b\in B$. 

%
%
\begin{remark}\label{rem:universal}
The $l=0$ case of Theorems \ref{th:Runge} and \ref{th:Mergelyan} applies 
to the universal or tautological family on $X$. 
(I wish to thank the referee for pointing this out.) 
The universal family has parameter space $B_u$ consisting of pairs 
$(J,f)$, where $J$ is a complex structure 
on $X$ of class $\Cscr^{(k,\alpha)}$ and $f$ is a 
continuous function on $K$, $J$-holomorphic 
on a neighbourhood of $K$ (in case of Theorem \ref{th:Runge}) 
or on the interior of $K$ (in case of Theorem \ref{th:Mergelyan}). 
The fibre of the universal family over $b=(J,f)\in B_u$ is the 
Riemann surface $(X,J)$. There is a natural metrizable topology on $B_u$,
so it is Hausdorff and paracompact by a theorem of Stone \cite{StoneAH1948}
(see also Rudin \cite{RudinME1969}).  
Any family $\{(J_b, f_b)\}_{b\in B}$ as in Theorem \ref{th:Runge}
or \ref{th:Mergelyan} can be pulled back from this universal family. 
A solution of the approximation problem for the universal family will 
therefore pull back to a solution for a general family, at least if $\epsilon$ 
is constant. For this result, one does not need any assumption on 
the topological space $B$. 
\end{remark}

We now introduce the parameter spaces used in 
our main result, Theorem \ref{th:Oka}. 

%
%
\begin{definition}\label{def:ENR}
In the following, all topological spaces are assumed to be metrizable.
\begin{enumerate}[\rm (i)] 
\item A space $B$ is an {\em absolute neighbourhood
retract} (ANR) if, whenever $B$ is a closed subset of a space $B'$,
then $B$ is a retract of a neighbourhood of $B$ in $B'$
(a neighbourhood retract).
\item A space $B$ is a {\em Euclidean neighbourhood
retract} (ENR) if it admits a closed 
topological embedding $\iota: B\hra \R^n$ 
for some $n$ whose image $\iota(B)\subset \R^n$ 
is a neighbourhood retract.
\item 
A space $B$ is a {\em local ENR}
if every point of $B$ has an ENR neighbourhood. 
\end{enumerate}
\end{definition}

We refer to Marde\v  si\'c \cite{Mardesic1999} for the theory of ANRs and ENRs.
Clearly, every local ENR is locally compact.
The class of ENRs includes many geometrically
relevant classes of spaces such as certain CW complexes.
Note that CW complexes generalise both manifolds and simplicial complexes, 
and they have particular significance for algebraic topology; 
see Hatcher \cite{Hatcher2002} and May \cite{May1999}. 
A CW complex is locally compact if and only if its  
collection of closed cells is locally finite, if and only if
it is metrizable (see Fritsch and Piccinini \cite[Theorem B]{FritschPiccinini1990}). 
If a CW complex $B$ can be embedded in a Euclidean space
$\R^m$, then $B$ has at most countably many cells, it is locally compact 
and has dimension at most $m$ \cite[Theorem D]{FritschPiccinini1990}.
Conversely, every countable locally compact CW-complex $B$ of finite dimension $m$ is an ENR. 
Indeed, by \cite[Theorem A]{FritschPiccinini1990}  
such $B$ admits a closed embedding in $\R^{2m+1}$. 
Since every metrizable CW-complex is an ANR \cite{Dugundji1952}, 
the image of the embedding is a neighbourhood retract, 
so $B$ is an ENR. In particular, every finite CW complex is an ENR  
\cite[Corollary A.10]{Hatcher2002}.

A topological space is said to be {\em $\sigma$-compact} if it is 
the union of countably many compact subspaces.
According to Michael \cite{Michael1957PAMS},
every locally compact and $\sigma$-compact Hausdorff
space is paracompact.  

The main result of this paper is the following Oka principle with approximation
for maps from families of open Riemann surfaces to Oka manifolds.
It is proved in Section \ref{sec:Oka}. 

%
%
\begin{theorem} \label{th:Oka}
Assume the following:
\begin{enumerate}[\rm (a)]
\item $B$ is a 
$\sigma$-compact Hausdorff local ENR (see Definition \ref{def:ENR}). 
In particular, $B$ may be a finite CW complex or a 
countable locally compact CW-complex of finite dimension.
\item $X$ is a smooth open surface and $\pi:B\times X\to B$ is
the projection. 
\item $\Jscr=\{J_b\}_{b\in B}$ is a continuous family of complex structures on 
$X$ of H\"older class $\Cscr^\alpha$, $0<\alpha<1$.  
\item $K\subset B\times X$ is a closed subset such that the projection 
$\pi|_K:K\to B$ is proper, and for every $b\in B$ the fibre 
$K_b=\{x\in X:(b,x)\in K\}$ is a compact Runge set in $X$, 
possibly empty. (In the sequel, such a set $K$ is called a proper Runge set.)
\item $Y$ is an Oka manifold endowed 
with a distance function $\dist_Y$ inducing its manifold topology.
\item $f:B\times X \to Y$ is a continuous map, 
and there is an open set $U\subset B\times X$ containing $K$ such that 
$f_b=f(b,\cdotp):X\to Y$ is $J_b$-holomorphic on 
$U_b=\{x\in X:(b,x)\in U\}$ for every $b\in B$.
\end{enumerate}
Given a continuous function $\epsilon:B\to (0,+\infty)$,  
there are a neighbourhood $U'\subset U$ of $K$ and a homotopy 
$f_t:B\times X\to Y$ $(t\in I=[0,1])$ satisfying the following conditions.
\begin{enumerate}[\rm (i)]
\item $f_0=f$.
\item The map $f_{t,b}=f_t(b,\cdotp):X\to Y$ is $J_b$-holomorphic
on $U'_b\supset K_b$ for every $b\in B$ and $t\in I$.
\item $\sup_{x\in K_b}\dist_Y(f_t(b,x),f(b,x))<\epsilon(b)$
for every $b\in B$ and $t\in I$.
\item The map $F=f_1$ is such that 
$F_b=F(b,\cdotp):X\to Y$ is $J_b$-holomorphic for every $b\in B$.
\item If $Q$ is a closed subset of $B$ and $U_b=X$ for all $b\in Q$,
then the homotopy $f_{t,b}$ $(t\in I)$ 
can be chosen to be fixed for every $b\in Q$,
and in particular $F=f$ on $Q\times X$.
\end{enumerate}
If $B$ is a manifold of class $\Cscr^l$ $(l\in \N)$, 
the set $Q$ in (v) is a closed $\Cscr^l$ submanifold of $B$, 
the family $\Jscr=\{J_b\}_{b\in B}$ is of class $\Cscr^{l,(k,\alpha)}$
where $l\le k+1$, and the map $f:B\times X\to Y$ is 
$\Jscr$-holomorphic on a neighbourhood 
$U$ of $K$ and $f|_U \in \Cscr^{l,0}(U,Y)$,  
then $f|_U\in \Cscr^{l,(k+1,\alpha)}(U,Y)$ and there is a homotopy 
$f_t:B\times X\to Y$ $(t\in I)$ which is of class $\Cscr^{l,(k+1,\alpha)}$ 
on a neighbourhood of $K$, it satisfies conditions (i)--(v), 
$f_t$ approximates $f$ in the fine $\Cscr^{l,(k+1,\alpha)}$-topology 
on $K$ to a desired precision uniformly in $t\in I$,  
and $F=f_1:B\times X\to Y$ is of class  
$\Cscr^{l,(k+1,\alpha)}(B\times X,Y)$ and $\Jscr$-holomorphic. 
\end{theorem}

The proof of Theorem \ref{th:Oka}, given in Section \ref{sec:Oka}, 
also applies to products $X\times Z$, 
where $X$ is a smooth open surface endowed with 
a family $\Jscr=\{J_b\}_{b\in B}$ of complex structures as above 
and $Z$ is a Stein manifold with a fixed 
complex structure; see Theorem \ref{th:Okabis}. 
Essentially the same proof, combined with Theorem \ref{th:Mergelyan}, 
yields Mergelyan approximation for maps from families of open 
Riemann surfaces to an arbitrary complex manifold; 
see Theorems \ref{th:Mergelyan-manifold} 
and \ref{th:Mergelyan-admissible}.

%
%
\begin{remark}\label{rem:Q}
Condition in (v) of Theorem \ref{th:Oka} can be strenghtened as follows. 
\begin{enumerate}[\rm (v')]
\item If $Q$ is a closed subset of $B$ and $f$ is $\Jscr$-holomorphic 
on $U\cup (Q\times X)$ then the homotopy $f_{t,b}$ $(t\in I)$ 
can be chosen to be fixed for every $b\in Q$,
and in particular $F=f$ on $Q\times X$.
\end{enumerate}
Thus, we need not assume that the open set $U\subset B\times X$ contains 
$Q\times X$. The sets $K$ and $U$ may be empty if we are not 
interested in an interpolation statement. 
The proof of (v') requires a minor modification in the induction step 
and is given (in the more general context of tame families of Stein structures
on a smooth manifold $X$ of any dimension)
in \cite[proof of Theorems 6.1 and 6.3]{ForstnericSigurdardottir2025}.
The same remark applies to Theorem \ref{th:Runge}. 
\end{remark}

Model Oka manifolds are the complex Euclidean spaces $\C^n$.
For $Y=\C^n$, Theorem \ref{th:Oka} generalises  
the approximation statement in Theorem \ref{th:Runge}
to any proper Runge subset $K\subset B\times X$. This generalisation
holds for any paracompact Hausdorff parameter space $B$ 
if we replace condition (d) in Theorem \ref{th:Oka} by the
condition that the fibres $K_b$ of $K$ are upper
semicontinuous with respect to $b\in B$ and $K_b$ is 
Runge in $(X,J_b)$ for every $b\in B$; 
see Definition \ref{def:Runge} and Corollary \ref{cor:Runge}.
If $B$ is locally compact then these two conditions are 
easily seen to be equivalent.
We do not know whether the interpolation statement 
in Theorem \ref{th:Runge} carries over to the more general setting
in Theorem \ref{th:Oka}. 

The condition on the parameter space $B$ to be a local ENR 
enables us to reduce the proof of Theorem \ref{th:Oka} to the Oka principle
in \cite[Theorem 5.4.4]{Forstneric2017E}.
It is likely that Theorem \ref{th:Oka} holds for more general parameter spaces.
Such a generalisation would require a proof from 
the first principles, developing the gluing techniques of Oka theory 
(see \cite[Chapter 5]{Forstneric2017E}) on Levi-flat CR 
foliations and laminations. We shall not pursue this approach 
in the present paper. 

The Oka theory has recently been developed for maps from 
open Riemann surfaces to the class of Oka-1 manifolds,  
which properly contains the class of Oka manifolds; see
\cite{AlarconForstneric2025MZ,ForstnericLarusson2024Oka1}.
However, I do not know whether this bigger class 
of manifolds can be used in Theorem \ref{th:Oka} since its proof relies on 
the Oka principle for maps from higher dimensional Stein manifolds
to Oka manifolds. 

%
%
Theorem \ref{th:Oka} and the related results in Section \ref{sec:Oka}  
apply in particular if $B$ is a 
Teichm\"uller space $T(g,k)$ of $k$-punctured compact Riemann
surfaces of genus $g$ and $\Jscr=\{J_b\}_{b\in B}$ is the associated 
universal family of complex structures on $X$. 
The space $T(g,k)$ is finite dimensional and 
\[
	(X,J_b)=\wh X_b \setminus \{p_1(b),\ldots,p_k(b)\},\quad 
	b\in T(g,k), 
\]
where $\wh X_b$ is a compact Riemann surface
with the complex structure $J_b$ determined by $b\in T(g,k)$ 
and the pairwise disjoint punctures $p_i(b)$ $(i=1,\ldots,k)$ 
are holomorphic sections of the universal family 
$\pi:\wh V(g,k)\to T(g,k)$. (See Nag \cite[pp.\ 322--323]{Nag1988}
or Imayoshi and Taniguchi \cite{ImayoshiTaniguchi1992}.) 
The open subset
$ 
	V(g,k) = \wh V(g,k)\setminus \bigcup_{i=1}^k p_i(T(g,k))
$ 
of the complex manifold $\wh V(g,k)$ is the universal family 
of $k$-punctured compact Riemann surfaces of genus $g$. If 
$2g+k\ge 3$ then the Teichm\"uller family $\pi:V(g,k)\to T(g,k)$ 
is the universal object in the category of topologically 
marked holomorphically varying families of $k$-punctured 
compact Riemann surfaces of genus $g$ 
(see \cite[Theorem 5.4.3]{Nag1988}).
It is shown in \cite{Forstneric2025BLMS} that $V(g,k)$ is a Stein manifold.
This gives analogues of the main results of this paper 
for holomorphic maps $V(g,k)\to Y$ to any Oka manifold $Y$.
Furthermore, holomorphic functions on $V(g,k)$ which are algebraic
on every fibre of the Teichm\"uller submersion $\pi:V(g,k)\to T(g,k)$
are dense in the space of all holomorphic functions
\cite[Theorem 3.4]{Forstneric2025BLMS}, and they provide
an affine embedding of $V(k,n)$ over any relatively compact domain
in $T(g,n)$ \cite[Theorem 3.1]{Forstneric2025BLMS}. 
However, the setting in Theorem \ref{th:Oka} is more general than the 
one in Teichm\"uller theory since the complex structures 
in a given family need not be quasiconformally equivalent.
For example, a punctured Riemann surface can be a member
of a family in which the punctures develop into boundary curves 
and vice versa. Our results concerning approximation of functions,
such as Theorems \ref{th:Runge} and \ref{th:Mergelyan}, 
also hold on infinite dimensional Teichm\"uller spaces, which 
are known to be metrisable and hence paracompact Hausdorff; 
see the surveys by Fletcher and Markovi\'c \cite{FletcherMarkovic2009} 
and Markovi\'c and \v Sari\'c \cite{MarkovicSaric2010}
for this topic.

%
%
The results and methods developed in the paper can be
used in constructions of families of holomorphic
curves with special properties and of related objects, such as 
conformal minimal surfaces. To illustrate the point, we give two such 
applications in Section \ref{sec:directed}. 
The first one, in Theorem \ref{th:directed}, gives families
of $J_b$-holomorphic immersions $X\to\C^n$
directed by an irreducible conical complex subvariety 
$\overline A=A\cup \{0\} \subset \C^n$ such that 
$A=\overline A\setminus \{0\}$ is an Oka manifold. 
By taking $A=\C^n_*$ we obtain families of 
ordinary immersions $X\to\C^n$. For $n=1$ this
gives an extension of the Gunning--Narasimhan theorem
\cite{GunningNarasimhan1967}
to families of $J_b$-holomorphic immersions $X\to \C$;
see Corollary \ref{cor:GunningNarasimhan}. 

Another major application pertains to 
the null cone $\boldA\subset\C^n$ for $n\ge 3$, see \eqref{eq:nullq}.
The real and the imaginary part of holomorphic immersions 
$X\to\C^n$ directed by the null cone are conformally  
immersed minimal surfaces $X\to\R^n$.  
We thus obtain continuous or smooth
families of conformally immersed minimal surfaces 
$X\to\R^n$, $n\ge 3$, for any continuous or smooth family
of complex structures $J_b$ on $X$ (see Corollary \ref{cor:minimal}).
Several other possible applications 
are indicated in Problem \ref{prob:problems}.
A common feature of these examples is that their construction
combines Oka theory with methods from Gromov's convex 
integration theory to ensure the period vanishing conditions of the 
derivative maps on a basis of the first homology group $H_1(X,\Z)$.

%
%
The paper is organised as follows. In Section \ref{sec:Beltrami} we recall 
the connection between Riemannian metrics, conformal structures, 
and the Beltrami equation. Section \ref{sec:CB} contains preparatory 
results on the Cauchy and Beurling transforms on open Riemann surfaces.
In Section \ref{sec:QC} we obtain results  
on deformations of complex structures which are used in the proofs.
Theorem \ref{th:Beltrami1} gives a solution of the Beltrami equation 
on any smoothly bounded relatively compact domain $\Omega$ 
in an open Riemann surface $X$ for Beltrami coefficients 
$\mu:\Omega\to \D=\{\zeta\in \C:|\zeta|<1\}$ 
with small H\"older $\Cscr^{(k,\alpha)}(\Omega)$ norm 
$(k\in \Z_+,\ 0<\alpha<1)$, with analytic dependence on $\mu$.
It is then shown in Theorem \ref{th:Beltrami2} that any sufficiently small  
$\Cscr^{(k,\alpha)}$ perturbation of the complex structure on 
$\Omega$ can be realised by a small $\Cscr^{(k+1,\alpha)}$ 
diffeomorphic perturbation of $\Omega$ in $X$, 
with analytic dependence of the map on the complex structure. 
This extends the Ahlfors--Bers theory \cite{AhlforsBers1960} of 
quasiconformal maps of the plane. 

With these tools in hand, 
Theorems \ref{th:Runge} and \ref{th:Mergelyan} 
are proved in Section \ref{sec:Runge}.

In Section \ref{sec:Oka} we prove our main result, 
Theorem \ref{th:Oka}, and obtain further Runge and 
Mergelyan type approximation results for families of 
manifold-valued maps; see Theorems \ref{th:Okabis}, 
\ref{th:Mergelyan-manifold}, and \ref{th:Mergelyan-admissible}.

In Section \ref{sec:trivialisation} we 
show that for a family of complex structures 
on a smooth open surface, the family of their 
holomorphic cotangent (canonical) bundles admits a 
family of holomorphic trivialisations (see Theorem \ref{th:thetab}),
which can be given by a family of holomorphic immersions to $\C$
(see Corollary \ref{cor:GunningNarasimhan}).

Finally, in Section \ref{sec:directed} we apply 
our results to the construction of families of directed
holomorphic immersions and conformal minimal immersions
to Euclidean spaces.

%
%

We conclude this introduction by mentioning some recent developments 
in which the results of an earlier version of the present paper were 
used in an important way.

%
%
A standard application of the Runge--Behnke--Stein approximation 
theorem on open Riemann surfaces and, 
more generally, of the Oka--Weil theorem on Stein manifolds,  
is the global solvability of the $\dibar$-equation for $(0,1)$-forms.
Using Theorem \ref{th:Runge} and the techniques developed
in Sections 3 and 4, I obtained in \cite{Forstneric2026CAS}
an optimal solution of the $\dibar$-equation on families 
of open Riemann surfaces with the gain one of derivative 
in the space variable and without any loss of regularity in the parameter;
 
Note that solvability of the tangential $\dibar$-complex in all bidegrees 
on $\Cscr^\infty$ smooth Cartan manifolds was shown by Jurchescu \cite[Sect.\ 3]{Jurchescu1994}. See also his paper 
\cite{Jurchescu1988AUB} for the approximation theorems.

%
%
Open Riemann surfaces are Stein manifolds of complex
dimension one. The Oka theory for continuous tame 
families of sufficiently smooth 
integrable Stein structures $\{J_b\}_{b\in B}$ 
on smooth manifolds $X$ of dimension $\ge 4$
was developed by Sigur{\dh}ard{\'o}ttir and the author
in \cite{ForstnericSigurdardottir2025}. 
(Every family of complex structures on a 
surface is tame, but there are nontame families of Stein structures 
in higher dimensions; see \cite[Sect.\ 4]{ForstnericSigurdardottir2025}.)
The construction relies on the approach developed in 
this paper and uses in particular Lemmas 
\ref{lem:Ck} and \ref{lem:main}. 
The use of Theorem \ref{th:Beltrami2} is replaced by 
a parametric version of Hamilton's theorem \cite{Hamilton1977}  
on strongly pseudoconvex domains, whose proof relies
on the work by Greene and Krantz \cite{GreeneKrantz1982} 
on stability of Kohn's solutions of the $\dibar$-equation 
under small integrable deformations of the Stein structure 
on a strongly pseudoconvex domain; 
see \cite[Theorem 3.1]{ForstnericSigurdardottir2025}. 
%

%
%
In \cite{DrinovecKalisnik2026}, Drinovec Drnov{\v s}ek and
Kali{\v s}nik proved that for any smooth open surface $X$ 
endowed with a family of complex structures $\{J_b\}_{b\in B}$ 
depending continuously on the parameter $b$ in a metrisable space $B$, 
there is a continuous family of proper holomorphic maps 
$F_{b}:(X,J_b)\to\mathbb C^{2}$, $b\in B$. 
Their result provides a partial affirmative 
answer to the first part of Problem \ref{prob:problems} (a). 
Several further applications of the results and techniques of
this paper were obtained by Alarc\'on and the author
in \cite{AlarconForstneric2026proper} (May 2026); 
see the discussion following Problem \ref{prob:problems}. 


%
%
%
%
\section{Riemannian metrics, complex structures, and the Beltrami equation}
\label{sec:Beltrami}

In this section, we recall the relevant background on the topics 
mentioned in the title. The details can be found in standard texts
on quasiconformal mappings and Teichm\"uller spaces; see e.g.\  
Ahlfors \cite{Ahlfors2006}, Lehto and Virtanen \cite{LehtoVirtanen1973}, 
Nag \cite{Nag1988}, and Imayoshi and Taniguchi \cite{ImayoshiTaniguchi1992}.

Let $z=x+\imath y$ be the complex coordinate on $\C$. 
Set $\di_x=\di/\di x$, $\di_y=\di/\di y$,
$dz=dx+\imath dy$, $d\bar z=dx-\imath dy$, 
\[
	\di_z= \frac{\di}{\di z}
	=\frac12 \left(\di_x-\imath \di_y\right),
	\qquad
	\di_{\bar z}=\frac{\di}{\di \bar z}
	=\frac12 \left(\di_x+\imath \di_y \right).
\]
For a differentiable function $f$ we shall write
$f_z =\di_z f$ and $f_{\bar z}=\di_{\bar z}f$. 
Note that $f$ is holomorphic if and only if $f_{\bar z}=0$.
The exterior differential on functions 
splits in the sum of its $\C$-linear and $\C$-antilinear parts:
$
	d=\di  + \dibar  = \di_z  dz + \di_{\bar z} d\bar z.
$

A Riemannian metric on a smooth surface $X$ is given 
in any local coordinates $(x,y)$ by 
\begin{equation}\label{eq:metric}
	g= E dx\otimes dx + F(dx\otimes dy+dy\otimes dx) + G dy\otimes dy  
	  = E dx^2+ 2F dxdy + Gdy^2, 
\end{equation}
where $E,F,G$ are real functions satisfying $EG-F^2>0$.
The area form determined by the metric $g$ is $\sqrt{EG-F^2}\, dx\wedge dy$.
The Euclidean metric and the area form on $\R^2\cong\C$ 
with the coordinate $z=x+\imath y$ are given by 
$g_{\rm st}=dx^2+dy^2 = |dz|^2$ and  
$dx\wedge dy = \frac{\imath}{2} dz\wedge d\bar z$.
On every tangent space $T_pX$, a Riemannian metric $g$  
defines a scalar product having the matrix $\bigl({E \ F \atop F \ G}\bigr)$ 
in the basis $\di_x,\ \di_y$. Hence, $g$ determines 
a unique conformal structure on $X$, and two Riemannian metrics 
$g_1,g_2$ determine the same conformal structure if and only if 
$g_2=\lambda g_1$ for a positive function $\lambda$.
A pair of nonzero tangent vectors $\xi,\eta\in T_p X$ 
is said to be a conformal frame if $\xi$ and $\eta$ have the same 
$g$-length and are $g$-orthogonal to each other. 
If $X$ is oriented, there is a unique endomorphism 
$J:  TX \to TX$ on the tangent bundle of $X$
such that for any tangent vector
$0\ne v \in T_p X$, $(v,Jv)$ is a positively oriented $g$-conformal frame.
Note that $J^2= -\Id$; an endomorphism of $TX$ satisfying this
condition is called an {\em almost complex structure} on $X$. 
We have the following local expression for the matrix of $J$ 
(in the standard oriented basis $\di_x,\ \di_y$) in 
terms of the metric $g$ \eqref{eq:metric}:
\begin{equation}\label{eq:matrixofJ}
	[J] = \frac{1}{\sqrt {EG-F^2}} 
	\left(
	\begin{matrix} -F & - G \\ E & F \end{matrix}
	\right)
	= 
	\left(
	\begin{matrix} -b & - c \\ (b^2+1)/c & b \end{matrix}
	\right)
\end{equation}
where $\delta =EG-F^2>0$, $b=F/{\sqrt \delta}$, and 
$c=G/{\sqrt \delta}>0$.
Every almost complex structure $J$ is of this form for
some Riemannian metric $g$, which is unique up to conformal
equivalence. The standard almost complex structure $\Jst$ on $\C$, 
defined by the Euclidean metric $g_{\rm st}$, has the matrix
$\bigl({0 \ \, -1 \atop 1 \ \ \ 0}\bigr)$. In complex notation, 
$\Jst$ amounts to multiplication by $\imath$. 
A differentiable function $f:U\to \C$ on a domain $U\subset X$ is 
said to be $J$-holomorphic (more precisely, $(J,\Jst)$-holomorphic) if
it satisfies the Cauchy--Riemann equation 
$df_p \circ J_p = \Jst \circ df_p$ at all points $p\in U$,
where $J_p$ denotes the restriction of $J$ to $T_p X$.
At a point where $df_p\ne 0$, such $f$ is an orientation preserving 
conformal map from the conformal structure on $X$ determined by 
$J=J_g$ to the standard conformal structure on $\C$.  

Assume that the metric $g$ is given in local coordinates $(x,y)$ on 
an open set $U\subset X$ by \eqref{eq:metric}. Taking $z=x+\imath y$ 
as a complex coordinate on $U$, we can write $g$ in 
the complex form as 
\begin{equation}\label{eq:metric-complex}
	g = \lambda|dz+\mu d\bar z|^2
\end{equation}
for a positive function $\lambda>0$ and the complex function 
\begin{equation}\label{eq:mu}
	\mu = \frac{1-c+\imath b}{1+c+\imath b}: U \to\D
\end{equation}
with values in the unit disc, where the numbers $b$ and $c$ are as in
\eqref{eq:matrixofJ}; see \cite[p.\ 51]{AlarconForstnericLopez2021}.
A diffeomorphism $f:U\to f(U)\subset \C$ is conformal from 
the $g$-structure on $X$ to the standard conformal structure on $\C$ 
if and only if $g=h |df|^2$ for a positive function $h>0$. 
A chart $f$ with this property is said to be {\em isothermal} for $g$. 
Assume that $f$ is orientation preserving, 
which amounts to $|f_z|>|f_{\bar z}|$. Then 
\[
	|df|^2 = |f_z dz + f_{\bar z}d\bar z|^2
	= |f_z|^2 \,\cdotp \Bigl|dz + \frac{f_{\bar z}}{f_z} d\bar z \Bigr|^2,
\]
and comparison with \eqref{eq:metric-complex} 
shows that $f$ is isothermal if and only if it satisfies the Beltrami equation 
\begin{equation}\label{eq:Beltrami}
	f_{\bar z} = \mu f_z
\end{equation}
with the Beltrami coefficient $\mu$ given by \eqref{eq:mu}.
We shall say that $f$ is $\mu$-conformal if \eqref{eq:Beltrami} holds.
Equivalently, $f$ is a biholomorphic map from $(U,J)$ to  $(f(U),\Jst)$
where $J$ is the complex structure on $X$ determined by $g$
(or by $\mu$). 

One can also consider quasiconformal maps $f:X\to Y$ between
a pair of Riemann surfaces. The quantity 
$\mu_f(z)=f_{\bar z}/f_z$, defined in a local holomorphic coordinate 
$z$ on $X$, is independent of the choice of the local holomorphic
coordinates on $Y$, and $\mu_f(z) d\bar z/dz$ is a section 
of the bundle $K_X^{-1}\otimes \overline K_X \to X$
where $K_X=T^*X$ is the canonical bundle of $X$
(see \cite[p.\ 46]{Nag1988}). 

\begin{remark}\label{rem:smoothness}
The formulas \eqref{eq:matrixofJ}--\eqref{eq:mu} show that the 
conformal class of a Riemannian metric $g$, the associated  
complex structure $J$, and the Beltrami coefficient $\mu$ 
are of the same smoothness class. 
\end{remark}

The situation is especially simple if we fix a reference 
complex structure on $X$, so it is an open Riemann surface.
By a theorem of Gunning and Narasimhan \cite{GunningNarasimhan1967},
such a surface admits a holomorphic immersion 
$z=u+\imath v:X\to \C$. Its differential $dz=du+\imath dv$ is a 
nowhere vanishing holomorphic 1-form on $X$ trivialising the 
canonical bundle $T^*X=K_X$, $|dz|^2= du^2+dv^2$ is a 
Riemannian metric on $X$ determining the given complex structure, 
$\frac{\imath}{2} dz\wedge d\bar z = du\wedge dv$ is the associated
area form, and $d\sigma=du\, dv$ is the surface measure on $X$.
The function $z$ provides 
a local holomorphic coordinate on $X$ at every point. 
Given a differentiable function $f:X\to \C$, its partial derivatives 
\begin{equation}\label{eq:derivatives}
	f_z = \di_z f =\di f/dz,\qquad 
	f_{\bar z} = \di_{\bar z} f =\dibar f/d{\bar z}
\end{equation}
are globally defined functions on $X$.
Any Riemannian metric $g$ on $X$ is globally of
the form \eqref{eq:metric} for some real functions $E,F,G$ on $X$. 
(However, these coefficients are not functions of $z$ unless $z$ 
is injective, in which case $X$ is a plane domain.) 
We can write $g$ in the form 
\eqref{eq:metric-complex} where the function $\mu:X\to \D$ is 
given by \eqref{eq:mu}. Conversely, any such function $\mu$ determines 
a Riemannian metric by \eqref{eq:metric-complex}, and hence
a complex structure $J_\mu$ by \eqref{eq:matrixofJ}. 
Note that $\mu=0$ corresponds to the given reference 
complex structure on $X$. This global 
viewpoint will be important in the sequel.

The study of isothermal charts on Riemannian surfaces  
is a classical subject going back 
to Lagrange and Gauss. For a H\"older continuous $\mu$,   
see Korn \cite{Korn1914}, Lichtenstein \cite{Lichtenstein1916},
and Chern \cite{Chern1955}. The existence of 
global quasiconformal homeomorphisms $\C\to\C$ 
follows from the local theorem by use of the uniformization theorem,
with direct proofs given by Ahlfors \cite{Ahlfors1955} 
and Vekua \cite{Vekua1955}. For a measurable function $\mu$ satisfying
$\|\mu\|_\infty \le k <1$
(where $\|\mu\|_\infty$ denotes the essential supremum), 
see Morrey \cite{Morrey1938} and Bojarski \cite{Bojarskij1958}. 
In this case, solutions of \eqref{eq:Beltrami} 
are $k$-quasiconformal homeomorphisms having distributional derivatives 
in $L^p$ for some $p\ge 1$. More precise results in $L^p(\C)$ spaces, 
with smooth dependence of solutions of the Beltrami
equation \eqref{eq:Beltrami} on the Beltrami coefficient $\mu$, 
are due to Ahlfors and Bers \cite{AhlforsBers1960}; 
see also Ahlfors \cite[Chapter V]{Ahlfors2006} and 
Astala et al.\ \cite{AstalaIwaniecMartin2009}.
We shall use the following result;  
see \cite[Theorem 5.3.4]{AstalaIwaniecMartin2009}
for the first part and \cite[Theorem 2.1]{BojarskiAll2005}
for the second part.

%
%
\begin{theorem} \label{th:isothermal}
An almost complex structure $J$ of H{\"o}lder class $\Cscr^{(k,\alpha)}$ 
$(k\in \Z_+,\ 0<\alpha<1)$ on a smooth surface $X$ admits a 
$J$-holomorphic chart of class $\Cscr^{(k+1,\alpha)}$ at any  
point of $X$. 
%
%
Hence, the smooth structure on $X$ determined by $J$
is $\Cscr^{(k+1,\alpha)}$ compatible with the given smooth structure. 
\end{theorem}

%
%
The following result shows that the assumptions in our main results
(Theorems \ref{th:Runge}, \ref{th:Mergelyan}, \ref{th:Oka})
are independent of the choice of the smooth structure on $X$
in the equivalence class of $\Cscr^{(k+1,\alpha)}$-equivalent structures.
It will be used in the proof of Corollary \ref{cor:Hamilton}.

\begin{proposition}\label{prop:pullback}
Assume that $X$ and $Y$ are smooth surfaces and $\phi:Y\to X$ is a
diffeomorphism of (local) class $\Cscr^{(k+1,\alpha)}$ for some
$k\in\Z_+$ and $0<\alpha<1$. 
Let $\{J_b\}_{b\in B}$ be a family of complex structures on $X$
of class $\Cscr^{l,(k,\alpha)}$. 
Then, the family of complex structures $\{J'_b=\phi^*J_b\}_{b\in B}$
on $Y$ is also of class $\Cscr^{l,(k,\alpha)}$.
Furthermore, if $f:B\times X\to \C$ is of class $\Cscr^{l,s}$
for some $s\le k+1+\alpha$ then the function $\tilde f:B\times Y\to \C$
given by $\tilde f(b,y)= f(b,\phi(y))$ is also of class $\Cscr^{l,s}$ 
on $B\times Y$.
\end{proposition}

\begin{proof}
For any point $y\in Y$ we have that 
$
	(J'_b)_{y} 	= (d\phi_y)^{-1} \circ (J_b)_{\phi(y)} \circ d\phi_y.
$
Let $A(y)$ denote the matrix of the differential $d\phi_y$ in a pair of 
smooth local trivialisations of the tangent bundle on $X$ and $Y$, and 
let $[J_{b}(x)]$ denote the matrix of $J_b$ at $x\in X$. 
The above then says that 
\[
	[J'_{b}(y)]= A(y)^{-1} [J_{b}(\phi(y))] A(y)
\]
where the operation is the matrix product.
By the assumption, the matrix $[J_{b}(x)]$ is of class 
$\Cscr^{l,(k,\alpha)}$ on $B\times X$. It is elementary to see 
that inserting $x=\phi(y)$, with $\phi$ of class $\Cscr^{(k+1,\alpha)}$, 
gives a matrix function $[J_{b}(\phi(y))]$ of the same class
$\Cscr^{l,(k,\alpha)}$ on $B\times Y$. Finally, the conjugation
by the matrix function $A(y)$ of class $\Cscr^{(k,\alpha)}$
preserves the class $\Cscr^{l,(k,\alpha)}$.
The last part of the proposition is a simple exercise.
\end{proof}

%
%
%
\section{The Cauchy and Beurling transforms on open 
Riemann surfaces}
\label{sec:CB}

In this section, we consider regularity properties
of the Cauchy and Beurling transforms 
on smoothly bounded relatively compact domains in open 
Riemann surfaces. Theorem \ref{th:PSproperties} 
is an important analytic ingredient for solving 
the Beltrami equation on such domains; see 
Theorems \ref{th:Beltrami1} and \ref{th:Beltrami2}. 
 
Let $X$ be an open Riemann surface. Fix a 
holomorphic immersion $z=u+\imath v:X\to \C$ 
(see \cite{GunningNarasimhan1967}) and let
$d\sigma=du\, dv$ denote the associated area measure on $X$.
Given a differentiable function $f:U\to \C$ on a domain
$U\subset X$, its derivatives $f_z$ and $f_{\bar z}$ 
given by \eqref{eq:derivatives} are well-defined functions on $U$.
The pullback of the Cauchy kernel $C(\zeta,z)=\frac{dz}{z-\zeta}$ on $\C$
by the immersion 
$z:X\to\C$ is a Cauchy-type kernel on $X$ with the correct behaviour 
near the diagonal $D_X=\{(x,x):x\in X\}$ (see \eqref{eq:CauchykernelX}), 
but with additional poles if $z$ is not injective. Since $D_X$ has a basis 
of Stein neighbourhoods in $X\times X$ and 
$X\times X\setminus D_X$ is also Stein, 
one can remove the extra poles by solving a Cousin problem
(see Scheinberg \cite[Lemma 2.1]{Scheinberg1978}). 
This gives a meromorphic $1$-form on $X\times X$ of the form
\begin{equation}\label{eq:Cauchykernel}
	\omega(q,x) = \xi(q,x) dz(x)\quad \text{for}\ q,x\in X, 
\end{equation}
where $dz(x)$ denotes the restriction of $dz$ to $T_x X$, 
$\xi$ is a meromorphic function on $X\times X$ which is 
holomorphic on $X\times X\setminus D_X$, and 
the 1-form $\omega(q,\cdotp)$ has a simple pole at $q\in X$ 
with residue $1$. In a neighbourhood $U\subset X\times X$ 
of $D_X$ the coefficient $\xi$ of $\omega$ is of the form
\begin{equation}\label{eq:CauchykernelX}
	\xi(q,x)= \frac{1}{z(x) - z(q)} + h(q,x), 
\end{equation}
where $h$ is a holomorphic function on $U$. 
Such Cauchy kernels were constructed by Scheinberg \cite{Scheinberg1978} 
and Gauthier \cite{Gauthier1979}, following the work 
by Behnke and Stein \cite[Theorem 3]{BehnkeStein1949}. 
(See also Behnke and Sommer \cite[p.\ 584]{BehnkeSommer1962}
and \cite[Remark 1, p.\ 141]{FornaessForstnericWold2020} for additional references.) 
Given a relatively compact smoothly bounded domain 
$\Omega\Subset X$, the usual argument 
using Stokes formula and the residue calculation 
gives the following Cauchy--Green formula 
for any $f\in\Cscr^1(\overline\Omega)$ and $q\in \Omega$:
\begin{eqnarray*}\label{eq:CF1}
	f(q) &=& \frac{1}{2\pi \imath}\int_{x\in b\Omega} f(x)\,\omega(q,x) 
	- \frac{1}{2\pi \imath} \int_{x\in \Omega}
	\overline\partial f(x)\wedge\omega(q,x) \\
	\label{eq:CF2}
	&=& \frac{1}{2\pi \imath}\int_{x\in b\Omega} f(x)\,\xi(q,x)dz(x) 
	- \frac{1}{\pi} \int_{x\in \Omega} f_{\bar z}(x) \xi(q,x) d\sigma(x).
\end{eqnarray*}
If $f$ is holomorphic in $\Omega$, we obtain the 
Cauchy representation formula 
\[ 
	f(q) = \frac{1}{2\pi \imath}\int_{x\in b\Omega} f(x)\,\omega(q,x),
	\quad q\in \Omega.
\] 
On the other hand, for a function $f\in \Cscr^1_0(X)$ with 
compact support we have 
\begin{equation}\label{eq:CF4}
	f(q) = - \frac{1}{\pi} \int_{x\in X} f_{\bar z}(x) \xi(q,x) d\sigma(x),
	\quad q\in X.
\end{equation}

To the Cauchy kernel $\omega$ we associate 
two transforms, defined for $\phi\in \Cscr_0(X)$ and $q\in X$ by 
\begin{eqnarray}\label{eq:P}
	P(\phi)(q) &=& 
	-\frac{1}{\pi} \int_{X} \phi(x) \xi(q,x) d\sigma(x), \\
	\label{eq:S}
	S(\phi)(q) &=& \di_z P(\phi)(q) 
	= -\frac{1}{\pi} \int_{X} \phi(x) \di_{z(q)} \xi(q,x) d\sigma(x). 
\end{eqnarray}
Here, $\di_{z(q)} \xi(q,x)$ denotes the $\di_z$ derivative \eqref{eq:derivatives}
of the function $\xi(\cdotp,x)$ at the point $q\in X$. 
The integral defining $S$ is understood as the Cauchy principal value
(compare with \eqref{eq:Bcal}), and the existence of $S(\phi)$ 
for any $\phi\in \Cscr_0(X)$ follows from Theorem \ref{th:PSproperties} (c). 

The operator $P$ is called the Cauchy--Green transform
associated to the Cauchy kernel \eqref{eq:Cauchykernel}. 
The integral converges absolutely, and we have that 
\[
	\di_{\bar z} \circ P = \Id = P \circ \di_{\bar z}
	\quad \text{on}\ \Cscr^1_0(X). 
\]
The second identity follows from \eqref{eq:CF4}. 
The first identity holds in a more precise form: 
for every relatively compact domain $\Omega \subset X$ 
with piecewise $\Cscr^1$ boundary, 
\begin{equation}\label{eq:propP}
	\di_{\bar z} P(\phi)=\phi \ \ 
	\text{holds on $\Omega$ for every $\phi\in \Cscr^1(\overline \Omega)$},
\end{equation}
where the integral defining $P$ is applied only over $\Omega$.
The equation \eqref{eq:propP} holds in the distributional sense
for every integrable $\phi$. It is obtained by following
the proof in the case when $\xi(q,x)=\frac{1}{x-q}$ on $X=\C$,
when $P$ equals the standard Cauchy--Green operator on $\C$:
\begin{equation}\label{eq:Ccal}
	\Ccal(\phi)(z)= 
	-\frac{1}{\pi} \int_\C \frac{\phi(\zeta)}{\zeta-z} d\sigma(\zeta),
	\quad z\in\C.
\end{equation}

The operator $S$ \eqref{eq:S} 
is an analogue of the Beurling transform $\Bcal$ 
in the plane (see \cite{Ahlfors2006} or  
\cite[p.\ 94]{AstalaIwaniecMartin2009}): 
\begin{equation}\label{eq:Bcal}
	\Bcal(\phi)(z) = -\frac{1}{\pi} \lim_{\epsilon\to 0} 
	\int_{|z-\zeta|>\epsilon} \frac{\phi(\zeta)}{(z-\zeta)^2} d\sigma(\zeta),
	\quad z\in\C. 
\end{equation}
This is a singular convolution operator of 
Calder\'on--Zygmund type with nonintegrable kernel $-1/\pi z^2$.
It extends to a bounded linear operator $L^p(\C)\to L^p(\C)$ 
for every $1<p<\infty$ (see \cite[Corollary 4.5.1]{AstalaIwaniecMartin2009}).
Its main property is that 
$
	\Bcal\circ \di_{\bar z} =\di_z
$
on $\Cscr^1_0(\C)$, so $\Bcal$ interchanges the operators 
$\di_{\bar z}$ and $\di_z$. Likewise, it follows from  
\eqref{eq:CF4}--\eqref{eq:S} that
\begin{equation}\label{eq:Sinterchanges}
	S(\phi_{\bar z}) = \di_z P(\phi_{\bar z}) = \phi_z
	\ \ \text{for every $\phi\in \Cscr^1_0(\Omega)$}.
\end{equation}

In order to understand the local regularity properties of $P$ and $S$,
we look more closely at their kernel functions $\xi(q,x)$ and $\di_{z(q)} \xi(q,x)$. 
We consider the latter one, which is more involved; the analogous 
analysis applies to the former.  
Let $U\subset X\times X$ be an open neighbourhood of the diagonal $D_X$ 
on which \eqref{eq:CauchykernelX} holds. On $U$ we have
\[
	\di_{z(q)} \xi(q,x) = \di_{z(q)} \frac{1}{z(x) - z(q)} + \di_{z(q)} h(q,x) 
	= \frac{1}{(z(x) - z(q))^2} + \di_{z(q)} h(q,x),
\]
and $\di_{z(q)} h(q,x)$ is holomorphic on $U$. Fix $q_0\in X$
and choose a neighbourhood $V\subset X$ of $q_0$ such that 
$V\times V\subset U$ and the immersion $z:X\to\C$ is injective on $V$.
Pick a smooth function $\chi:X\to[0,1]$ with $\supp \chi\subset V$ 
such that $\chi=1$ on a smaller neighbourhood $V'\Subset V$ of $q_0$. 
For $q\in V$ we have 
\begin{eqnarray*}
	S(\phi)(q) &=& -\frac{1}{\pi} \int_{X} \chi(x) \phi(x) \di_{z(q)} \xi(q,x) d\sigma(x)
	+ \frac{1}{\pi} \int_{X} (\chi(x)-1) \phi(x) \di_{z(q)} \xi(q,x) d\sigma(x) \\
	&=& S_1(\phi)(q) + S_2(\phi)(q),
\end{eqnarray*}
where the operators $S_1$ and $S_2$ are given by 
\begin{eqnarray*}
	S_1(\phi)(q) &=& 
	-\frac{1}{\pi} \int_{X} \frac{\chi(x) \phi(x) }{(z(x) - z(q))^2} d\sigma(x), \\
	S_2(\phi)(q) &=& 
	-\frac{1}{\pi} \int_{X} \chi(x) \phi(x) \di_{z(q)} h(q,x)d\sigma(x) \\ 
	&& \qquad \quad
	+ \frac{1}{\pi} \int_{X} (\chi(x)-1) \phi(x) \di_{z(q)} \xi(q,x) d\sigma(x).
\end{eqnarray*}
In the complex coordinate $z$ on $V$, $S_1(\phi)=\Bcal(\chi\phi)$ is 
the Beurling operator applied to $\chi\phi$, while $S_2$ has smooth kernel. 
The same construction can be carried out with the operator $P$.

%
%
The conclusion is that the operators $P$ and $S$ have the same 
local regularity properties as their classical models $\Ccal$ 
\eqref{eq:Ccal} and $\Bcal$ \eqref{eq:Bcal}, respectively.  

Let $\Omega$ be a relatively compact smoothly bounded 
domain in $X$. One may consider truncated operators $P$ and $S$  
defined by integration over $\Omega$.
While $P$ has the expected regularity 
on H\"older spaces, the regularity of $S$ 
fails at the boundary points of $\Omega$ since the effect of averaging 
in \eqref{eq:Ccal} is lost. To circumvent this problem,
we shall use a bounded linear extension operator, which we now describe. 

Let $\dist$ denote the distance function on a surface $X$ induced by 
a smooth Riemannian metric. We recall the basics concerning H\"older 
spaces; see Gilbarg and Trudinger 
\cite[Sect. 4.1]{GilbargTrudinger1983} for more information.
Let $\Omega$ be a domain in $X$ with piecewise $\Cscr^1$ boundary. 
For $\alpha\in (0,1)$, the H\"older 
$\Cscr^\alpha(\Omega)$ norm of a function $f:\Omega\to \C$ is given by 
\begin{equation}\label{eq:alphanorm}
	\|f\|_\alpha = \sup_{x\in\Omega} |f(x)|
	+ \sup\{ |f(x)-f(y)|/\dist(x,y)^\alpha: x,y\in\Omega,\ x\ne y\},
\end{equation} 
and the associated H\"older space is 
$
	\Cscr^{\alpha}(\Omega) = \{f:\Omega\to \C: \|f\|_\alpha<\infty\}.
$ 
Similarly we define the norm $\|f\|_{(k,\alpha)}$ for $k>0$, and the 
corresponding H\"older space $\Cscr^{(k,\alpha)}(\Omega)$, by adding 
to $\|f\|_\alpha$ in \eqref{eq:alphanorm} the $\Cscr^{\alpha}(\Omega)$ 
norms of partial derivatives of $f$ of the highest order $k$. 
In particular, $\Cscr^{\alpha}(\Omega)=\Cscr^{(0,\alpha)}(\Omega)$.
These spaces are Banach algebras with the pointwise product of functions.
%
%
Compositions of $\Cscr^{(k,\alpha)}$ maps with $k\ge 1$ are again
of the same class (but this fails for $k=0$). 
The inverse of a $\Cscr^{(k,\alpha)}$ diffeomorphism with $k\ge 1$
is of the same class (see \cite[Theorem 2.1]{BojarskiAll2005}).
Every function in $\Cscr^{(k,\alpha)}(\Omega)$ has a unique extension 
to a function in $\Cscr^{(k,\alpha)}(\overline\Omega)$. We shall need the 
following lemma. (The analogous result holds in a smooth 
Riemannian manifold $X$ of arbitrary dimension.)

%
%
\begin{lemma}\label{lem:extension}
Given a smoothly bounded relatively compact domain $\Omega\Subset X$ in 
a smooth open Riemannian surface $X$ and a domain $\Omega'\subset X$ 
containing $\overline \Omega$, there is for every $k\in\Z_+$ and $0\le \alpha<1$ 
a continuous linear extension operator 
$E: \Cscr^{(k,\alpha)}(\Omega) \to \Cscr^{(k,\alpha)}_0(\Omega')$
with range in the space of compactly supported functions in 
$\Cscr^{(k,\alpha)}(\Omega')$.
\end{lemma}


\begin{proof}
For domains in Euclidean spaces and $k\ge 1$, 
this is \cite[Lemma 6.37]{GilbargTrudinger1983}; it is clear
from the construction that one obtains a linear extension operator.
We can reduce to this case by noting that every component $S$ 
of $b\Omega$ has a neighbourhood $U\subset \Omega'$
smoothly diffeomorphic to an annulus in $\R^2$, with $S$ 
corresponding to the unit circle. 
(See Bellettini \cite[Theorem 1.18, p.\ 14]{Bellettini2013}.)
Assume now that $k=0$. Using the above notation, 
let $\tau:U\to S$ denote the smooth radial projection of the annulus 
onto the circle $S$. Set $U_+=U\setminus \Omega$, and let
$\chi:\overline \Omega\cup U\to [0,1]$ be a smooth function 
which equals $1$ on $\overline \Omega$ and the restriction
$\chi|_{U_{+}}$ has compact support. Given $f\in \Cscr(\overline \Omega)$,
we let $E(f):\overline \Omega\cup U\to\C$ be defined by 
$E(f)(x)=f(x)$ for $x\in \overline \Omega$ and $E(f)(x)=\chi(x) f(\tau(x))$
for $x\in U_+$. We perform the same construction on each of the finitely
many boundary components of $\Omega$.
\end{proof}

With the notation of Lemma \ref{lem:extension} we 
define the operators $P_\Omega$ and $S_\Omega$ on 
$\phi\in\Cscr(\overline \Omega)$ and $q\in \overline \Omega$ by 
\begin{eqnarray}\label{eq:POmega}
	P_\Omega(\phi)(q) &=& 
	-\frac{1}{\pi} \int_{x\in \Omega'} E(\phi)(x) \xi(q,x) d\sigma(x), \\
	\label{eq:SOmega}
	S_\Omega(\phi)(q) &=&  
	- \frac{1}{\pi} \int_{x\in \Omega'} E(\phi)(x) \di_{z(q)} \xi(q,x) d\sigma(x).
\end{eqnarray}

%
%
\begin{theorem} \label{th:PSproperties}
Let $X$ be an open Riemann surface with a Cauchy kernel 
\eqref{eq:Cauchykernel}, \eqref{eq:CauchykernelX}.
\begin{enumerate}[\rm (a)]
\item $P_\Omega:\Cscr^{(k,\alpha)}(\overline \Omega) 
\to \Cscr^{(k+1,\alpha)}(\overline \Omega)$ 
is a bounded linear operator for every $0<\alpha<1$ and $k\in\Z_+$, 
and it satisfies $\di_{\bar z} P_\Omega(\phi) = \phi$ 
on $\overline \Omega$ for every $\phi\in \Cscr^{\alpha}(\overline\Omega)$. 
\item 
$S_\Omega:\Cscr^{(k,\alpha)}(\overline \Omega)\to 
\Cscr^{(k,\alpha)}(\overline \Omega)$ is 
a bounded linear operator for every $k\in \Z_+$ and $0<\alpha<1$,
and it satisfies $S_\Omega(\phi)=\di_z P_\Omega(\phi)$ for every 
$\phi\in \Cscr^{(k,\alpha)}(\overline \Omega)$.
\item 
$S_\Omega$ extends to a bounded linear operator $L^p(\Omega)\to L^p(\Omega)$ 
for every $1<p<\infty$.
\end{enumerate}
\end{theorem}

\begin{proof}
We have seen above that the operators $P$ and $S$ have the same 
local regularity properties as their classical models  
$\Ccal$ \eqref{eq:Ccal} and $\Bcal$ \eqref{eq:Bcal}, respectively. 
Part (a) then follows from \cite[Theorem 4.7.2]{AstalaIwaniecMartin2009}
and \eqref{eq:CF4}, part (b) from 
\cite[Theorem 4.7.1]{AstalaIwaniecMartin2009} and \eqref{eq:SOmega},
and part (c) from \cite[Corollary 4.5.1]{AstalaIwaniecMartin2009}. 
(Part (c) is only stated to justify the existence of the integral for continuous
functions.) The analogous properties hold on Sobolev spaces $W^{k,p}$, 
but we shall not need them.
\end{proof}

%
%
%
%
\section{Quasiconformal deformations of the identity map}
\label{sec:QC}

In this section, $z:X\to \C$ denotes a   
holomorphic immersion from an open Riemann surface $X$
(see \cite{GunningNarasimhan1967}). 
We shall call the pair $(X,z)$ a {\em Riemann domain over $\C$}.
Given a $\Cscr^1$ function $f:X\to\C$, 
the derivatives $f_z =\di f/dz$ and $f_{\bar z} = \dibar f/d{\bar z}$
\eqref{eq:derivatives} are well-defined continuous functions on $X$.
We endow $X$ with the smooth structure determined by its Riemann
surface structure and define the H\"older norms on domains 
$\Omega\Subset X$ with respect to a fixed smooth Riemannian metric
on $X$. The following result gives a solution of the Beltrami equation on 
any smoothly bounded relatively compact domain 
for Beltrami coefficients with sufficiently small H\"older norm.
Recall that $\D$ is the unit disc in $\C$. 

%
%
\begin{theorem}\label{th:Beltrami1}
Let $\Omega$ be a smoothly bounded relatively compact domain
in a Riemann domain $(X,z)$. For any $k\in\Z_+$ 
and $0<\alpha<1$ there is a constant $c=c(\Omega,k,\alpha)>0$ 
such that for every 
$\mu\in \Cscr^{(k,\alpha)}(\Omega,\D)$ with $\|\mu\|_{(k,\alpha)}<c$ 
there is function $f=f(\mu)\in \Cscr^{(k+1,\alpha)}(\Omega)$ solving
the Beltrami equation $f_{\bar z} = \mu f_z$, with $f(\mu)$
depending 
analytically on $\mu$ and satisfying $f(0)=z|_\Omega$.
\end{theorem}

Interpreting $\mu\in \Cscr^{(k,\alpha)}(\Omega,\D)$ 
as a complex structure $J_\mu$ on $\Omega$ 
(see \eqref{eq:matrixofJ}--\eqref{eq:mu}),
with $J_0$ coinciding with the initial complex structure, 
the function $f(\mu):\Omega\to\C$ is $J_\mu$-holomorphic, 
and it is an immersion for $\mu$ close to $0$
since $f(\mu)$ is then close to $f(0)=z|_\Omega$ in
$\Cscr^{(k+1,\alpha)}(\Omega)$.
Thus, $(\Omega,J_\mu,f(\mu))$ is a family of Riemann 
domains over $\C$ depending analytically on $\mu$
in a neighbourhood of $\mu=0$.

%
%
\begin{proof}[Proof of Theorem \ref{th:Beltrami1}] 
The idea is inspired by the proof 
of the corresponding result for $\mu\in L^p(\C)$ $(p>2)$,
due to Ahlfors and Bers \cite[Theorem 4]{AhlforsBers1960}.

Recall that the algebra $\mathrm{Lin}(E)$ of all bounded linear operators 
on a Banach space $E$, with functional composition 
as multiplication and the operator norm, is a unital Banach algebra
(see Conway \cite{Conway1990}). In our case, $E$ will be 
the Banach space $\Cscr^{(k,\alpha)}(\Omega)$.

We look for a solution of the Beltrami equation $f_{\bar z} = \mu f_z$ 
on $\Omega$ in the form
\begin{equation}\label{eq:fmu}
	f=f(\mu)=z|_\Omega +P(\phi),\quad \phi\in \Cscr^{(k,\alpha)}(\Omega).
\end{equation}
Here, 
$P=P_\Omega:  \Cscr^{(k,\alpha)}(\Omega)\to \Cscr^{(k+1,\alpha)}(\Omega)$ 
is the Cauchy--Green operator \eqref{eq:POmega}.
Thus, $\phi=0$ corresponds to $f=z|_\Omega$. 
By Theorem \ref{th:PSproperties} (a), $P$ is a continuous 
linear operator. We have that 
\[
	f_{\bar z}= \di_{\bar z}P(\phi) = \phi,
	\qquad
	f_z = 1 + \di_z P(\phi) = 1+ S(\phi),
\]
where $S=S_\Omega \in \mathrm{Lin}(\Cscr^{(k,\alpha)}(\Omega))$ 
is the Beltrami operator \eqref{eq:SOmega}.
The first identity follows from Theorem \ref{th:PSproperties} (a),
and the second one follows from the definition \eqref{eq:SOmega} of 
$S$. Inserting the above expressions in the Beltrami equation 
$f_{\bar z} = \mu f_z$ gives the following equation for $\phi$:
\begin{equation}\label{eq:main}
	\phi = \mu(S(\phi) +1) = \mu S(\phi) + \mu
	\ \Longleftrightarrow\  (I - \mu S) \phi =\mu,
\end{equation}
where $I$ denotes the identity map on $\Cscr^{(k,\alpha)}(\Omega)$. 
By Theorem \ref{th:PSproperties} (b), $S$ is a bounded linear operator 
on $\Cscr^{(k,\alpha)}(\Omega)$. Hence, for 
$\mu$ small enough we have $\|\mu S\|_{(k,\alpha)}<1$, 
so the operator $I- \mu S$ is invertible with 
\begin{equation}\label{eq:Theta}
	\Theta(\mu) = (I - \mu S)^{-1} =
	\sum_{j=0}^\infty (\mu S)^j 
	\in  \mathrm{Lin}(\Cscr^{(k,\alpha)}(\Omega)). 
\end{equation}
The equation \eqref{eq:main} then has a unique solution 
$\phi=\Theta(\mu) \mu =\sum_{j=0}^\infty (\mu S)^j \mu$.
Inserting into \eqref{eq:fmu} gives the following solution to the 
Beltrami equation $f_{\bar z} = \mu f_z$ on $\Omega$:
\begin{equation}\label{eq:fmusolution}
	f(\mu) = z|_\Omega +  P (\Theta(\mu) \mu) 
	= z|_\Omega + P((I - \mu S)^{-1} \mu)
	 \in \Cscr^{(k+1,\alpha)}(\Omega).
\end{equation}
By standard results 
(see Mujica \cite[29.3 Theorem]{Mujica1986}),  
$\Theta(\mu)$ and hence $f(\mu)$ \eqref{eq:fmusolution} 
depend analytically on $\mu$, and the other properties of $f$ 
are obvious from the construction.
\end{proof}

%
%
\begin{remark}\label{rem:parameters}
The fact that the map $\mu\to f(\mu)$ is analytic 
implies that if $\mu(t)$ is a function of class $\Cscr^l$ 
on an open set $t\in U\subset \R^m$ (or $t\in U\subset\C^n$), 
with $\|\mu(t)\|_{(k,\alpha)} < c$ for all $t\in U$, then the map
$U\ni t\mapsto f(\mu(t))\in \Cscr^{(k+1,\alpha)}(\Omega)$ 
is also of class $\Cscr^l(U)$. This holds for any $l\in \{0,1,\ldots,\infty\}$
as well as for real analytic or holomorphic dependence on $t$.
The analogous statement was proved by Ahlfors and Bers
\cite[Theorem 2]{AhlforsBers1960} for solutions of the Beltrami
equation with $\mu\in L^p(\C)$ for $p>2$. 
\end{remark}

Given an open Riemann surface $(X,J)$ 
and a domain $\Omega\subset X$, a family of smooth diffeomorphisms 
$\Phi_b:\Omega \to \Phi_b(\Omega) \subset X$ $(b\in B)$
induces a family of complex structures $J_b=\Phi_b^* J$ on $\Omega$.
The following result shows that the converse holds on any 
smoothly bounded relatively compact 
domain $\Omega\Subset X$ for sufficiently
small variations of the complex structure. 

%
%
\begin{theorem} \label{th:Beltrami2}	
Assume that $(X,z)$ is a Riemann domain over $\C$,  
$\Omega$ is a relatively compact smoothly bounded domain in $X$,
and $a_1,\ldots,a_m \in\Omega$ are distinct points.
For any $k\in\Z_+$ and $0<\alpha<1$ there is a constant 
$c=c(k,\alpha)>0$ such that for every function 
$\mu\in \Cscr^{(k,\alpha)}(\Omega,\D)$ with $\|\mu\|_{(k,\alpha)}<c$ 
there is a $\mu$-conformal diffeomorphism 
$\Phi_\mu:\Omega\to \Phi_\mu(\Omega)\subset X$ in 
$\Cscr^{(k+1,\alpha)}(\Omega,X)$, depending analytically on $\mu$, 
such that $\Phi_0=\Id_\Omega$ and $\Phi_\mu(a_j)=a_j$ for all such 
$\mu$ and $j=1,\ldots,m$.
\end{theorem}

\begin{proof}
If $c>0$ is small enough then for every 
$\mu\in \Cscr^{(k,\alpha)}(\Omega)$ with $\|\mu\|_{(k,\alpha)}<c$
the function $f(\mu)\in \Cscr^{(k,\alpha)}(\Omega)$, furnished by 
Theorem \ref{th:Beltrami1}, is so close to the holomorphic immersion 
$f(0)=z|_\Omega:\Omega\to\C$ in $\Cscr^{(k+1,\alpha)}(\Omega)$ 
that it is an immersion. 
If $f(\mu)$ is sufficiently close to $f(0)=z|_\Omega$, 
we can lift it with respect to the holomorphic 
immersion $z:X\to\C$ to a unique diffeomorphism 
$\Phi_\mu: \Omega\to \Phi_\mu(\Omega)\subset X$ in  
$\Cscr^{(k+1,\alpha)}(\Omega,X)$, close to $\Phi_0=\Id_\Omega$, 
such that 
\begin{equation}\label{eq:lifting}
	\text{$z\circ \Phi_\mu = f(\mu)$ \ \ holds on $\Omega$.}
\end{equation}
To see this, pick $r>0$ such that 
for any $q\in \overline \Omega$ the immersion $z:X\to\C$ is injective on
the disc $U_r(q)\subset X$ of radius $r$ around $q$ 
in the metric $|dz|^2$. 
If $f(\mu)(q)\in \C$ is close enough to $z(q)\in\C$
(which holds if $c>0$ is small enough),
there is a unique point $p\in U_r(q)$ such that 
$z(p)=f(\mu)(q)$, and we set $\Phi_\mu(q)=p$. 
Thus, $\Phi_\mu(q)$ is the unique closest point to $q$ among 
the points in the closed discrete set $z^{-1}(f(\mu)(q)) \subset X$, 
so $\Phi_\mu$ is well-defined on $\overline\Omega$. This implies 
$
	z(\Phi_\mu(q))=z(p)=f(\mu)(q), 
$
so \eqref{eq:lifting} holds. 
Since $\Phi_\mu$ is locally obtained by postcomposing the 
immersion $f(\mu):\Omega\to\C$ with a local inverse of 
the $J$-holomorphic immersion $z:X\to\C$, $\Phi_\mu$ is an immersion, 
its Beltrami coefficient is the same as that of $f(\mu)$,
which is $\mu$, and the regularity properties remain unchanged. 
It is easily seen that $\Phi_\mu$ is injective if $f(\mu)$ 
is close enough to $z$, which holds if $c>0$ is small enough. 

This shows that $\Phi_\mu:\Omega\to \Phi_\mu(\Omega)$
is a family of $\mu$-conformal diffeomorphisms 
in $\Cscr^{(k+1,\alpha)}(\Omega,X)$ depending 
analytically on $\mu$. The interpolation conditions 
$\Phi_\mu(a_j)=a_j$ are achieved as follows.
For every $j=1,\ldots,m$ we choose a holomorphic vector field 
$v_j$ on $X$ which is nonzero at the point $a_j$ and
it vanishes at the points $a_i$ for $i\in \{1,\ldots,m\}\setminus \{j\}$.
Let $t\to \psi_{j,t}$ denote the local flow of $v_j$ for complex time $t$. 
If $c>0$ is small enough, there is an open relatively compact domain 
$\Omega'\Subset X$ such that $\Phi_\mu(\Omega)\subset \Omega'$ 
holds for all $\mu\in \Cscr^{(k,\alpha)}(\Omega)$ with 
$\|\mu\|_{(k,\alpha)}<c$. Choose a bigger domain $\Omega''\Subset X$
such that $\overline{\Omega'}\subset\Omega''$.
Since $\overline {\Omega''}$ is compact,
there is a $t_0>0$ such that the holomorphic map 
$
	\Psi_t := \psi_{1,t_1}\circ\cdots\circ \psi_{m,t_m}: \Omega''\to X
$
is well-defined for all $t=(t_1,\ldots,t_m)\in \C^m$ in the polydisc
$\Delta^m_{t_0}=\{|t_j|<t_0,\ j=1,\ldots,m\}$. 
For every $t\in \Delta^m_{t_0}$ the map $\Psi_t$, 
being a composition of flows of holomorphic
vector fields, is biholomorphic onto its image 
$\Psi_t(\Omega'')\subset X$. The choice of the vector fields
$v_j$ ensures, by the inverse function theorem, 
that for every $m$-tuple of points 
$a'=\{a'_1,\ldots,a'_m\}\subset \Omega$ such that $a'_j$ is close enough 
to $a_j$ for $j=1,\ldots,m$ there is a unique $t=t(a')\in \Delta^m_{t_0}$
close to the origin such that $\Psi_t(a_j)=a'_j$ for $j=1,\ldots,m$,
and the map $a'\mapsto t(a')$ is holomorphic. 
Let $a'(\mu)=(\Phi_{\mu}(a_1),\ldots, \Phi_{\mu}(a_m))$.
Then, for all $\mu\in \Cscr^{(k,\alpha)}(\Omega)$ 
with $\|\mu\|_{(k,\alpha)}<c$
for a small enough $c>0$, the injective holomorphic map 
$\Psi^{-1}_{t(a'(\mu))}:\Omega'\to X$ 
sends the point $\Phi_\mu(a_j)$ back to $a_j$ for $j=1,\ldots,m$.
Replacing $\Phi_\mu$ by $\Psi^{-1}_{t(a'(\mu))} \circ \Phi_\mu$ 
completes the proof.
\end{proof} 

%
%
\begin{remark}\label{rem:homotopic}
Note that every diffeomorphism 
$\Phi_\mu:\Omega\to \Phi_\mu(\Omega)\subset X$ 
in Theorem \ref{th:Beltrami2} is homotopic to the identity 
map on $\Omega$ by the homotopy 
$[0,1] \ni t\mapsto \Phi_{t\mu}$. In particular, if
the domain $\Omega$ is Runge in $X$ then so is 
$\Phi_\mu(\Omega)$. In fact, a domain $\Omega$ in a 
Riemann surface $X$ is Runge if and only
if the inclusion-induced homomorphism $H_1(\Omega,\Z)\to H_1(X,\Z)$
of the first homology groups is injective, and this
condition is clearly invariant under homotopies.
\end{remark}

%
%
\begin{corollary}\label{cor:Hamilton}
Let $B$, $X$, and $\mathscr{J}=\{J_b\}_{b\in B}$ be as in Theorem \ref{th:Runge}, with $\mathscr{J}$ of class $\Cscr^{l,(k,\alpha)}$
$(k\in \Z_+,\ l\le k+1,\ 0<\alpha<1$).
Given $b_0\in B$ and a smoothly bounded 
domain $\Omega\Subset X$, there are a neighbourhood 
$B_0\subset B$ of $b_0$ and a map 
$\Phi:B_0\times \Omega\to B_0\times X$ of the form 
\begin{equation}\label{eq:Hamilton}
	\Phi(b,x)=(b,\Phi_b(x)),\quad b\in B_0,\ x\in \Omega 
\end{equation}
such that $\Phi_{b_0}$
is the identity on $\Omega$ and $\Phi_b:\Omega\to\Phi_b(\Omega)$
is a $(J_b,J_{b_0})$-biholomorphic map in 
$\Cscr^{(k+1,\alpha)}(\Omega,X)$ which is of class $\Cscr^{l}$
with respect to $b\in B_0$.
If $A$ is a finite subset of $\Omega$, we can choose $\Phi$ such that
$\Phi_b(a)=a$ holds for all $a\in A$ and $b\in B_0$.
\end{corollary}

\begin{proof}
Choose a $J_{b_0}$-holomorphic immersion $z:X\to\C$. 
Theorem \ref{th:isothermal} and Proposition \ref{prop:pullback}
imply that the family $\Jscr$ is also of class $\Cscr^{l,(k,\alpha)}$ 
in the smooth structure on $X$ determined by $J_{b_0}$.
Hence there is a family of Beltrami multipliers 
$\mu_b:X\to \D$ $(b\in B)$ of class $\Cscr^{l,(k,\alpha)}$, 
with $\mu_{b_0}=0$, such that $\mu_b$ represents $J_b$ 
(see \eqref{eq:matrixofJ} and \eqref{eq:mu}). 
Pick $c>0$ such that Theorem \ref{th:Beltrami2} applies
to all $\mu\in \Cscr^{(k,\alpha)}(\Omega)$ with $\|\mu\|_{(k,\alpha)}<c$.
By continuity of the map $B\ni b\mapsto \mu_b$ 
there is a neighbourhood $B_0\subset B$ 
of $b_0$ such that $\|\mu_b|_\Omega\|_{(k,\alpha)} <c$ for all $b\in B_0$.
Let $\Phi_b:\Omega\to \Phi_b(\Omega)\subset X$ for $b\in B_0$ 
be a family of $\mu_b$-conformal diffeomorphisms  
furnished by Theorem \ref{th:Beltrami2}. 
The map $\Phi$ in \eqref{eq:Hamilton} then satisfies the corollary.
\end{proof}

%
%

%
%
%
%
\section{Runge and Mergelyan theorems on families of open Riemann surfaces}
\label{sec:Runge}

In this section we prove Theorems \ref{th:Runge} and \ref{th:Mergelyan}.
We begin with the former.

\begin{proof}[Proof of Theorem \ref{th:Runge}]
We first consider the basic case $l=0$ and arbitrary $k\in \Z_+$
and $0<\alpha<1$.
By slightly increasing $K$ and $L$ and adding to $K$ small pairwise 
disjoint discs around the finitely many points in $A\cap (L\setminus K)$,
we may assume that $A\cap L$ is contained in the interior of $K$. 
Given a continuous function $\epsilon:B\to (0,+\infty)$, 
we shall prove that for any compact Runge set $L\subset X$ 
containing $K$ in its interior there exist an open set 
$\Omega\subset X$ containing $L$ and a continuous function 
$F\in\Cscr(B\times \Omega)$ satisfying 
the following conditions for every $b\in B$.
\begin{enumerate}[\rm (a)]
\item The function $F_b=F(b,\cdotp):\Omega \to\C$ is $J_b$-holomorphic. 
\item $\sup_{x\in K}|F_b(x) - f_b(x)|<\epsilon(b)$.
\item $F_b-f_b$ vanishes in the points of the finite set 
$A'=A\cap L =\{a_1,\ldots,a_m\}$.
\end{enumerate}
Approximation in the fine $\Cscr^{(k+1,\alpha)}$-topology will follow 
from condition (b) in view of the Cauchy estimates.
A function $B\times X\to\C$ satisfying Theorem \ref{th:Runge} 
for $l=0$ is then obtained by induction with respect 
to an exhaustion of $X$ by an increasing family of compact Runge sets.


Given a point $b_0\in B$, it suffices to find an open 
neighbourhood $B_0\subset B$ of $b_0$ and a function 
$F:B_0\times \Omega\to \C$ satisfying conditions (a)--(c) 
for all $b\in B_0$. Since $B$ is Hausdorff and paracompact,
this will give a locally finite cover of $B$ 
by open sets $B_j$ and functions $F_j:B_j\times \Omega\to \C$ 
satisfying conditions (a)--(c) for $b\in B_j$. 
Choose a partition of unity $1=\sum_j\chi_j$
with $\supp\,\chi_j\subset B_j$ for every $j$. 
The function $F:B\times \Omega\to\C$, defined by 
\begin{equation}\label{eq:partition}
	F(b,x)=\sum_j \chi_j(b) F_j(b,x)\quad 
	\text{for $b\in B$ and $x\in \Omega$}, 
\end{equation} 
then clearly satisfies conditions (a)--(c). 

With these reductions in mind, we consider the 
problem near a parameter value $b_0\in B$.
We endow $X$ with the Riemann surface structure determined by $J_{b_0}$. 
By Theorem \ref{th:isothermal}, the smooth structure on $X$,
induced by $J_{b_0}$, is $\Cscr^{(k+1,\alpha)}$-compatible with 
the given smooth structure on $X$. 
Choose a relatively compact, smoothly bounded 
domain $\Omega\Subset X$ with $L\subset \Omega$. 
By Corollary \ref{cor:Hamilton} there are a neighbourhood 
$B'_0\subset B$ of $b_0$ and a continuous family 
of $(J_b,J_{b_0})$-biholomorphic maps 
$\Phi_b:\Omega\to \Phi_b(\Omega)\subset X$ $(b\in B'_0)$ 
in $\Cscr^{(k+1,\alpha)}(\Omega, X)$ 
such that $\Phi_{b_0}=\Id_\Omega$ and
$\Phi_b(a)=a$ holds for all $a\in A$ and $b\in B'_0$.
Choose a compact Runge set $K'\subset X$ containing $K$ in its 
interior such that $f_{b_0}$ is holomorphic 
on a neighbourhood of $K'$. Pick a 
neighbourhood $B_0\subset B'_0$ of $b_0$ such that 
$\Phi_b(K)\subset \mathring K'$ holds for all $b\in B_0$.
By Runge theorem in open Riemann surfaces \cite{BehnkeStein1949},
we can approximate $f_{b_0}$ uniformly on $K'$
by a $J_{b_0}$-holomorphic function $F_{b_0}:X \to\C$.
The function $F_b=F_{b_0}\circ \Phi_b:\Omega\to \C$
$(b\in B_0)$ is then $J_b$-holomorphic,
it depends continuously on $b\in B_0$, and $F_b$ is as close
as desired to $f_b$ uniformly on $K$ if $b$ is close enough 
to $b_0$ and $F_{b_0}$ is close enough to $f_{b_0}$ on $K'$. 
Indeed, for $x\in K$ we have
\begin{eqnarray*}
	|F_b(x)-f_b(x)| &\le& 
	|F_{b_0}\circ \Phi_b(x)-f_{b_0}\circ \Phi_b(x)| \\
	&& \ \ + \  |f_{b_0}\circ \Phi_b(x)-f_{b_0}(x)|
	+ |f_{b_0}(x)-f_b(x)|,
\end{eqnarray*}
and each term on the right hand side is 
as small as desired if $b$ is close enough to $b_0$
and $F_{b_0}$ is close enough to $f_{b_0}$ on $K'$. 
Hence, shrinking the neighbourhood $B_0$ around $b_0$ if necessary, 
the family $\{F_b\}_{b\in B_0}$ satisfies conditions (a) and (b).
Furthermore, as the family of $J_{b_0}$-holomorphic 
functions $f_b\circ \Phi_b^{-1}$ $(b\in B_0)$ is uniformly
close to $F_{b_0}$ on the family 
of compact sets $\Phi_b^{-1}(K')$, approximation of $f_b$ by $F_b$
on $K$ in the $\Cscr^{(k+1,\alpha)}$-topology follows from uniform
approximation on a bigger compact Runge set, 
containing $K$ in its interior, in view of the Cauchy estimates. 

This proves Theorem \ref{th:Runge} in the case $l=0$, except
for the interpolation condition (c) which will be dealt with later.
The same proof applies to variable families of compact Runge
sets in a family of open Riemann surfaces as in the following
definition. For later purposes (see in particular in Lemma \ref{lem:Ck}), 
we introduce this notion in the bigger generality of 
families of complex manifolds.

%
%
\begin{definition}\label{def:Runge}
Assume that $B$ is a topological space, $X$ is a smooth manifold
of real dimension $2n\ge 2$, and $\Jscr=\{J_b\}_{b\in B}$ 
is a family of integrable complex structures on $X$.
A closed subset $K$ of $B\times X$ is proper over $B$ 
if the following two conditions hold. 
\begin{enumerate}
\item For every $b\in B$ the fibre $K_b=\{x\in X:(b,x)\in K\}$ is compact.
($K_b$ may be empty.) 
\item For every $b_0\in B$ and open set $U\subset X$ containing
$K_{b_0}$ there is a neighbourhood $B_0\subset B$ of $b_0$ such that
$K_b\subset U$ for all $b\in B_0$.
\end{enumerate}
The set $K$ is Runge in $B\times X$ if it is proper over $B$ 
and the fibre $K_b$ is holomorphically convex in $X$ with respect to
the complex structure $J_b$ for every $b\in B$.
\end{definition}

Condition (2) means that the compact fibres $K_b\subset X$ of $K$ 
are upper semicontinuous with respect to $b\in B$.  
It is easily seen that if $B$ is Hausdorff and locally compact 
then a closed subset $K\subset B\times X$ is proper over $B$
if and only if the restriction of the projection $\pi:B\times X\to B$ to $K$
is a proper map $\pi|_K:K\to B$ (that is,
the inverse image of any compact set in $B$ is compact). 

Our proof of Theorem \ref{th:Runge} in the case $l=0$
gives the following more general result concerning approximation 
on Runge sets in families of open Riemann surfaces.
For the interpolation statement on $Q\times X$, see Remark \ref{rem:Q}. 

%
%
\begin{corollary}\label{cor:Runge}
Assume that $B$ is a paracompact Hausdorff space, $X$
is a smooth open orientable surface, $\Jscr=\{J_b\}_{b\in B}$ 
is a family of complex structures on $X$ of class $\Cscr^{0,(k,\alpha)}$
for some $k\ge 0$ and $0<\alpha<1$, $K\subset B\times X$
is a closed Runge subset (see Definition \ref{def:Runge}),
and $Q$ is a closed subset of $B$.
Given an open subset $U\subset B\times X$ containing $K$
and a continuous $\Jscr$-holomorphic function $f:U \cup (Q\times X)\to\C$, 
we can approximate $f$ in the fine $\Cscr^{0,(k+1,\alpha)}$ topology on $K$
by $\Jscr$-holomorphic functions $F:B\times X\to \C$ of class 
$\Cscr^{0,(k+1,\alpha)}$ satisfying $F=f$ on $Q\times X$.
\end{corollary}

Next, we consider approximation in the $\Cscr^{l,(k+1,\alpha)}$-topology 
for any pair of integers $0\le l\le k+1$.
As before, locally in the parameter we can use Corollary 
\ref{cor:Hamilton} to reduce the approximation problem 
for a variable family of complex structures to the case
of a moving family of compact Runge sets in a fixed complex
structure. With future applications in mind, we consider a 
more general situation for a family of compact holomorphically convex 
sets in a Stein manifold $X$ of arbitrary dimension. 

%
%
\begin{lemma} \label{lem:Ck}
Assume that $B$ is a paracompact Hausdorff space if $l=0$ 
and a manifold of class $\Cscr^l$ if $l>0$, $X$ is a Stein manifold, 
$K$ is a Runge subset of $B\times X$ (see Definition \ref{def:Runge}), 
$U\subset B\times X$ is an open set containing $K$, 
and $f:U\to\C$ is a function of class $\Cscr^{l,0}(U)$ 
such that for every $b\in B$ the function 
$f_b=f(b,\cdotp):U_b=\{x\in X:(b,x)\in U\} \to \C$ is holomorphic. 
Then, $f\in \Cscr^{l,\infty}(U)$ and for any $s\in \Z_+$, 
$f$ can be approximated in the fine $\Cscr^{l,s}$ 
topology on $K$ by $\Cscr^{l,\infty}$ functions 
$F:B\times X\to \C$ such that $F_b=F(b,\cdotp)\in \Oscr(X)$ 
for every $b\in B$.
If $B$ is a topologically closed $\Cscr^l$ submanifold 
of $\R^n\subset\C^n$ (possibly with boundary), 
or a closed subset of $\R^n$ when $l=0$,
then $f$ can be approximated in the fine $\Cscr^{l,s}$ topology 
on $K$ by holomorphic functions $F:\C^n\times X\to\C$.
\end{lemma}

%
%
\begin{remark}\label{rem:HC}
If $B$ is a subset of $\R^n$ then a compact set 
$K\subset B\times X$ with $\Oscr(X)$-convex fibres $K_b$ is 
$\Oscr(\C^n\times X)$-convex
(see Remark 1.3 and Proposition 1.4 in \cite{ForstnericWold2010PAMS}).
\end{remark}

%
%
\begin{proof}[Proof of Lemma \ref{lem:Ck}]
The assumptions on the function $f$ in the lemma clearly imply 
that it is of class $\Cscr^{l,\infty}(U)$, so we may talk
of $\Cscr^{l,s}$ approximation for any $s\in \Z_+$.
Fix a point $b_0\in B$. Since $K$ is Runge in $B\times X$,
there are an open neighbourhood $B_0\subset B$ of $b_0$ and 
neighbourhoods $K'\subset U' \subset X$ of $K_{b_0}$, where 
$K'$ is a compact $\Oscr(X)$-convex set and $U'$ is an open set, such that 
\begin{equation}\label{eq:inclusions0}
	K\cap (B_0\times X)\subset B_0\times \mathring K' 
	\subset B_0\times U' \subset U\cap (B_0\times X).
\end{equation} 
Choose a smoothly bounded strongly pseudoconvex domain 
$D$ in $X$ with $K'\subset D\Subset U'$.
On $D$, there is a Henkin--Ramirez type kernel
$\omega(x,\zeta)$, which is holomorphic in $x\in D$
for every $\zeta\in bD$, such that every 
$h\in \Oscr(\overline D)$ can be represented on $D$ by 
the integral $h(x)=\int_{\zeta\in bD} h(\zeta)\omega(x,\zeta)$, $x\in D$.
(See Henkin and Leiterer \cite{HenkinLeiterer1984}  
or Lieb and Michel \cite{LiebMichel2002}.  
When $X$ is an open Riemann surface, we can use
a Cauchy kernel \eqref{eq:Cauchykernel} on $X$.)
Applying this to the function $f$ in the lemma, we thus have 
\begin{equation}\label{eq:integral}
	f(b,x)=\int_{\zeta\in bD} f(b,\zeta)\omega(x,\zeta)
	\quad \text{for all $x\in D$ and $b\in B_0$}.
\end{equation}
Fix $\epsilon>0$. Approximating the integral for $b=b_0$ by Riemann sums 
and shrinking $B_0$ around $b_0$ if necessary gives
a finite set of points $\zeta_i\in bD$ and functions 
$g_i\in \Oscr(D)$, which come from the kernel $\omega(x,\zeta_i)$, such that
\begin{equation}\label{eq:estimate0}
	\Big| f(b,x) - \sum_{i} f(b,\zeta_i)g_i(x)\Big| <\epsilon
	\quad \text{for all $b\in B_0$ and $x\in K'$}. 
\end{equation}
By the Oka--Weil theorem, we can approximate each 
$g_i$ uniformly on $K'$ by holomorphic functions $g_i:X\to\C$. 
This gives uniform approximation of $f$ on $B_0 \times K'$
by continuous functions $F:B_0\times X\to \C$ 
which are holomorphic on the fibres $\{b\}\times X$. 
It follows from \eqref{eq:inclusions0} and the Cauchy 
estimates that $F-f$ can be arbitrarily small in 
$\Cscr^{0,s}(K\cap (B_0\times X))$ for a given $s\in\Z_+$. 

In the case $l=0$ and with $B$ a paracompact Hausdorff space,
the proof is completed just as above.
Assume now that $l>0$, so $B$ is a manifold of class $\Cscr^l$.
As before, we cover $K$ by open sets of the form 
$B_j \times U_j \subset B\times X$ such that the cover $\{B_j\}_j$
of $B$ is locally finite, the set $\overline B_j$ is 
compact with $\Cscr^l$ boundary for every $j$, and we have that 
\[ 
	K_j:= K\cap (\overline B_j \times X) \subset 
	\overline B_j \times U_j \subset U.
\] 
Choose a $\Cscr^l$ partition of unity $\{\chi_j\}_j$ 
on $B$ subordinate to the cover $\{B_j\}_j$. 
Fix a $j$ and choose the sets $K'\subset D$ as in the special case 
described above with $j=0$. We represent $f$ by an integral 
of the form \eqref{eq:integral} for $b\in \overline B_j$.
Given $\epsilon>0$, there are finitely many points
$\zeta_i\in bD$ and functions $g_i\in \Oscr(D)$ satisfying  
\eqref{eq:estimate0} for all $b\in \overline B_j$ and $x\in K'$. 
Furthermore, for any linear differential operator $L$ of order $\le l$
in the variable $b\in B_j$ we have that  
$L f(b,x)=\int_{\zeta\in bD} L f(b,\zeta)\omega(x,\zeta)$.
By adding more points $\zeta_i\in bD$
to the Riemann sum if necessary, we approximate 
$L f(b,x)$ uniformly on $\overline B_j\times K'$ by functions 
$\sum_{i} L f(b,\zeta_i)g_i(x)= L \sum_{i} f(b,\zeta_i)g_i(x)$. 
Applying the Cauchy estimates in the $x$-variable, 
we see that $f$ can be approximated in $\Cscr^{l,s}(K_j)$
by functions $F_j: \overline B_j \times X \to\C$ of the form
\[
	F_j(b,x) = \sum_i h_{j,i}(b) g_{j,i}(x), 
	\quad b\in \overline B_j,\ x\in X, 
\]
where $h_{j,i}\in \Cscr^l(\overline B_j)$ and $g_{j,i}\in\Oscr(X)$. 
We define the function $F:B\times X\to \C$ by 
\begin{equation}\label{eq:F}
	F(b,x)=\sum_j \chi_j(b) F_j(b,x) 
	= \sum_{j,i} \chi_j(b) h_{j,i}(b) g_{j,i}(x)\quad
	\text{for $b\in B$ and $x\in X$.}
\end{equation}
Clearly, $F$ approximates $f$ to a given precision in the fine 
$\Cscr^{l,s}$ topology on $K$ provided that $F_j|_{K_j}$ is sufficiently 
close to $f|_{K_j}$ in $\Cscr^{l,s}(K_j)$ for every $j$. 
(When differentiating $F(b,x)$ on the variable $b$, the functions
$\chi_j$ get differentiated as well, so it is important to keep them
fixed when approximating $f$ by $F_j$ in the $\Cscr^{l,s}(K_j)$ topology.)

Finally, if $B$ is a closed $\Cscr^l$ submanifold of $\R^n\subset \C^n$, possibly with boundary, 
we can apply \cite[Theorem 1]{RangeSiu1974} by Range and Siu 
to approximate each function $\chi_j h_{j,i}\in \Cscr^l(B)$ in 
\eqref{eq:F} (which has compact support contained in $B_j$) 
in the fine $\Cscr^l(B)$ topology by an entire function 
$\tilde h_{j,i} \in\Oscr(\C^n)$. 
(Another argument is to extend $\chi_j h_{j,i}$ from the submanifold 
$B\subset\R^n$ to a $\Cscr^l$ function on $\R^n$ and then approximate 
it in the fine $\Cscr^l(\R^n)$ topology by entire functions using 
Carleman's theorem \cite{Carleman1927}.) 
The function $\wt F=\sum_{j,i}\tilde h_{j,i} g_{j,i}$ is then holomorphic on 
$\C^n\times X$ and it approximates $f$ in the fine $\Cscr^{l,s}$
topology on $K$. For $l=0$, the same holds if $B$ is any  
closed subset of $\R^n$, which is seen by combining 
Tietze's extension theorem with Carleman's approximation theorem.
\end{proof}

%
%
We now prove Theorem \ref{th:Runge} for arbitrary pair of 
integers $k\ge 0$ and $0 \le l \le k+1$. 
Let $\Jscr=\{J_b\}_{b\in B}$ and 
$K\subset  X$ be as in the theorem, and pick a
compact Runge set $L\subset X$ containing $K$ in its interior.
By the argument in the beginning of the proof, we may assume 
that $K$ contains the finite set $A\cap L$ in its interior. 
Choose a smoothly bounded Runge domain $\Omega\Subset X$
such that $L\subset \Omega$. Fix a point $b_0\in B$. 
By Corollary \ref{cor:Hamilton} there are a 
neighbourhood $B_0\subset B$ of $b_0$ and a map
\begin{equation}\label{eq:Phi}
	\Phi:B_0\times \Omega \to 
	\Phi(B_0\times \Omega)\subset B_0\times X,
	\quad 
	\Phi(b,x)=(b,\Phi_b(x))
\end{equation}	
in $\Cscr^{l,(k+1,\alpha)}(B_0\times \Omega,X)$ (hence, of class $\Cscr^l$ 
jointly in both variables $(b,x)$) such that for every $b\in B_0$
the map $\Phi_b:\Omega\to \Phi_b(\Omega)\subset X$ 
is a biholomorphism from $(\Omega,J_b)$ 
onto $(\Phi_b(\Omega),J_{b_0})$ satisfying 
\begin{equation}\label{eq:phia}
	\Phi_b(a)=a\ \ \text{for all $a\in A\cap L$}, 
	\ \text{and}\ \Phi_{b_0}=\Id_\Omega. 
\end{equation}	
Clearly, $\Phi$ has a continuous inverse $\Phi^{-1}(b,z)=(b,\psi(b,z))$, 
and if $l>0$ then $\Phi^{-1}$ and hence $\psi$ are of class $\Cscr^{l}$ 
by the inverse function theorem. The observations in the following lemma
are simple consequences of the chain rule and we leave the proof 
to the reader.

%
%
\begin{lemma}\label{lem:composition} 
\begin{enumerate}[\rm (a)] 
\item If $\Phi$ \eqref{eq:Phi} is of class $\Cscr^{l,k+1}$ and 
$l\le k+1$, then $\Phi^{-1}$ is of class $\Cscr^{l,k+1-l}$.
\item If $z=f(b,x)$ is of class $\Cscr^{l,k}$ and $g(b,z)$ is of class 
$\Cscr^{l,l+k}$, then $g(b,f(b,x))$ is of class $\Cscr^{l,k}$.
\end{enumerate}
\end{lemma}

%
%
\begin{lemma}\label{lem:Xholomorphic}
Assume that $0\le l\le k+1$, the function $f:B_0\times\Omega\to\C$
is of class $\Cscr^{l,0}$, $f(b,\cdotp):\Omega\to\C$ is $J_b$-holomorphic
for every $b\in B_0$, and $\Phi$ is as above \eqref{eq:Phi}.
Then, the function $F=f\circ\Phi^{-1}:\Phi(B_0\times \Omega)\to\C$ 
is of class $\Cscr^{l,\infty}$ in the smooth structure on 
$X$ determined by $J_{b_0}$, 
$F(b,\cdotp):\Phi_b(\Omega)\to\C$ is $J_{b_0}$-holomorphic 
for every $b\in B_0$, and $f$ is of class $\Cscr^{l,(k+1,\alpha)}$. 
The analogous result holds for maps to any complex manifold
in place of $\C$.
\end{lemma}

\begin{proof}
Clearly, $F$ is continuous. Since the function $F(b,\cdotp)=f(b,\psi(b,\cdotp))$ 
is a composition of the $(J_{b_0},J_b)$-holomorphic
map $\psi(b,\cdotp)$ and the $J_b$-holomorphic function $f(b,\cdotp)$,
$F(b,\cdotp)$ is $J_{b_0}$-holomorphic for every $b\in B_0$. 
It follows that $F$ is of class $\Cscr^{0,\infty}$ in the complex structure 
$J_{b_0}$ on $X$, and hence of class $\Cscr^{0,(k+1,\alpha)}$ in the 
original smooth structure on $X$ (see Theorem \ref{th:isothermal}).
Since $f=F\circ \Phi$ and $\Phi$ is of class $\Cscr^{l,(k+1,\alpha)}$, 
we infer that $f$ is of class $\Cscr^{0,(k+1,\alpha)}$.
This proves the lemma for $l=0$. Suppose now that $l>0$.
Then, $\psi$ is of class $\Cscr^{l}$. We shall prove that $F$ is of 
class $\Cscr^l$ in the variable $b\in B_0$, 
and hence of class $\Cscr^{l,\infty}$ 
(since it is $J_{b_0}$-holomorphic in the space variable). 
We make the calculation in a local coordinate $b$ of class 
$\Cscr^l$ on $B_0$, and we assume for simplicity of exposition
that $B_0=[0,1]\subset \R$. On $X$, we use a 
$J_{b_0}$-holomorphic coordinate $z$. Differentiating the 
equation $F(b,z)=f(b,\psi(b,z))$ on $b$ and denoting 
the partial derivatives by the lower case indices gives
\begin{equation}\label{eq:Fb}
	F_b(b,z)=f_b(b,\psi(b,z)) + f_x(b,\psi(b,z)) \psi_b(b,z).
\end{equation}
Here, $f_x$ denotes the total derivative of $f$ with respect
to a smooth local coordinate $x=(u,v)$ on $X$.
This shows that $F_b(b,z)$ exists and is continuous in $(b,z)$. 
Since $F$ is also holomorphic in $z$, it follows that $F\in \Cscr^{1,\infty}$ 
and therefore $f=F\circ\Phi\in \Cscr^{1,(k+1,\alpha)}$
(see Lemma \ref{lem:composition} (b)).
Suppose now that $l\ge 2$, so $k+1\ge l\ge 2$, 
$\psi\in \Cscr^2$, and $f\in \Cscr^{2,0}\cap \Cscr^{1,2}$. 
Differentiating the equation \eqref{eq:Fb} on $b$ gives
\[
	F_{bb} = f_{bb} + 2 f_{bx}\psi_b + f_{xx}(\psi_b)^2
	+ f_x \psi_{bb}, 
\]	
and $F_{bb}$ is continuous in $(b,z)$.
Since it is holomorphic in $z$, it follows that 
$F\in \Cscr^{2,\infty}$ and therefore $f\in \Cscr^{2,(k+1,\alpha)}$.
This process can be continued up to $l=k+1$ but not beyond.
\end{proof}

We continue with the proof of Theorem \ref{th:Runge}. 
With $\Phi$ as in \eqref{eq:Phi}, the set 
\[
	\wt K =\Phi(B_0\times K) \subset B_0\times X
\]
is Runge in $B_0\times X$. Indeed, its fibre 
$\wt K_b=\Phi_b(K)$ is compact and Runge in 
$\Phi_b(\Omega)$ for every $b\in B_0$, and since 
$\Phi_b(\Omega)$ is Runge in $X$ (see Remark \ref{rem:homotopic}),
$\wt K_b$ is Runge in $X$ as well. Furthermore, the fibres 
$\wt K_b$ depend continuously on $b\in B_0$. Recall that $U$ is an 
open neighbourhood of $B\times K$ and $f:U\to\C$
is a $\Jscr$-holomorphic function of class $\Cscr^{l,0}$. 
Pick an open subset $V\subset B_0\times X$
such that $\overline V\cap (B_0\times X) \subset U\cap (B_0\times \Omega)$
and set $\wt V=\Phi(V)\subset B_0\times X$. We have 
$f=\tilde f \circ \Phi$ where by Lemma \ref{lem:Xholomorphic}
the function $\tilde f=f\circ\Phi^{-1}:\wt V\to \C$ is fibrewise
$J_{b_0}$-holomorphic and of class $\Cscr^{l,\infty}$.
By Lemma \ref{lem:Ck}, for any given $s\in \Z_+$ we can approximate 
$\tilde f$ in $\Cscr^{l,s}(\wt K)$ by a fibrewise $J_0$-holomorphic function 
$\wt F:B_0\times X\to \C$ of class $\Cscr^{l,\infty}$.
If $s$ is chosen big enough and the approximation is close enough  
then the function $F=\wt F\circ\Phi:B_0\times \Omega\to \C$
is $\Jscr$-holomorphic, of class $\Cscr^{l,(k+1,\alpha)}$,
and it approximates $f$ in the $\Cscr^{l,(k+1,\alpha)}$ topology 
on $B_0\times K$.
This gives a locally finite open cover $B_j$ of $B$ 
such that $f$ can be approximated as closely as desired 
in the $\Cscr^{l,(k+1,\alpha)}$ topology on $B_j\times K$  
by $\Jscr$-holomorphic functions $F_j:B_j\times \Omega\to \C$ 
of class $\Cscr^{l,(k+1,\alpha)}$. 
Choose a $\Cscr^l$ partition of unity $\{\chi_j\}_j$ 
on $B$ with $\supp\,\chi_j\subset B_j$ for every $j$. 
Assuming that $F_j$ is close enough to $f$ in
$\Cscr^{l,(k+1,\alpha)}(B_j\times K)$ for every $j$, 
the $\Jscr$-holomorphic function $F:B\times \Omega\to\C$ defined by 
\eqref{eq:partition} is of class $\Cscr^{l,(k+1,\alpha)}$ 
and it satisfies the required approximation condition.

It remains to obtain the interpolation conditions (c)
at the points of $A'=A\cap L =\{a_1,\ldots,a_m\} \subset \mathring K$. 
It suffices to explain this in the local situation given by Lemma \ref{lem:Ck};
the subsequent steps in the proof preserve this condition up to order $k+1$. 
In view of \eqref{eq:Phi} and \eqref{eq:phia}, the points
of $A'$ are fixed under the maps $\Phi_b$.
Choose $r\in \Z_+$ and set $n=m(r+1)$; this is the complex dimension 
of the space of complex $r$-jets 
of holomorphic functions on $X$ at the points of $A'$. 
By the classical function theory on 
open Riemann surfaces, we can find a family of 
$J_{b_0}$-holomorphic functions $\xi_t:X\to\C$, 
depending holomorphically on $t\in\C^n$,
such that $\xi_0=0$ and for every collection of $r$-jets at the points 
of $A'$ there is precisely one member 
$\xi_t$ of this family which assumes these $r$-jets at the 
given points. Let $F$ be a function in \eqref{eq:F}
which approximates $f$ to a given precision in the fine 
$\Cscr^{l,(k+1,\alpha)}$ topology on $B\times K$. Hence, the $r$-jets of 
the function $F_b=F(b,\cdotp)$ at the points of $A'$ are close 
to the respective $r$-jets of $f_b=f(b,\cdotp)$ for any $b\in B$.  
We subtract from each $F_b$ the appropriate uniquely determined member 
of the family $\xi_t$ so that the $r$-jets of the new function 
at the points of $A'$ agree with those of $f_b$ (i.e., the interpolation 
condition (c) holds.) This does not affect the approximation condition (b) 
very much since the jets of $\xi_t$ in question are close to those of the 
zero function, and hence $\xi_t$ is close to zero on $K$. 
This completes the proof of Theorem \ref{th:Runge}.
\end{proof}

%
%
\begin{proof}[Proof of Theorem \ref{th:Mergelyan}]
Let $f:B\times K\to\C$ be as in the theorem.
Set $f_b=f(b,\cdotp):K\to\C$ for any $b\in B$.
Fix $b_0\in B$. Choose a smoothly bounded domain 
$\Omega\Subset X$ containing $K$. Let $B_0\subset B$ 
be a neighbourhood of $b_0$ and 
\[
	\Phi:B_0\times \Omega \to \Phi(B_0\times \Omega)\subset B_0\times X
\]
be a map of class $\Cscr^{0,(1,\alpha)}$  
furnished by Corollary \ref{cor:Hamilton}. 
Thus, $\Phi(b,x)=(b,\Phi_b(x))$ where 
$\Phi_b:\Omega\to \Phi_b(\Omega)\subset X$ 
is a biholomorphism from $(\Omega,J_b)$ 
onto $(\Phi_b(\Omega),J_{b_0})$ satisfying 
$\Phi_b(a)=a$ for all $a\in A\cap L$, and $\Phi_{b_0}=\Id_\Omega$. 
Applying the Bishop--Mergelyan theorem
\cite{Bishop1958PJM} we can approximate the function
$f_{b_0}:K\to \C$ uniformly on $K$ by functions $\tilde f_{b_0}$ 
defined on open neighbourhoods $U\subset X$ of $K$. 
For $b$ close enough to $b_0$ the function 
$\tilde f_b:= \tilde f_{b_0}\circ \Phi_b$ is $J_b$-holomorphic
on a neighbourhood of $K$ and it approximates $f_b$ to a given 
precision uniformly on $K$.
This gives local approximation near any given point of $B$,
and the proof can be concluded by applying a continuous
partition of unity on $B$ as in the proof of Theorem \ref{th:Runge}.
Interpolation in the points of the finite set $A\subset \mathring K$
is handled as in the proof of Theorem \ref{th:Runge}.
\end{proof}

The same proof gives the following version of Theorem \ref{th:Mergelyan} 
with fine $\Cscr^0$ approximation on proper subsets of $B\times X$; 
compare with the Runge approximation Corollary \ref{cor:Runge}.

%
%
\begin{corollary}\label{cor:Mergelyan}
Assume that $B$ is a paracompact Hausdorff space, $X$
is a smooth open orientable surface, $\Jscr=\{J_b\}_{b\in B}$ 
is a family of complex structures on $X$ of class $\Cscr^{\alpha}$
for some $0<\alpha<1$, and $K\subset B\times X$
is a closed Runge subset (see Definition \ref{def:Runge}). 
Given a continuous function $f:K \cup (Q\times X)\to\C$ such that 
$f(b,\cdotp)$ is $J_b$-holomorphic on $\mathring K_b$ for all
$b\in B$ and is $J_b$-holomorphic on $X$ for all $b\in Q$, 
we can approximate $f$ in the fine $\Cscr^{0}$ topology on $K$
by $\Jscr$-holomorphic functions $F:B\times X\to \C$
satisfying $F=f$ on $Q\times X$.
\end{corollary}


%
%
%
%
\section{The Oka principle for maps from families of open Riemann surfaces
to Oka manifolds}
\label{sec:Oka}

In this section we prove Theorem \ref{th:Oka}. The same proof 
gives the generalisation in Theorem \ref{th:Okabis}.
We then obtain a couple of Mergelyan-type approximation theorems
for manifold-valued maps from families of open Riemann surfaces;
see Theorems \ref{th:Mergelyan-manifold} 
and \ref{th:Mergelyan-admissible}.
With future applications in mind, the technical results 
in Lemmas \ref{lem:graph} and \ref{lem:main} are obtained 
in the bigger generality when $X$ is a Stein manifold 
of arbitrary dimension. 

Recall that a compact set $K$ in a complex manifold $X$ is said to be 
a {\em Stein compact} if it admits a basis of open Stein neighbourhoods.
Every compact $\Oscr(X)$-convex set in a Stein manifold
$X$ is a Stein compact (see \cite[Theorem 5.1.6]{Hormander1990}). 
Given a complex manifold $Y$, we denote by 
\[
	\Oscrc(K,Y)
\] 
the space of continuous maps $K\to Y$ which are uniform 
limits of holomorphic maps on open neighbourhoods of $K$ in $X$.
Furthermore, we denote by
\[ 
		\Oscrcl(K,Y)
\]
the space of continuous maps $f:K\to Y$ with the property 
that every point $x \in K$ has an 
open neighbourhood $U\subset X$ such that 
$f|_{K\cap \overline U}\in \Oscrc(K\cap \overline U,Y)$. Clearly, 
we have the inclusions 
\[
	 \{f|_K: f\in \Oscr(K,Y)\} 
	 \subset \Oscrc(K,Y) \subset \Oscrcl(K,Y)\subset \Ascr(K,Y),
\] 
where $\Oscr(K,Y)$ is the space of holomorphic maps $U\to Y$ on 
open neighbourhoods $U\subset X$ of $K$ 
and $\Ascr(K,Y)$ is the space of continuous maps $K\to Y$
which are holomorphic in $\mathring K$.
The importance of the space $\Oscrcl(K,Y)$ lies in the following result 
of Poletsky \cite[Theorem 3.1]{Poletsky2013}.

%
%

\begin{theorem}[Poletsky \cite{Poletsky2013}] \label{th:Poletsky3.1} 
If $K$ is a Stein compact in a complex manifold $X$, $Y$ 
is a complex manifold and $f\in \Oscrcl(K,Y)$, then the graph of $f$ 
on $K$ is a Stein compact in $X\times Y$. 
\end{theorem}

We consider $\R^n$ as the standard real subspace 
of $\C^n$. Using Theorem \ref{th:Poletsky3.1} we prove 
the following.

%
%
\begin{lemma}\label{lem:graph}
Assume that $X$ is a Stein manifold, $\pi:\C^n\times X\to \C^n$ is 
the projection, $K\subset \C^n\times X$ is a compact set 
such that $B:=\pi(K)\subset \R^n$ and 
$K_b=\{x\in X:(b,x)\in K\}$ is $\Oscr(X)$-convex for every $b\in B$,
$U$ is an open neighbourhood of $K$ in $B \times X$, 
$Y$ is a complex manifold with a distance function $\dist_Y$
inducing its manifold topology, 
and $f:U\to Y$ is a continuous map
such that for every $b\in B$ the map 
$f_b=f(b,\cdotp):U_b=\{x\in X:(b,x)\in U\}\to Y$ is holomorphic. 
Then, the graph 
\begin{equation}\label{eq:graph}
	G_f=\{(b,x,f(b,x)) : (b,x)\in K\} \subset \C^n\times X\times Y
\end{equation} 
of $f$ on $K$ is a Stein compact in $\C^n\times X\times Y$ 
and $f\in \Oscrc(K,Y)$. Furthermore, given $\epsilon>0$ 
there are a neighbourhood $V$ of $K$ in $\C^n\times X$,
a neighbourhood $\wt U\subset U\cap V$ of $K$ in $B\times X$, 
a holomorphic map $\tilde f:V\to Y$, 
and a homotopy $g_t:\wt U \to Y$ $(t\in I=[0,1])$
satisfying the following conditions. 
\begin{enumerate}[\rm (a)]  
\item $g_0=f|_{\wt U}$ and $g_1=\tilde f|_{\wt U}$.
\item $g_t(b,\cdotp):\wt U_b \to Y$ is holomorphic 
for every $b\in B$ and $t\in I$. 
\item $\sup_K \dist_Y(g_t,f)<\epsilon$ for all $t\in I$. 
\end{enumerate}
If in addition $B$ is a $\Cscr^l$ submanifold of $\R^n$ 
(possibly with boundary) for some $l\in\N$
and $f\in \Cscr^{l,0}(U,Y)$, then for any $s\in\Z_+$ and after
shrinking $U\supset K$, the homotopy $f_t$ can be chosen such that, 
in addition to the above, $f_t\in \Cscr^{l,s}(U,Y)$ for all $t\in I$ and the 
approximation in (c) holds in $\Cscr^{l,s}(K)$.
\end{lemma}

\begin{proof}
By Remark \ref{rem:HC}, the set $K$ is 
$\Oscr(\C^n\times X)$-convex, whence a Stein compact.
We shall now verify that the map $f$ satisfies 
the conditions in Theorem \ref{th:Poletsky3.1}, and hence 
$G_f$ \eqref{eq:graph} is a Stein compact.

Fix a point $b_0 \in B$. Since $K$ is compact,
given a neighbourhood $V\subset X$ of the fibre $K_{b_0}$ we have 
\begin{equation}\label{eq:usc}
	K_{b}\subset V \ \ 
	\text{for all $b\in B$ sufficiently close to $b_0$}.
\end{equation}
Since the map $f_{b_0}:U_{b_0} \to Y$ is holomorphic
and $K_{b_0}$ is a Stein compact, 
the graph $G_{b_0}=\{(b_0,x,f_{b_0}(x)): x\in K_{b_0}\}$ 
has an open Stein neighbourhood 
$\Gamma \subset\C^n\times X\times Y$ by Siu's theorem \cite{Siu1976}
(see also \cite{Coltoiu1990}, \cite[Theorem 1]{Demailly1990}, 
and \cite[Theorem 3.1.1]{Forstneric2017E}). 
Choose a holomorphic embedding $\Theta:\Gamma\hra\C^N$.
By a theorem of Docquier and Grauert \cite{DocquierGrauert1960}
(see also \cite[Theorem 3.3.3]{Forstneric2017E})
there are a Stein neighbourhood $O\subset \C^N$
of $\Theta(\Gamma)$ and a holomorphic retraction 
$\rho:O\to \Theta(\Gamma)$. 
In view of \eqref{eq:usc} there is a compact neighbourhood 
$B_0\subset B$ of $b_0$ such that, setting 
\[
	S:=\{(b,x): b\in B_0,\ x\in K_b\} \subset U,
\]
we have that 
$
	\wt S:=\{(b,x,f(b,x)): (b,x)\in S\} \subset \Gamma.
$
Hence, the map  
\begin{equation}\label{eq:map-h}
	 h(b,x) := \Theta(b,x,f(b,x)) \in O\subset \C^N
\end{equation}
is well-defined on a neighbourhood of $S$ in $B\times X$,
and $h(b,\cdotp)$ is holomorphic on a neighbourhood of $K_b$
in $X$ for every $b\in B_0$.
By Lemma \ref{lem:Ck} we can approximate $h$ as closely as desired 
uniformly on $S$ by a holomorphic map $\tilde h:W\to\C^N$ 
from a neighbourhood $W\subset \C^n\times X$ of $S$. 
Assuming that the approximation is close enough and the
neighbourhood $W\supset S$ is small enough, 
we have that $\tilde h(W)\subset O$. 
Let $\tau:\C^n\times X\times Y\to Y$ denote the projection. The map 
\begin{equation}\label{eq:inY}
	\tilde f := \tau\circ \Theta^{-1} \circ \rho \circ \tilde h: W\to Y
\end{equation}
is then well defined and holomorphic, and 
it approximates $f$ uniformly on $S$.
Since this holds for every $b_0\in B$, we see that 
$f\in \Oscrcl(K,Y)$. Hence, Theorem \ref{th:Poletsky3.1} 
implies that $G_f$ \eqref{eq:graph} is a Stein compact.

To prove that $f\in \Oscrc(K,Y)$ and the last statement in the lemma, 
we apply the same argument with the entire parameter space $B$. 
Choose a Stein neighbourhood 
$\Gamma\subset \C^n\times X\times Y$ of $G_f$ \eqref{eq:graph},  
a holomorphic embedding $\Theta:\Gamma\hra \C^N$,
and a holomorphic retraction $\rho:O\to \Theta(\Gamma)$
from a neighbourhood $O\subset \C^N$ of $\Theta(\Gamma)$.
The map $h$ given by \eqref{eq:map-h} is now 
defined on a neighbourhood $\wt U \subset (B\times X)\cap U$ of $K$, 
and the map $h(b,\cdotp)$ is holomorphic on $\wt U_b$ for every $b\in B$.
By Lemma \ref{lem:Ck} we can approximate $h$ as closely as desired 
uniformly on $K$ by a holomorphic map $\tilde h:V\to\C^N$ 
from a neighbourhood $V\subset \C^n\times X$ of $K$. 
We may assume that $\tilde h(V)\subset O$. The map 
$\tilde f:V\to Y$ given by \eqref{eq:inY} is then holomorphic
and approximates $f$ on $K$.
Furthermore, if $\tilde h$ is close enough to $h$ 
on $K$ and shrinking the neighbourhood $\wt U \supset K$
if necessary, the family of convex combinations
\begin{equation}\label{eq:ht}
	h_t = (1-t)h + t\tilde h : \wt U \to\C^N,\quad t\in I
\end{equation}
assumes values in $O$. The family of maps  
\begin{equation}\label{eq:gt}
	g_t = \tau\circ \Theta^{-1} \circ \rho \circ h_t: \wt U \to Y,
	\quad t\in I
\end{equation}
is then a homotopy from $g_0=f|_{\wt U}$ to $g_1=\tilde f|_{\wt U}$
with the stated properties.
The last statement of the lemma follows by the same argument,
using Lemma \ref{lem:Ck} with approximation in $\Cscr^{l,s}(K)$. 
\end{proof}

The next result is a version of Lemma \ref{lem:Ck}
for maps with values in an Oka manifold and with 
homotopies added to the picture. This is the main 
ingredient in the proof of Theorem \ref{th:Oka}. 

%
%
\begin{lemma}\label{lem:main}
Assume that $B''\subset\R^n$ is a neighbourhood retract
and $B_0\subset B_1\subset B \subset B'$ are compact subsets of $B''$, 
each contained in the relative interior of the next one.
Let $X$ be a Stein manifold, $\pi:\C^n\times X\to \C^n$ be the projection, 
and $K\subset \C^n\times X$ be a compact subset such that
$\pi(K)\subset B$ and the fibre $K_b=\{x\in X:(b,x)\in K\}$ 
is $\Oscr(X)$-convex for every $b\in B$.
Assume that $U$ is an open neighbourhood of $K$ 
in $B' \times X$, $Y$ is an Oka manifold, 
and $f:B' \times X \to Y$ is a continuous map
such that for every $b\in B$ the map $f_b=f(b,\cdotp):X\to Y$
is holomorphic on $U_b=\{x\in X:(b,x)\in U\}$. 
Fix $\epsilon>0$ and $s\in \Z_+$. After shrinking the open set $U\supset K$,
there is a homotopy $f_t : B\times X \to Y$ $(t\in I=[0,1])$ 
with the following properties. 
\begin{enumerate}[\rm (a)]
\item $f_0=f|_{B\times X}$.
\item $f_t(b,\cdotp) : X\to Y$ is holomorphic on $U_b$
for every $b\in B$ and $t\in I$.
\item $f_t$ approximates $f$ in $\Cscr^{0,s}(K)$ to precision $\epsilon$.  
\item $f_t(b,\cdotp)=f(b,\cdotp)$ for all $b\in B\setminus B_1$ and $t\in I$.
\item The map $f_1(b,\cdotp):X\to Y$ is holomorphic for every $b$
in a neighbourhood of $B_0$.
\end{enumerate}
%
%
If in addition $B'$ is a $\Cscr^l$ submanifold of $\R^n$ 
for some $l\in\N$ and $f\in \Cscr^{l,0}(B'\times X,Y)$, 
then for any $s\in\Z_+$ 
the homotopy $f_t$ can be chosen such that, in addition to the above, 
$f_t\in \Cscr^{l,s}(B\times X,Y)$ for all $t\in I$ and the 
approximation in (c) holds in $\Cscr^{l,s}(K)$.
\end{lemma}

\begin{proof}
We focus on the case $l=0,\ s=0$. It will be clear that the same proof 
gives the corresponding results in the general case by using the
corresponding versions of Lemmas \ref{lem:Ck} and \ref{lem:graph}.

By the assumption, there are a neighbourhood $V''\subset \C^n$ of $B''$
and a retraction $\rho:V''\to B''$ onto $B''$. The conditions
imply that $B'$ is a neighbourhood of $B$ in $B''$.
Since $\rho|_{B'}$ is the identity map, it follows that there is an open 
neighbourhood $V\subset\C^n$ of $B$ such that 
$V\subset V''$ and $\rho(V)\subset B'$. Replacing $f(b,x)$ by 
$f(\rho(b),x)$ extends $f$ to a continuous map $V\times X\to Y$, 
still denoted $f$. Since every compact subset $B$ of $\R^n$ 
is polynomially convex in $\C^n$ \cite[p.\ 3]{Stout2007}, 
$V$ may be chosen Stein. If $B'$ is a $\Cscr^l$ submanifold of $\R^n$
then the retraction $\rho$ as above always exists and can be chosen 
of class $\Cscr^l$.

We claim that there are an open neighbourhood $W\subset V\times X$
of $K$ in $\C^n\times X$
and a homotopy $g_t:W\to Y$ $(t\in I=[0,1])$ connecting $g_0=f$ 
to a holomorphic map $g_1:W\to Y$ such that  
%
%
\begin{equation}\label{eq:apponK}
	\sup_{(b,x)\in K} 
	\dist_Y(g_t(b,x),f(b,x))<\epsilon/2\ \ \text{holds for all $t\in I$}
\end{equation}
and the map $g_t(b,\cdotp):W_b\to Y$ is holomorphic 
for every $b\in B$ and $t\in I$. Note that Lemma \ref{lem:graph}
furnishes a homotopy $g_t$ with the desired properties on 
a neighbourhood of $K$ in $B\times X$; see \eqref{eq:map-h},
\eqref{eq:ht}, and \eqref{eq:gt}. In the present situation,
all maps in the construction of $g_t$ are defined on 
a neighbourhood of $K$ in $\C^n\times X$.
Hence, the same argument, using convex combinations as
in \eqref{eq:ht} and defining $g_t$ by \eqref{eq:gt},  
yields a desired homotopy on a neighbourhood 
$W\subset V\times X$ of $K$ in $\C^n\times X$.
%
%
Applying the same argument on a somewhat bigger
compact set $K'\subset U\cap (B\times K)$ with $\Oscr(X)$-convex fibres 
and containing $K$ in its relative interior, the Cauchy estimates 
show that \eqref{eq:apponK} can be upgraded to approximation
in the $\Cscr^{0,s}$ topology on $K$ for any given $s\in \Z_+$.
Furthermore, if $l>0$ then the same arguments give approximation
in the $\Cscr^{l,s}$ topology on $K$ (cf.\ Lemma \ref{lem:graph}).

Pick a smooth function $\chi:\C^n\times X\to [0,1]$ with support 
in $W$ such that $\chi=1$ on a smaller neighbourhood 
$W'\Subset W$ of $K$. With $g_t$ as above, 
consider the map $h_0:V\times X\to Y$ given by
\begin{equation}\label{eq:h}
	h_0(z,x) = g_{\chi(z,x)}(z,x)\quad \text{for $z\in V$ and $x\in X$}.
\end{equation}
For $(z,x)\in W'$ we have $\chi=1$ 
and hence $h_0|_{W'}=g_1|_{W'}$, which is a holomorphic map.
On $(V\times X)\setminus W$ we have $\chi=0$
and hence $h_0=g_0=f$. Furthermore, $h_0$ is homotopic to $f$
by the homotopy $I\ni t\mapsto g_{t\chi}$, and every map in
this homotopy has the same properties as $f$. 

Since $V$ is Stein, the set $K\subset V\times X$ is 
$\Oscr(\C^n\times X)$-convex (see Remark \ref{rem:HC}) 
and $Y$ is an Oka manifold, 
the main result of Oka theory (see \cite[Theorem 5.4.4]{Forstneric2017E}) 
furnishes a homotopy $h_t:V \times X \to Y$ $(t\in I)$ 
from $h_0$ to a holomorphic map $h_1:V \times X \to Y$ such that 
the homotopy $f_t:V\times X\to Y$ $(t\in I)$ given by 
\[
	f_t = \begin{cases} g_{2t\chi}, & 0\le t\le 1/2, \\
				       h_{2t-1},   & 1/2\le t\le 1
	        \end{cases}				
\]
satisfies conditions (a)--(c) in the lemma (where we take $s=0$ in (c)), 
and it satisfies condition (e) for all $b\in B$ since $f_1=h_1$. 
To obtain (d), choose a smooth function 
$\xi:\R^n\to [0,1]$ which equals $1$ on a neighbourhood of $B_0$ 
and vanishes on $B\setminus B_1$, and replace 
$f_t(b,\cdotp)$ by $f_{t\xi(b)}(b,\cdotp)$ for $b\in B$ and $t\in I$.

This proves the lemma for $l=s=0$. The same arguments apply
when $s>0$, and also for $l>0$ when $B$ is a $\Cscr^l$ 
submanifold of $\R^n$, noting that the approximation 
by holomorphic functions in $\Cscr^{l,s}(K)$ is furnished by 
Lemma \ref{lem:Ck} and the existence of a 
homotopy $\{g_t\}_{t\in I}$ with approximation in $\Cscr^{l,s}(K)$ 
is given by Lemma \ref{lem:graph}.
\end{proof}

%
%
\begin{proof}[Proof of Theorem \ref{th:Oka}]
We consider the case $l=k=0$ with approximation in the 
fine $\Cscr^0$-topology. The arguments in the general case are similar
by using the corresponding version of Lemma \ref{lem:main}.
As pointed out in the proof of Theorem \ref{th:Runge}, 
approximation in the fine $\Cscr^{(k+1,\alpha)}$-topology follows 
from the fine $\Cscr^0$ approximation on a somewhat bigger Runge set 
in view of the Cauchy estimates. 

Let $\Jscr=\{J_b\}_{b\in B}$ be a family of complex structures on $X$ as 
in the theorem. Recall (see Def.\ \ref{def:Runge})
that a closed subset $K\subset B\times X$ 
is Runge if the projection $\pi|_K:K\to B$ is proper 
and every fibre $K_b=\{x\in X:(b,x)\in K\}$ $(b\in B)$ is 
Runge in $X$. Thus, the set $K$ in the theorem is Runge. 

We first explain the proof in the case when the parameter space $B$
is compact. Let $K^0=K\subset B\times X$ 
and $f^0=f:B\times X\to Y$ be as in the theorem,
so $f^0$ is $\Jscr$-holomorphic on a neighbourhood of $K^0$.
Choose an increasing sequence of compact Runge
sets $K'_1\subset K'_2\subset\cdots \subset \bigcup_{j=1}^\infty K'_j=X$
such that every set is contained in the interior of the next one, 
and let $K^j=B\times K'_j \subset B\times X$ for $j=1,2,\ldots$. 
We choose $K'_1$ big enough such that $K^0\subset K^1$. 
Given a decreasing sequence $\epsilon_j>0$, 
we shall find a sequence of maps
$f^j:B\times X\to Y$ and homotopies $f^j_t:B\times X\to Y$
$(t\in I=[0,1])$ satisfying the following conditions for every $j=1,2,\ldots$.
\begin{enumerate}[\rm (i)] 
\item $f^j$ is $\Jscr$-holomorphic on a neighbourhood of $K^j$. 
(By the assumption, this also holds for $j=0$.) 
\item $f^j_0=f^{j-1}$ and $f^j_1=f^j$. 
\item $f^j_t$ is $\Jscr$-holomorphic on a neighbourhood of $K^{j-1}$
for all $t\in I$.
\item $\max_{K^{j-1}} \dist_Y (f^{j-1},f^j_t)<\epsilon_j$
for all $t\in I$. 
\item The homotopy $f^j_t(b,\cdotp)$ is fixed for all $b$
in a neighbourhood of $Q$ in $B$.
\end{enumerate}
Assuming that the sequence $\epsilon_j>0$ is chosen to converge to $0$
sufficiently fast, the limit map $F=\lim_{j\to\infty}f^j:B\times X\to Y$
exists and is $\Jscr$-holomorphic, it approximates $f$ as closely
as desired uniformly on $K$, and $F(b,\cdotp)=f(b,\cdotp)$
holds for all $b\in Q$. Furthermore,
the homotopies $f^j_t$ $(j\in\N,\ t\in I)$ can be assembled into a 
single homotopy $f_t:B\times X\to Y$ $(t\in I)$ from $f_0=f$ to $f_1=F$
such that $f_t$ is $\Jscr$-holomorphic on a neighbourhood 
of $K$, it approximates $f$ on $K$ for every $t\in I$, and it
is fixed for all $b\in Q$.

Every step in the induction is of the same kind, so it suffices 
to explain the initial step, that is, the construction of a homotopy
$f^1_t:B\times X\to Y$ $(t\in I)$ which is $\Jscr$-holomorphic 
on a neighbourhood of $K^0$, it approximates $f=f^0$ on $K^0$,
it is fixed for $b$ in a neighbourhood of the subset $Q\subset B$
(see condition (v) in the theorem),  
and such that the map $f^1_1=f^1$ is $\Jscr$-holomorphic on a 
neighbourhood of $K^1$. This is accomplished by a finite 
induction with respect to an increasing family of compact 
subsets of the parameter space $B$, which we now explain.

Recall that $K^0\subset K^1=B\times L$, where $L$ is a compact 
Runge set in $X$. If the subset $Q\subset B$ in condition (v) 
is nonempty, it has a compact neighbourhood $\wt Q\subset B$ 
such that the map $f^0_b$ is holomorphic on a neighbourhood 
of $L$ for every $b\in \wt Q$. Let $\wt K^0\subset B\times X$ 
denote the compact set with fibres
\begin{equation}\label{eq:wtK0}
	\wt K^0_b = \begin{cases} 
				L, & b\in \wt Q; \\
				K^0_b, & b\in B\setminus \wt Q.
			\end{cases} 
\end{equation}
Clearly, $K^0 \subset \wt K^0 \subset K^1$ and $\wt K^0$
is Runge in $B\times X$. If $Q=\varnothing$,
we take $\wt Q=\varnothing$ and hence $K^0 =\wt K^0$. 
Pick a smoothly bounded domain $\Omega\Subset X$ 
and domains $V, V'\subset X$ such that 
\begin{equation}\label{eq:incl1}
	L\subset V\subset V' \subset \Omega
\end{equation}
and the closure of each of these sets is contained in the interior 
of the next one. 
Fix a point $b_0\in B$. The conditions on $B$ 
imply that there is a neighbourhood $P''\subset B$ of $b_0$
which is an ENR (see Definition \ref{def:ENR}). 
We may therefore consider $P''$ as a 
neighbourhood retract in some $\R^n\subset\C^n$. 
By Corollary \ref{cor:Hamilton} there are a compact neighbourhood
$P'$ of $b_0$, contained in the interior of $P''$,  
and a continuous family of biholomorphic maps
$\Phi_b: (\Omega,J_b)\to (\Phi_b(\Omega),J_{b_0})$ 
$(b\in P')$ such that
\begin{equation}\label{eq:incl2}
	\Phi_b(V)\subset V' \subset \Phi_b(\Omega)
	\ \  \text{holds for every $b\in P'$.}
\end{equation}
Pick a compact neighbourhood $P\subset B$ of $b_0$ 
contained in the interior of $P'$. 
Let $K'\subset L'$ be compact subsets of $P\times X$ 
whose fibres over any point $b\in P$ are given by 
\[
	K'_b=\Phi_b(\wt K^0_b),\qquad L'_b=\Phi_b(L). 
\]
By \eqref{eq:wtK0}--\eqref{eq:incl2} we have that 
\[
	K'_b\subset L'_b\subset \Phi_b(V)\subset V'
	\quad \text{for all $b\in P$}.
\]
Consider the family of maps 
\[
	f'_b=  f_b\circ \Phi_b^{-1}: \Phi_b(\Omega) \to Y,\quad b\in P.
\]
Since $f_b$ is $J_b$-holomorphic on a neighbourhood
of $\wt K^0_b$ and the map $\Phi_b: (\Omega,J_b)\to (\Phi_b(\Omega),J_{b_0})$ 
is biholomorphic, $f'_b$ is $J_{b_0}$-holomorphic on a 
neighbourhood of $K'_b$ for every $b\in P$.
Pick a pair of smaller compact 
neighbourhoods $P_0\subset P_1\subset P$ of 
$b_0$, each of them contained in the interior of the next one.
Lemma \ref{lem:main}, applied with $X$ replaced by $V'\subset X$,
furnishes a homotopy of maps 
\[
	f'_{t,b}: V' \to Y\quad \text{for $b\in P$ and $t\in I$} 
\]
satisfying conditions (a)--(e) in the lemma 
with the sets $B_0\subset B_1\subset B$ replaced by 
$P_0\subset P_1\subset P$. In particular, 
$f'_{t,b}=f'_{0,b}=f'_b$ holds for $b\in P\setminus P_1$, 
the map $f'_{t,b}$ approximates $f'_b$ on $K'_b$ for $b\in P$,
and $f'_{1,b}$ is $J_{b_0}$-holomorphic on $V'$ for 
$b$ in a neighbourhood of $P_0$. 
By \eqref{eq:incl2} we have $\Phi_b(V)\subset V'$. Hence, 
\begin{equation}\label{eq:ftb}
	f_{t,b}:=f'_{t,b} \circ \Phi_b: V \to Y\quad 
	\text{for $b\in P$ and $t\in I$} 
\end{equation}
is a homotopy of maps which are $J_b$-holomorphic on 
a neighbourhood of $\wt K^0_b$, 
they approximate $f_b$ uniformly on $\wt K^0_b$,
they agree with $f_{0,b}=f_b$ for $b\in P\setminus P_1$
(so we can extend the family to all $b\in B$),
and the map $f^1_b:=f_{1,b}:V \to Y$ is 
$J_b$-holomorphic for all $b$ in a neighbourhood of $P_0$.
By using a cut-off function in the parameter of the homotopy,
we can extend the maps $f_{t,b}$ to $X$ without 
changing their values on a neighbourhood of $L$
(compare with \eqref{eq:h}).

If the sets $Q\subset \wt Q$ are nonempty, we make
another modification to the above homotopy in order to 
ensure condition (v) in the theorem. Choose a function 
$\chi:B\to [0,1]$ such that $\chi=1$ on $B\setminus \wt Q$
and $\chi=0$ on a neighbourhood of $Q$.
With $f_{t,b}$ as in \eqref{eq:ftb} we set
\[
	\tilde f_{t,b}=f_{t\chi(b),b}
	\quad \text{for $b\in B$ and $t\in I$}. 
\]	
For $b\in B$ in a neighbourhood of $Q$ we then have 
$\tilde f_{t,b}=f_{0,b}=f_b$ as desired, and the 
other required properties still hold.

What was just explained serves as a step in a finite induction 
which we now describe. 

The assumptions imply that there is a finite family of triples
$P_0^j\subset P_1^j \subset P^j$ $(j=1,2,\ldots,m)$ of compact 
sets in $B$ such that 
$\bigcup_{j=1}^m P_0^j=B$ and the above construction
can be performed on each of these triples with the same sets
in \eqref{eq:incl1}. The induction proceeds as follows. 

In the first step, we perform the procedure explained above on 
the first triple $(P^1_0,P^1_1,P^1)$ with the set $K^0$
and the map $g^0:=f^0=f$. 
The resulting map $g^1:B\times X \to Y$ is $\Jscr$-holomorphic 
on a neighbourhood of the compact set 
\begin{equation}\label{eq:S1}
	S^1 := \big[(P^1_0\times X)\cap K^1\big]
	\cup  K^0 
	\subset B\times X,
\end{equation}
and $g^1_b=f^0_b$ holds for all $b$ in a neighbourhood of $Q$.
Note that the fibre $S^1_b$ of $S^1$ 
over any point $b\in B$ is Runge in $X$. 
Indeed, we have $S^1_b=L$ for $b\in P^1_0$ and 
$S^1_b=K^0_b$ for $b\in B\setminus P^1_0$. Since $K^0_b\subset L$ 
for every $b\in B$, the set $S^1$ is compact and Runge in $B\times X$,
and we clearly have that $K^0\subset S^1\subset K^1=B\times L$.
Furthermore, we obtain a homotopy from $f^0=g^0$ to $g^1$ 
such that every map in the homotopy is $\Jscr$-holomorphic on
a neighbourhood of $K^0$ and it approximates $f^0$ there,
and the homotopy is fixed for $b$ in a neighbourhood of
$(B\setminus P^1_1)\cup Q$.

In the second step, the same argument is applied to the map $g^1$ 
on the triple $(P^2_0,P^2_1,P^2)$ 
but with $K^0$ replaced by the set $S^1$ in \eqref{eq:S1}. 
The resulting map $g^2: B\times X\to Y$ is $\Jscr$-holomorphic on a 
neighbourhood of the compact Runge set
\begin{equation}\label{eq:S2}
	S^2 = \big[((P^1_0\cup P^2_0)\times X)\cap K^1\big]
	\cup K^0 
	\subset B\times X.
\end{equation}
Note that $S^2_b=L$ for $b\in P^1_0\cup P^2_0$ and 
$S^2_b=S^1_b=K^0_b$ for $b\in B\setminus (P^1_0\cup P^2_0)$. 
We also obtain a homotopy from $g^1$ to $g^2$ consisting of
maps which are $\Jscr$-holomorphic on a neighbourhood 
of $S^1$, they approximate $g^1$ there, and the homotopy
is fixed for $b$ in a neighbourhood of $(B\setminus P^2_1)\cup Q$.

Proceeding inductively, we obtain after $m$ steps a map 
$g^m:B\times X\to Y$ which is $\Jscr$-holomorphic on a neighbourhood
of $S^m=K^1=B\times L$. We define $f^1:=g^m$.
Furthermore, the individual homotopies between the subsequent maps 
$g^j$ and $g^{j+1}$ for $j=0,1,\ldots,m-1$ can be assembled 
into a homotopy $f^1_t$ $(t\in I)$ from $f^1_0 = f^0=g^0$ 
to $f^1_1=f^1=g^m$ such that $f^1_t$ is $\Jscr$-holomorphic
on a neighbourhood of $K^0$ for all $t\in I$ and the homotopy
is fixed for $b\in B$ in a neighbourhood of $Q$. 
This completes the proof of the theorem if the parameter space 
$B$ is compact. 

In the general case when $B$ is only $\sigma$-compact 
we choose a normal exhaustion 
$B_1\subset B_2\subset \cdots\subset \bigcup_{j=1}^\infty B_j=B$
by compact sets (i.e., such that every set $B_j$ is contained in 
the interior of the next one) and a normal exhaustion 
$L_1\subset L_2 \subset \cdots\subset \bigcup_{j=1}^\infty L_j=X$
by compact Runge subsets of $X$ such that 
\[
	(B_{j}\times X)\cap K^0 \subset B_{j}\times L_j
	\quad\text{holds for all $j=1,2,\ldots$}.
\] 
Define the increasing sequence of subsets 
$K=K^0\subset K^1\subset \cdots 
\subset \bigcup_{j=0}^\infty K^j=B\times X$ by 
\[
	K^j = (B_j\times L_j) \cup K^0,  
	\quad j=1,2,\ldots.
\]
Note that every $K^j$ is a closed Runge set in $B\times X$.
Applying the special case proved above gives a sequence
of maps $f^j:B\times X\to Y$ $(j=0,1,\ldots)$ with $f^0=f$
such that for every $j=1,2,\ldots$ the map $f^j$ is $\Jscr$-holomorphic 
on a neighbourhood of $K^j$, it approximates $f^{j-1}$ in the
fine topology on $K^{j-1}$, it is homotopic to $f^{j-1}$
by a homotopy of maps which are $\Jscr$-holomorphic on a
neighbourhood of $K^{j-1}$ and approximate $f^{j-1}$ on $K^{j-1}$,
and the homotopy is fixed for $b\in Q$. 
Assuming that the approximation is close enough at every step,
we obtain a limit map $F=\lim_{j\to\infty} f^j:B\times X\to Y$ 
which is $\Jscr$-holomorphic, it approximates the initial map $f$ 
as closely as desired in the fine topology on $K=K^0$, 
it agrees with $f$ on $Q\times X$, and it is homotopic to $f$ 
by maps having the same properties. 
\end{proof}

%
%
The following result generalises Theorem \ref{th:Oka}.
The proof is essentially the same and is omitted.

\begin{theorem} \label{th:Okabis}
Let $B$, $X$, $\{J_b\}_{b\in B}$, $K\subset B\times X$, and $Y$ be 
as in Theorem \ref{th:Oka}. Assume that $Z$ is a
Stein manifold and $L$ is a compact $\Oscr(Z)$-convex set in $Z$.
For every $b\in B$ let $\wt J_b$ be the almost 
complex structure on $X\times Z$
which equals $J_b$ on $TX$ and equals the given 
almost complex structure on $TZ$. 
Assume that $f:B\times X \times Z \to Y$ is a continuous map, 
and there is an open set $U\subset B\times X\times Z$ containing 
$K\times L$ such that $f_b=f(b,\cdotp,\cdotp):X\times Z\to Y$ is 
$\wt J_b$-holomorphic on $U_b=\{(x,z) \in X\times Z:(b,x,z)\in U\}$ 
for every $b\in B$.
Given a continuous function $\epsilon:B\to (0,+\infty)$,  
there is a homotopy $f_t:B\times X\times Z\to Y$ $(t\in I)$ 
satisfying the following. 
\begin{enumerate}[\rm (i)]
\item $f_0=f$.
\item The map $f_{t,b}=f_t(b,\cdotp,\cdotp):X\times Z\to Y$ is 
$\wt J_b$-holomorphic near $K_b\times L$ for every $b\in B$.
\item $\sup_{(x,z) \in K_b\times L}\dist_Y(f_{t,b}(x,z),f_b(x,z))<\epsilon(b)$
for every $b\in B$ and $t\in I$.
\item The map $F=f_1$ is such that 
$F_b=F(b,\cdotp,\cdotp):X\times Z\to Y$ is 
$\wt J_b$-holomorphic for every $b\in B$.
\end{enumerate}
If in addition $0\le l\le k+1$, $B$ is a $\Cscr^l$ manifold if $l>0$, 
the family $\{J_b\}_{b\in B}$ is of class 
$\Cscr^{l,(k,\alpha)}(B\times X)$ $(0<\alpha<1)$ and $f|_U$ is of class 
$\Cscr^{l,0}$, then $f|_U$ is of class $\Cscr^{l,(k+1,\alpha)}$ 
and the homotopy $\{f_t\}_{t\in I}$ can be chosen such that 
it is of class $\Cscr^{l,(k+1,\alpha)}$ on a neighbourhood of 
$K\times L$, it approximates $f$ in the fine 
$\Cscr^{l,(k+1,\alpha)}$-topology on $K\times L$, 
and $F=f_1$ is class $\Cscr^{l,(k+1,\alpha)}$ on $B\times X\times Z$. 
\end{theorem}

By using the techniques in the proof of Theorem \ref{th:Oka},
we can also extend Mergelyan approximation for functions 
in Theorem  \ref{th:Mergelyan} to manifold-valued maps
as in the following theorem. A similar result in 
the nonparametric case is 
\cite[Corollary 5, p.\ 176]{FornaessForstnericWold2020}.

%
%
\begin{theorem} \label{th:Mergelyan-manifold}
Assume that $X$ is a smooth open surface, 
$B$ is as in Theorem \ref{th:Oka},
$\Jscr=\{J_b\}_{b\in B}$ is a family of complex structures on $X$
of class $\Cscr^{\alpha}$ $(0<\alpha<1)$, 
$K\subset X$ is a compact Runge set, 
$A\subset \mathring K$ is a finite set, 
$Y$ is a complex manifold,  
and $f: B\times K \to Y$ is a continuous map which is
$\Jscr$-holomorphic on $B\times \mathring K$.
Given a continuous function $\epsilon:B\to (0,+\infty)$, 
there are a neighbourhood $U\subset B\times X$ 
of $B\times K$ and a continuous $\Jscr$-holomorphic map $F:U\to Y$ 
such that for every $b\in B$ 
we have $\sup_{x\in K}\dist_Y(F_b(x),f_b(x))<\epsilon(b)$ and 
$F_b$ agrees with $f_b$ to order $1$ in every point $a\in A$.
\end{theorem}

\begin{proof}
Since the open set $X\setminus K$ has no relatively compact
connected components, there are arbitrarily small open coordinate 
discs $U_1,\ldots, U_N\subset X$ and compact 
discs $D_j\subset U_j$ for $j=1,\ldots,N$ such that
$K\subset \bigcup_{j=1}^N \mathring D_j$ and 
$U_j\setminus (K\cap D_j)$ is connected for every $j$.
Fix $b_0\in B$. We may assume that the discs $U_j$ are chosen small 
enough so that $f_{b_0}(K\cap D_j) \subset Y$
is contained in a coordinate chart of $Y$ for each $j$. 
Hence, by Theorem \ref{th:Mergelyan} we can approximate $f_{b_0}$ 
uniformly on $K\cap D_j$ by holomorphic maps 
from open neighbourhoods of $K\cap D_j$ to $Y$ for $j=1,\ldots,N$.
This shows that the hypotheses of Theorem \ref{th:Poletsky3.1}
hold, so the graph of $f_{b_0}$ on $K$ is a Stein compact in $X\times Y$. 
By the argument in the proof of Lemma \ref{lem:graph}
we reduce the Mergelyan approximation problem for maps $f_b:K\to Y$,
with $b\in B$ close enough to $b_0$, to the scalar-valued case
furnished by Theorem \ref{th:Mergelyan}. The local 
$J_b$-holomorphic approximants of $f_b$ can be glued together
by finding homotopies as in the proof of Lemma \ref{lem:graph}
(see \eqref{eq:ht} and \eqref{eq:gt}) and using 
cut-off functions in the parameter of the homotopy.
The inductive procedure is similar to the one in 
the proof of Theorem \ref{th:Oka} and will not be repeated.
\end{proof}

Before stating our next result, we recall the following notion;
see \cite[p.\ 69]{AlarconForstnericLopez2021}.

%
%
\begin{definition} \label{def:admissible}
Let $X$ be a smooth surface. An {\em admissible set} 
in $X$ is a compact set of the form $S=K\cup E$, where $K$ is a 
(possibly empty) finite union of pairwise disjoint compact domains with
piecewise smooth boundaries in $X$ and $E = S \setminus\mathring  K$ 
is a union of finitely many pairwise disjoint smooth Jordan arcs and closed 
smooth Jordan curves meeting $K$ only at their endpoints (if at all) 
such that their intersections with the boundary $bK$ of $K$ 
are transverse.
\end{definition}

Admissible sets arise in handlebody decompositions of surfaces; see 
\cite[Sect.\ 1.4]{AlarconForstnericLopez2021}. For this reason, 
approximation on such sets plays a major role in 
constructions of directed holomorphic maps, minimal surfaces  
and related objects, as is evident from the results in
\cite{AlarconForstnericLopez2021}. The basic case  
for continuous functions follows from Theorem \ref{th:Mergelyan}. 
In Section \ref{sec:directed} we shall also use the following version
of the Mergelyan theorem on admissible sets in families
of open Riemann surfaces. 

%
%
\begin{theorem}
\label{th:Mergelyan-admissible}
Assume that $X$ is a smooth open surface, $1\le l\le k+1$ are integers,
$B$ is a manifold of class $\Cscr^l$,
$\Jscr=\{J_b\}_{b\in B}$ is a family of complex structures on 
$X$ of class $\Cscr^{l,(k,\alpha)}(B\times X)$ for some $0<\alpha<1$,  
$S=K\cup E$ is a Runge admissible set in $X$, $U\subset X$ is an open 
set containing $K$, and $f: B\times (U\cup E) \to\C$ is a function of 
class $\Cscr^{l}$ which is $\Jscr$-holomorphic on $B\times U$.
Then, $f$ can be approximated in the fine $\Cscr^l$ topology on 
$B\times S$ by $\Jscr$-holomorphic functions $F:B\times X\to \C$ 
of class $\Cscr^{l,(k+1,\alpha)}$.
%
The analogous result holds for maps to any complex manifold $Y$, 
where the approximating maps $F$ are defined on small open 
neighbourhoods of $B\times S$ in $B\times X$. If $Y$ is an Oka manifold
then there are maps $F:B\times X\to Y$ satisfying the same conclusion.
\end{theorem}

\begin{proof}
It suffices to prove the result locally in the parameter.
Thus, fix a point $b_0\in B$, a smoothly bounded 
domain $\Omega\Subset X$ containing $S$,
and a compact neighbourhood $B_0\subset B$ of $b_0$ 
for which Corollary \ref{cor:Hamilton} applies and gives
a family of $(J_b,J_{b_0})$-biholomorphic maps 
$\Phi_b:\Omega\to\Phi_b(\Omega)\subset X$ $(b\in B_0)$ 
of class $\Cscr^{l,(k+1,\alpha)}$.
We may assume that $B_0$ is a $\Cscr^l$ submanifold 
of $\R^n\subset\C^n$ for some $n\in\N$. 
Recall that the function $f_b$ in the theorem is $J_b$-holomorphic
on a neighbourhood $U\subset X$ of $K$ for every $b\in B_0$.
Since $B_0$ is compact, we may choose $U$ independent of $b\in B_0$. 

As in the proof of Theorem \ref{th:Oka},
this reduces the approximation problem for $b\in B_0$ to the
situation in Lemma \ref{lem:main} where the compact sets 
$\wt K, \wt E, \wt S$ in $B_0\times X\subset\C^n\times X$ have fibres
$K_b=\Phi_b(K)$, $E_b=\Phi_b(E)$, and $S_b=\Phi_b(S)=K_b\cup E_b$,
respectively. Note that $\wt K$ and $\wt S$ are holomorphically convex
in $\C^n\times X$ (see Remark \ref{rem:HC}).
Let $\wt U\subset B_0\times X$ be the set with fibres 
$U_b = \Phi_b(U)$. The function $\tilde f_b:U_b\cup E_b \to\C$, 
defined by $\tilde f_b\circ\Phi_b = f_b$ on $U\cup E$
$(b\in B_0)$, is $J_{b_0}$-holomorphic on $U_b$ for every $b\in B_0$. 
Let $\tilde f:\wt U\cup \wt E \to\C$ be given by 
$\tilde f(b,\cdotp)=\tilde f_b$ for $b\in B_0$. 
Note that $\tilde f$ is of class $\Cscr^{l}$.
Choose a compact $\Oscr(\C^n\times X)$-convex 
set $L\subset \wt U$ containing $\wt K$ in its relative interior.
By Lemma \ref{lem:Ck} 
we can approximate $\tilde f$ in $\Cscr^{l}(L)$
by a function $h:V\to\C$ on an open 
neighbourhood $V\subset \C^n\times X$ of $L$ 
which is holomorphic with respect to the standard
complex structure on $\C^n$ and the complex structure $J_{b_0}$
on $X$. By smoothly gluing $h$ with $\tilde f: \wt E \to\C$
on the set $L\setminus \wt K$, 
we may assume that $h$ is unchanged (and hence holomorphic) 
on a neighbourhood $\wt V\subset V$ of $\wt K$ in $\C^n\times X$, 
it is of class $\Cscr^{l}$ on $\wt E$,
it agrees with $\tilde f$ on $\wt E\setminus L$, 
and it approximates $\tilde f$ to a desired precision in $\Cscr^l(\wt E)$. 
Note that $\wt E$ is a totally real submanifold of class $\Cscr^l$
in $\C^n\times X$, and the set $\wt S=\wt K\cup \wt E$ is 
admissible in the sense of 
\cite[Definition 5 (a), p.\ 156]{FornaessForstnericWold2020}.
Hence, by \cite[Theorem 20, p.\ 161]{FornaessForstnericWold2020}
we can approximate $h$ in $\Cscr^l(\wt S)$ by holomorphic 
functions on $\C^n\times X$. By the argument in the proof of
Lemma \ref{lem:Ck} this gives functions $F$, defined and
$\Jscr$-holomorphic on open 
neighbourhoods of $B_0\times S$ in $B_0\times X$,
which approximate $f$ in $\Cscr^l(B_0\times S)$.
The proof is completed by using $\Cscr^l$ 
partitions of unity on $B$ and Theorem \ref{th:Runge}.
For maps to manifolds, we follow the argument in the proof 
Lemma \ref{lem:main}, and the statement for maps to 
an Oka manifold $Y$ follows from Theorem \ref{th:Oka}. 
\end{proof}

%
%
%
%
\section{Trivialisation of canonical bundles 
of families of open Riemann surfaces}
\label{sec:trivialisation}

Every open Riemann surface $X$ has trivial 
holomorphic cotangent bundle $K_X=T^*X$, 
trivialised by a nowhere vanishing holomorphic $1$-form $\theta$
on $X$. (In fact, every holomorphic vector bundle on 
an open Riemann surface is holomorphically trivial by the 
Oka--Grauert principle;
see Oka \cite{Oka1939}, Grauert \cite{Grauert1958MA},
and \cite[Theorem 5.3.1]{Forstneric2017E}.) 
We prove the following generalisation to families of complex structures.
See also Corollary \ref{cor:GunningNarasimhan}, which extends 
the theorem of Gunning and Narasimhan 
\cite{GunningNarasimhan1967}. 

%
%
\begin{theorem}\label{th:thetab}
Given a smooth open surface $X$ and a family $\Jscr=\{J_b\}_{b\in B}$
of complex structures of class $\Cscr^{l,(k,\alpha)}$ 
on $X$ as in Theorem \ref{th:Oka} (with $l\le k+1$),
there exists a family $\{\theta_b\}_{b\in B}$ of nowhere 
vanishing holomorphic $1$-forms on $(X,J_b)$
of class $\Cscr^{l,(k,\alpha)}(B\times X)$.
\end{theorem}

Note that a family of holomorphic 
$1$-forms $\{\theta_b\}_{b\in B}$ as in the theorem, which 
is of class $\Cscr^l$ in $b\in B$, is necessarily of class 
$\Cscr^{l,(k,\alpha)}(B\times X)$ by Lemma \ref{lem:Xholomorphic}.
Theorem \ref{th:thetab} is used in Section \ref{sec:directed}.

\begin{proof}
Write $X_b=(X,J_b)$ for $b\in B$. 
We see as in the proof of Corollary \ref{cor:Hamilton} that there is a family 
$\{\theta_b\}_{b\in B}$ of nowhere vanishing $(1,0)$-forms $\theta_b$ 
on $X_b$ of class $\Cscr^{l,(k,\alpha)}$. 
We will deform it to a family of nowhere vanishing 
$J_b$-holomorphic 1-forms $\{\tilde \theta_b\}_{b\in B}$
of class $\Cscr^l$ in the parameter $b\in B$.

Note that $\theta=\{\theta_b\}_{b\in B}$ is a section of the
complex line bundle $E\to B\times X$ whose restriction to
the fibre $X_b$ over $b\in B$ equals $T^*X_b$, the complex
cotangent bundle of $X_b$. We shall inductively deform $\theta$ 
so as to make it $\Jscr$-holomorphic
on larger and larger subsets of $B\times X$. 
We follow the scheme in the proof of Theorem \ref{th:Oka}. 
Assuming that $K\subset L$ are 
compact Runge sets in $X$ and $\theta$ is  
$\Jscr$-holomorphic on a neighbourhood of $B\times K$,
we shall find a multiplier $f:B\times X\to \C^*$
which is homotopic to the constant function $1$  
and is $\Jscr$-holomorphic and close to $1$ on a neighbourhood of 
$B\times K$ (thereby ensuring that $f\theta$ is 
close to $\theta$ on $B\times K$), such that 
$f\theta$ is $\Jscr$-holomorphic on a neighbourhood of $B\times L$.
We will then conclude the proof by an induction on a normal
exhaustion of $X$ by an increasing family of compact Runge sets.

It remains to explain the basic case described above.
As in the proof of Theorem \ref{th:Oka}, we proceed by induction
with respect to a normal exhaustion of $B$ by compact subsets.
For the inductive step, assume that $L$ is a compact Runge set in $X$ and  
$K^0\subset B\times L$ is a closed Runge subset (see
Definition \ref{def:Runge}). 
We allow for the possibility that some fibres are empty. 
Assume that $\theta$ as above is $\Jscr$-holomorphic on a 
neighbourhood $U\subset B\times X$ of $K^0$.
Pick a smoothly bounded domain $\Omega\Subset X$ 
with $L\subset \Omega$. Fix a point $b_0\in B$. 
Corollary \ref{cor:Hamilton} furnishes a compact 
neighbourhood $P\subset B$ of $b_0$ 
and a family of biholomorphic maps
$\Phi_b: (\Omega,J_b)\to (\Phi_b(\Omega),J_{b_0})$ $(b\in P)$
of class $\Cscr^{l,k+1}$.
By the Oka--Grauert principle there is a function $g:X\to\C^*$,
homotopic to the constant $X\to 1$ through functions
$X\to \C^*$, such that the 1-form $g\theta_{b_0}$ is 
$J_{b_0}$-holomorphic on $X$. Hence,
\[
	\phi_b:=\Phi_b^*(g\theta_{b_0}) = f_b\theta_b,\quad b\in P
\]
is a family of nowhere vanishing $J_b$-holomorphic 1-forms on $\Omega$.
Since $g\theta_{b_0}$ is independent of $b\in P$, 
the family $\{\phi_b\}_{b\in P}$ is of class $\Cscr^{l}(P)$. 
Hence, the same holds for the family of functions
$f_b=\phi_b/\theta_b:\Omega\to \C^*$. 
By shrinking $U\supset K^0$ if necessary 
we may assume that $U\subset B\times \Omega$.
Since $\theta_b$ is $J_b$-holomorphic on
$U_b\supset K^0_b$ for every $b\in P$, 
the function $f_b=\phi_b/\theta_b$
is also $J_b$-holomorphic on $U_b$ for every $b\in P$.
Theorem \ref{th:Oka} applied with the Oka manifold $Y=\C^*$  
furnishes a homotopy of functions
$
	f_{t,b}: \Omega \to \C^* \ (b\in P,\ t\in I) 
$
of class $\Cscr^{l,k}$ satisfying the following conditions:
\begin{enumerate}[\rm (i)]
\item $f_{0,b}=f_b$ for all $b\in P$,
\item $f_{1,b}$ is $J_{b}$-holomorphic on $\Omega$ for 
all $b\in P$ and of class $\Cscr^l$ in $b$, and 
\item $f_{t,b}$ is $J_{b}$-holomorphic on a neighbourhood of $K^0_b$
and it approximates $f_b$ on $K_b$ as closely as desired 
for all $b\in P$ and $t\in I$. (In fact, the approximation is 
in the $\Cscr^{l,k}$ topology.) 
\end{enumerate}
The homotopy of $1$-forms $\theta'_{t,b}= \phi_b/f_{t,b}$ $(b\in P,\ t\in I)$
on $\Omega$ is of class $\Cscr^{l,k}$ and satisfies 
\begin{enumerate}[\rm (i')]
\item $\theta'_{0,b}= \phi_b/f_{0,b}=\phi_b/f_{b}=\theta_b$ for all $b\in P$,
\item $\theta'_{1,b}= \phi_b/f_{1,b}$ is $J_b$-holomorphic
on $\Omega$ for every $b\in P$, and
\item $\theta'_{t,b}$ is $J_b$-holomorphic on a 
neighbourhood of $K^0_b$ and it approximates $\theta_b$ on $K^0_b$ 
for all $b\in P$. (The approximation is in the $\Cscr^{l,k}$ topology.) 
\end{enumerate}
Pick a pair of neighbourhoods $P_0\subset P_1\subset P$ of $b_0$, 
each contained in the interior of the next one, 
and a function $\xi:B\to [0,1]$ of class $\Cscr^l$ which equals $1$ on
a neighbourhood of $P_0$ and vanishes on $B\setminus P_1$.
We define a new homotopy of $1$-forms on $\Omega$ of class
$\Cscr^{l,k}(B\times \Omega)$ by 
\[
	\theta_{t,b}=\theta'_{t\xi(b),b}\ \ 
	\text{for every $b\in B$ and $t\in I$.}
\]
Then, $\theta_{t,b}=\theta'_{t,b}$ holds for $b$ in
a neighbourhood of $P_0$ (where $\xi=1$), and 
$\theta_{t,b}=\theta_{0,b}=\theta_b$ holds for 
all $b\in B\setminus P_1$ (where $\xi=0$) and $t\in I$. It follows that 
\begin{enumerate}[\rm (i'')]
\item $\theta_{0,b}=\theta'_{0,b}=\theta_b$ for all $b\in B$,
\item $\theta_{t,b}$ is $J_b$-holomorphic on a neighbourhood of 
$K^0_b$ and it approximates $\theta_b$ on $K^0_b$ 
for all $b\in B$ and $t\in I$ 
(the approximation is in the fine $\Cscr^{l,k}$ topology), and 
\item $\theta_{1,b}=\theta'_{1,b}$ is $J_b$-holomorphic on 
$\Omega$ for all $b\in P_0$.
\end{enumerate}
By using another cut-off function in the parameter of the homotopy,
we can extend $\theta_{t,b}$ for $p\in P$ to all of $X_b$ without 
changing its values on a neighbourhood of 
$(\mathring P\times L)\cup ((B\setminus \mathring P)\times X)$.

Using this device inductively with respect
to an exhaustion of $B$ as in the proof of Theorem \ref{th:Oka}, 
we can approximate $\theta$ in the fine $\Cscr^{l,k}$ topology on $K^0$ 
by a family of nowhere vanishing $1$-forms $\{\tilde \theta_b\}_{b\in X}$ 
of class $\Cscr^{l,k}(B\times X)$ which are $\Jscr$-holomorphic 
on a neighbourhood of $B\times L$. Theorem \ref{th:thetab} then 
follows by an obvious induction with respect to a normal exhaustion 
of $X$ by compact Runge sets.
\end{proof}

%
%
%
\section{Families of directed holomorphic immersions and of 
conformal minimal immersions}
\label{sec:directed}

In this section, we illustrate how the results of this paper 
can be used to construct families of directed holomorphic immersions
and of conformal minimal immersions 
from a family of open Riemann surfaces as in Theorem \ref{th:Oka}.
There are many further problems of this kind which may possibly or even 
likely be treated by these new methods, and we indicate a few of them 
in Problem \ref{prob:problems}. 

A connected compact complex submanifold $Y$ of the complex 
projective space $\CP^{n-1}$, $n\in\N$, determines the punctured 
complex cone 
\begin{equation}\label{eq:A}
	A= \{(z_1,\ldots,z_n)\in \C^n_* : [z_1:\cdots:z_n]\in Y\}.
\end{equation}
Note that $A$ is smooth and connected, and its closure 
$\overline A=A\cup\{0\} \subset \C^n$ is an algebraic subvariety 
of $\C^n$ by Chow's theorem \cite{Chow1949}.
By \cite[Theorem 5.6.5]{Forstneric2017E}, $A$ 
is an Oka manifold if and only if $Y$ is an Oka manifold.
By \cite[Lemma 3.5.1]{AlarconForstnericLopez2021},
the convex hull of $A$ 
is the smallest complex subspace of $\C^n$ containing $A$, 
and we shall assume without loss of generality that this hull is all of $\C^n$.

Let $X$ be a connected open Riemann surface and $\theta$ be a 
nowhere vanishing holomorphic 1-form on $X$.
A holomorphic immersion $h:X \to\C^n$ is said to be {\em directed by} $A$, 
or an {\em $A$-immersion}, if its complex derivative with respect to any 
local holomorphic coordinate on $X$ takes its values in $A$. 
Clearly, this holds if and only if the holomorphic map
$f=dh/\theta:X\to\C^n$ assumes values in $A$. 
Conversely, a holomorphic map $f:X\to A$ satisfying
the period vanishing conditions
\begin{equation}\label{eq:periodvanishing}
	\int_C f\theta =0\quad \text{for all closed curves $C\subset X$}
\end{equation}
integrates to a holomorphic $A$-immersion $h:X\to\C^n$ by setting
\[
	h(x)= v+\int_{x_0}^x f\theta,\quad x\in X
\]
for any $x_0\in X$ and $v\in\C^n$.
Since $f\theta$ is a holomorphic 1-form, it suffices to verify 
conditions \eqref{eq:periodvanishing} 
on a basis of the homology group $H_1(X,\Z)\cong \Z^r$, 
a free abelian group of some rank $r\in \Z_+ \cup\{\infty\}$.

Note that a map directed by the cone $A=\C^n_*$ is simply an immersion.
Another case of major interest is the {\em null quadric} 
\begin{equation}\label{eq:nullq}
	\boldA =\bigl\{(z_1,\ldots,z_n)\in \C^n_* =\C^n\setminus \{0\}: 
	z_1^2+z_2^2+\cdots + z_n^2=0\bigr\},\quad n\ge 3.
\end{equation}
Holomorphic immersions directed by $\boldA$ are called
{\em holomorphic null curves} in $\C^n$. 
The real and the imaginary part of a holomorphic null immersion 
$X\to\C^n$ are conformal harmonic immersions $X\to\R^n$. 
Such immersions parameterize minimal surfaces, hence
are called conformal minimal immersions.
Conversely, every conformal minimal immersion $X\to \R^n$ 
is locally (on any simply connected
domain) the real part of a holomorphic null curve. 
See \cite{AlarconForstnericLopez2021,Osserman1986} 
for more information.

Effective methods for construction directed holomorphic immersions 
were developed by Alarc\'on and Forstneri\v c 
in \cite{AlarconForstneric2014IM}. Assuming 
that $A$ is an Oka manifold, they proved an Oka principle with Runge 
and Mergelyan approximation for holomorphic $A$-immersions
\cite[Theorems 2.6 and 7.2]{AlarconForstneric2014IM}.  
They also showed that every holomorphic $A$-immersion can be
approximated by holomorphic $A$-embeddings when $n\ge 3$,
and by proper holomorphic $A$-embeddings under an additional  
assumption on the cone \cite[Theorem 8.1]{AlarconForstneric2014IM}.
Alarc\'on and Castro-Infantes \cite{AlarconCastro-Infantes2019APDE} 
added interpolation to the picture.  
A parametric Oka principle for $A$-immersions 
was proved in \cite[Theorem 5.3]{ForstnericLarusson2019CAG}. 
Algebraic $A$-immersions from affine Riemann surfaces 
are studied in \cite{AlarconLarusson2025Crelle} under the assumption 
that $A$ is algebraically elliptic in the sense of Gromov 
\cite{Gromov1989} (see also \cite[Definition 5.6.13]{Forstneric2017E}). 
Several cones arising in geometric applications, 
in particular the null quadric $\boldA$ \eqref{eq:nullq}, 
are algebraically elliptic. Recently, Alarc\'on et al.\ 
\cite{AlarconLarusson2025Crelle,AlarconForstnericLarusson2024}
obtained h-principles for algebraic immersions directed by cones 
which are flexible in the sense of Arzhantsev et al.\ 
\cite{ArzhantsevFlennerKalimanKutzschebauchZaidenberg2013}.
Minor variations of these results for the null cone \eqref{eq:nullq}
yield similar results for conformal minimal immersions 
of open Riemann surfaces in Euclidean spaces; 
see the monograph \cite{AlarconForstnericLopez2021}.

The main advantage of the techniques in the mentioned papers, 
when compared to the previously known results, is that they 
allow a complete control of the conformal structure of the 
resulting directed curves or minimal surfaces. 
By using the approximation results developed in the present paper,
one can go substantially further and construct families of such objects
with a control of the conformal structure of every member of
the family, which may depend continuously or smoothly on 
a parameter. We now present a few specific results in this direction,
which are only the tip of an iceberg of possibilities.

In the following, $X$ is a connected, smooth, open oriented surface, 
$\Jscr=\{J_b\}_{b\in B}$ is a family of complex structures on $X$
as in Theorem \ref{th:Oka}, and $\{\theta_b\}_{b\in B}$
is a family of nowhere vanishing $J_b$-holomorphic 1-forms
on $X$ furnished by Theorem \ref{th:thetab}.
Recall that a map $f:B\times X\to Y$ is said to be 
$\Jscr$-holomorphic if 
$f(b,\cdotp):X\to Y$ is $J_b$-holomorphic for every $b\in B$. 
The first two items in the following definition 
come from \cite[Definition 2.2]{AlarconForstneric2014IM}
and apply to any open Riemann surface $X$.

%
%
\begin{definition}\label{def:nondegenerate}
Let $A\subset \C^n_*$ be a smooth connected 
punctured complex cone of the form \eqref{eq:A}.
\begin{enumerate}[\rm (i)] 
\item 
A holomorphic map $f:X\to A$ is nondegenerate if 
the tangent spaces $T_{f(x)} A \subset T_{f(x)}\C^n \cong \C^n$ 
over all points $x\in X$ span $\C^n$. 
\item 
A holomorphic $A$-immersion $h:X\to\C^n$ is 
nondegenerate if the map $f=dh/\theta: X\to A$ is such,
where $\theta$ is any nowhere vanishing holomorphic 1-form on $X$.
\item
A $\Jscr$-holomorphic map $f:B\times X\to A$ is 
nondegenerate if $f_b=f(b,\cdotp):X\to A$ is nondegenerate 
for every $b\in B$.
\item
A map $h:B\times X\to\C^n$ is an $A$-immersion if 
$h_b=h(b,\cdotp):X\to\C^n$ is a $J_b$-holomorphic $A$-immersion 
for every $b\in B$, and is nondegenerate if $dh_b/\theta_b:X\to A$ 
is such for every $b\in B$.
\item
Let $S=K\cup E\subset X$ be an admissible set 
(see Definition \ref{def:admissible}).
A map $h:B\times S \to\C^n$ of class $\Cscr^{l,k}$ $(l\ge 0, k\ge 1)$ 
is a {\em generalized $A$-immersion} if for every $b\in B$
the map $h_b=h(b,\cdotp):S \to\C^n$ is an immersion 
whose derivative assumes values in $A$, and $h_b$ is 
$J_b$-holomorphic on $\mathring S=\mathring K$. 
Such an $h$ is nondegenerate if for every $b\in B$, 
$h_b$ is nondegenerate on every connected component 
of $K$ and of $E$. 
(See \cite[Definition 3.1.2]{AlarconForstnericLopez2021}.) 
\end{enumerate}
\end{definition}

Assuming that $A$ is not contained in any hyperplane of 
$\C^n$, we show in Lemma \ref{lem:nondegenerate} that for any $r\in\N$ 
the set of $r$-tuples of points in $A$ at which the tangent spaces 
to $A$ fail to span $\C^n$ is a closed algebraic 
subvariety of $A^r$, the Cartesian product of $r$ copies of $A$. 
From this and the identity principle it follows that 
a holomorphic map $X\to A$ from a connected open Riemann surface 
$X$ is nondegenerate if and only if its restriction to any nonempty 
open subset $U\subset X$ is such.
Nondegenerate holomorphic maps $X\to A$ clearly form an open 
subset of the space $\Oscr(X,A)$ of all holomorphic maps in the
compact-open topology. By Lemma \ref{lem:nondegenerate} 
this set is also dense in $\Oscr(X,A)$, and this also holds 
for families of maps depending continuously on a parameter
in a suitable topological space. 

A holomorphic map $f:X\to \boldA\subset\C^n$ from a connected open 
Riemann surface $X$ to the null cone \eqref{eq:nullq} 
is nondegenerate if and only if $f(X)$ is not contained in 
a ray of $\boldA$ (see \cite[Lemma 3.1.1]{AlarconForstnericLopez2021}). 

We shall prove the following h-principle for 
families of directed holomorphic immersions from a family of 
open Riemann surfaces. Compare with the h-principles
for maps from a fixed open Riemann surface in 
\cite[Theorem 2.6]{AlarconForstneric2014IM} and
\cite[Theorem 5.3]{ForstnericLarusson2019CAG}.  

%
%
\begin{theorem}\label{th:directed}
Assume that $A\subset\C^n$ is a connected smooth 
punctured Oka cone \eqref{eq:A} which is 
not contained in any hyperplane, $X$ is a connected smooth open surface, 
$B$ is a parameter space as in Theorem \ref{th:Oka},  
$\Jscr=\{J_b\}_{b\in B}$ is a family of complex structures on $X$ of class 
$\Cscr^{l,(k,\alpha)}$ $(k\ge 1,\ 0\le l\le k+1,\ 0<\alpha<1)$, 
and $\{\theta_b\}_{b\in B}$ is a family of nowhere vanishing 
$J_b$-holomorphic $1$-forms on $X$ furnished by Theorem \ref{th:thetab}.
Given a continuous map $f_0:B\times X\to A$,  
there is a nondegenerate $A$-immersion
$h:B\times X\to \C^n$ of class $\Cscr^{l,k+1}(B\times X)$ 
such that the map $f:B\times X\to A$ defined by 
$f(b,\cdotp)=dh_b/\theta_b$ for all $b\in B$ 
is homotopic to $f_0$. 
\end{theorem}

The analogue of Theorem \ref{th:directed} also holds with 
approximation on a Runge subset $K\subset B\times X$ 
(see Definition \ref{def:Runge}) and interpolation on 
$Q\times X$ (for a closed subset $Q\subset B$)
as in Theorem \ref{th:Oka} and Remark \ref{rem:Q}. 
Furthermore, it holds with approximation on $B\times S$ 
where $S=K\cup E$ is a Runge admissible subset of $X$ 
(see Definition \ref{def:admissible}); cf.\ 
Theorem \ref{th:Mergelyan-admissible}. 
These additions will be evident from the proof.

Taking $A=\C^*=\C\setminus \{0\}$ we obtain the following
corollary to Theorem \ref{th:directed}, which extends the 
Gunning--Narasimhan theorem \cite{GunningNarasimhan1967} 
to families of complex structures on a smooth open surface.
This result also holds with the addition of approximation
conditions described above.

%
%
\begin{corollary}\label{cor:GunningNarasimhan}
Given a smooth open surface $X$ and a family 
$\{J_b\}_{b\in B}$ of complex structures on $X$ as in 
Theorem \ref{th:directed}, there is a function 
$h:B\times X\to\C$ of class $\Cscr^{l,k+1}$ 
such that $h(b,\cdotp):X\to\C$ is a $J_b$-holomorphic immersion 
for every $b\in B$. 
\end{corollary}

If $h$ is as in the corollary then $|h(b,\cdotp)|^2$ is 
a smooth strongly subharmonic function on the Riemann surface
$(X,J_b)$ for every $b\in B$. By using Theorem \ref{th:Runge}
it is easy to find a function $\rho:B\times X\to\R_+$ of the form 
$\rho=\sum_i |f_i|^2$, where each $f_i:B\times X\to\C$ is 
$\Jscr$-holomorphic, satisfying the following.

%
%
\begin{corollary}\label{cor:subharmonic}
Given a smooth open oriented surface $X$ and a family 
$\{J_b\}_{b\in B}$ of complex structures on $X$ as in 
Theorem \ref{th:directed}, there is a function $\rho:B\times X\to\R_+$
of class $\Cscr^{l,k+1}$ such that 
$\rho(b,\cdotp):X\to\R_+$ is a smooth strongly 
$J_b$-subharmonic exhaustion function for every $b\in B$.
\end{corollary}

%
%
\begin{proof}[Proof of Theorem \ref{th:directed}]
We shall adapt the proof of the parametric h-principle for directed 
holomorphic immersions from an open Riemann surface, given in 
\cite[Secs.\ 4--5]{ForstnericLarusson2019CAG}. 
(See especially the proofs of Theorems 4.1 and 5.3 in 
\cite{ForstnericLarusson2019CAG}, where the reader can 
find further details.) For the nonparametric case, see  
\cite[Theorem 2.6]{AlarconForstneric2014IM} and 
\cite[Theorem 3.6.1]{AlarconForstnericLopez2021}.

For simplicity of exposition, we assume that the parameter space 
$B$ is a compact neighbourhood retract in a Euclidean space $\R^m$ . 
The general case requires an additional induction with 
respect to a covering of $B$ by such compacts, which proceeds 
as in the proof of Theorem \ref{th:Oka}. 

By Theorem \ref{th:Oka} we can deform $f_0$ to a $\Jscr$-holomorphic 
map $f_1:B\times X\to A$ of class $\Cscr^{l,k+1}$. 
The following lemma shows that we can choose $f_1$
to be nondegenerate in the sense of Definition \ref{def:nondegenerate}. 

%
%
\begin{lemma}\label{lem:nondegenerate}
(Assumptions as in Theorem \ref{th:directed}.) Every $\Jscr$-holomorphic 
map $f:B\times X\to A$ of class $\Cscr^{l,k+1}$ 
can be approximated in the $\Cscr^{l,k+1}$ topology on 
compacts by nondegenerate $\Jscr$-holomorphic maps 
of class $\Cscr^{l,k+1}$ homotopic to $f$. If in addition $Q$
is a closed subset of $B$ and the map $f(b,\cdotp):X\to A$ 
is nondegenerate for every $b\in Q$ then the deformation 
with the stated properties can be chosen to agree with $f$
on $Q\times X$.
\end{lemma}

\begin{proof}
Since the cone $A$ is not contained in any hyperplane of $\C^n$,
there is an integer $q\in \N$ such that for a generic $q$-tuple of points 
$z_1,\ldots,z_q$ in $A$ we have $\sum_{i=1}^q T_{z_i}A=\C^n$. 
(We consider the tangent space $T_zA$ as a subspace of $\C^n$.) 
To deform a map $f:X\to A$ to a nondegenerate one, it suffices 
to push a finite subset of $f(X)$ to such a generic position. 
A procedure for a single map is given in 
\cite[Theorem 2.3 (a)]{AlarconForstneric2014IM};  
for a family of maps from an open Riemann surface, 
see \cite[Theorem 5.4]{ForstnericLarusson2019CAG}.
Since the proof of the latter result does not include the details
on how to apply the general position argument,
I take this occasion to provide the details also for 
a variable family of complex structures on $X$.

Assume as before that $B$ is a compact subset of $\R^m$;
the general case will follow by using cutoff functions in the parameter
$b\in B$ as will be clear from the proof. 
For any integer $r\in\N$ we define a subvariety 
$I_r\subset A^r=\overbrace{A\times \cdots\times A}^{r\rm\ times}$ by
\begin{equation}\label{eq:Ir}
	I_r = \big\{z=(z_1,\ldots,z_r)\in A^r: 
		\sum_{i=1}^r T_{z_i}A \ne \C^n\big\}.
\end{equation}
Note that $I_r$ is invariant under the action of the symmetric group
$S_r$ on $r$ elements permuting the components of $A^r$.
Since $A\cup \{0\}$ is an algebraic subvariety of $\C^n$ having 
the origin as the only singular point, 
there are finitely many algebraic vector fields 
$V_1,\ldots, V_d$ on $\C^n$ that are tangent to $A$, 
vanish in $0\in \C^n$, and span the tangent space
$T_z A$ for every $z\in A$. 
Then, $I_r$ is the set of points $z=(z_1,\ldots,z_r)\in A^r$ 
at which the $n\times (rd)$ matrix with columns $V_i(z_j)$ for $i=1,\ldots,d$
and $j=1,\ldots,r$ has rank $<n$, so it is an algebraic subvariety of $A^r$.
Since $A$ is not contained in any hyperplane of $\C^n$,
there is an integer $q\in\N$ such that 
$I_{q}$ is a proper algebraic subvariety of $A^{q}$.

We claim that the codimension of $I_r$ in $A^r$
increases to infinity as $r\to\infty$.
This implies that for $r$ big enough, every continuous map 
$g=(g_1,\ldots,g_r):B \to A^r$ can be uniformly approximated by maps 
whose images avoid $I_r$. For smooth maps on open domains in $\R^m$, 
this is an immediate consequence of the transversality theorem
when $\mathrm{codim} (I_r,A^r) > m$. By the Tietze extension
theorem, the continuous map $g:B\to A^r$ extends to a 
continuous map $g:\R^m\to (\C^n)^r$. There is an open neighbourhood
$\Omega\subset (\C^n)^r$ of $A^r$ and a smooth retraction
$\rho:\Omega\to A^r$. Pick a neighbourhood $U\subset \R^m$ 
of $B$ such that $g(U)\subset \Omega$. 
Then, the map $\rho\circ g:U \to A^r$ 
is a continuous extension of $g$.
Approximating this map uniformly on $B$ gives smooth maps 
$U \to A^r$ to which the previous argument applies.

The claim that $\lim_{r\to\infty} \mathrm{codim}(I_r,A^r)=\infty$ is especially 
simple to prove when $A$ is the null quadric \eqref{eq:nullq}
in $\C^n$ for $n>2$. In this case, $I_r=\wt I_r \cap A^r$ where 
$\wt I_r\subset (\C^n_*)^r$ consists of all $r$-tuples $z=(z_1,\ldots,z_r)$ 
of $\C$-collinear vectors $z_i\in \C^n_*$. For $r=2$, the collinearity 
condition on $z_1=(z_{1,1},\ldots,z_{1,n})\in\C^n_*$ 
and $z_2=(z_{2,1},\ldots,z_{2,n})\in\C^n_*$ is locally expressed by $n-1$ 
independent algebraic equations, given by vanishing of certain 
minors of the $2\times n$ matrix of coefficients. 
For $r\ge 2$ we have locally $(r-1)(n-1)$ independent equations
for $\wt I_r$, so this is the codimension of $\wt I_r$ in $(\C^n_*)^r$. 
Since $A^r$ has codimension $r$ in $(\C^n_*)^r$, 
the codimension of $I_r$ in $A^r$ is at least $(r-1)(n-1)-r=r(n-2)-n+1$. 
Since $n>2$, this grows to infinity with $r$.

In the general case we argue as follows. Assume that $r \in \N$ is big enough
such that $I_r$ is a proper subvariety of $A^r$. Set $s=\dim I_r$,
so $\mathrm{codim} (I_{r},A^{r})=r\dim A - s$. 
Let $I'$ be an irreducible component of $I_r$ of 
maximal dimension $s$. Since $I'$ is a proper subvariety of $A^r$,
up to a permutation of the components there are points 
$(z_1,\ldots,z_r)\in I'$ and $z'_r\in A\setminus \{z_r\}$ such that 
$(z_1,\ldots,z_{r-1},z'_r)\in A^r\setminus I_r$. It follows that 
$(z_1,\ldots,z_{r-1},z_r,z'_r)\in A^{r+1}\setminus I_{r+1}$.
Note that the projection $\pi:A^{r+1}\to A^r$ which deletes the 
last component maps $I^{r+1}$ to $I^r$.
It follows that $I_{r+1}\cap \pi^{-1}(I')$ is a proper subvariety 
of $I'\times A$, so its dimension is at most 
$\dim I'+\dim A -1=s +\dim A-1$. On an irreducible component
$I''$ of $I_r$ of dimension $\dim I''<s$ we have 
$\dim (I_{r+1}\cap \pi^{-1}(I'')) \le s-1 + \dim A$. This shows that 
\begin{eqnarray*}
	\mathrm{codim} (I_{r+1},A^{r+1}) & \ge & 
	(r+1)\dim A-(s-1 + \dim A) = r\dim A - s +1 \\
	&=& \mathrm{codim} (I_{r},A^{r})+1.
\end{eqnarray*}
This establishes the claim that 
$\lim_{r\to\infty} \mathrm{codim} (I_{r},A^{r})=\infty$.

Assume now that $f:B\times X\to A$ is a $\Jscr$-holomorphic map.
Choose $r\in \N$ big enough such that $\mathrm{codim}(I_r,A^r)>m$.
Pick a compact Runge set $K\subset X$
with nonempty interior and distinct points $x_1,\ldots, x_r\in \mathring K$,
and set $g_i=f(\cdotp,x_i):B\to A$ for $i=1,\ldots,r$.
By what was said above, we can approximate the maps $g_i$
uniformly on $B$ by maps $\tilde g_i: B\to A$ such that 
the image of the map $\tilde g=(\tilde g_1,\ldots,\tilde g_r):B\to A^r$ 
avoids the subvariety $I_r$ in \eqref{eq:Ir}. To prove the lemma,
it remains to show that if the above
approximations are close enough, there is a $\Jscr$-holomorphic map 
$\tilde f:B\times X\to A$ approximating $f$ on $B\times K$ 
such that $\tilde f(\cdotp,x_i):B\to A$ is so close to $\tilde g_i:B\to A$ 
for $i=1,\ldots,r$ that $\tilde f$ is nondegenerate.
This type of perturbation can be obtained by flows of vector fields
as described in \cite[proof of Theorem 5.4]{ForstnericLarusson2019CAG};
let us recall the main idea. 

Assume first that $A$ is flexible, 
in the sense that there exist finitely many $\C$-complete holomorphic 
vector fields $V_1,\ldots, V_d$ on $A$ which span the tangent space of $A$ 
at every point. (This holds in particular for the null quadric $\boldA$, 
cf.\ \cite[Example 4.4]{AlarconForstneric2014IM}.
A partial list of references to examples of flexible varieties
can be found in \cite[p.\ 394]{Forstneric2023Indag} 
preceding Example 6.3 ibid.) 
For each $i=1,\ldots, r$ choose a $\Jscr$-holomorphic function 
$\xi_i:B\times X\to\C$ such that 
$\xi_i(b,x_j)=\delta_{i,j}$ (the Kronecker delta) for $j=1,\ldots,r$;
see Theorem \ref{th:Runge}. Let $\phi^j_t$ denote the flow of $V_j$
for $j=1,\ldots,d$. For every $i=1,\ldots,r$ we consider the map 
$\Psi_i: B\times X \times \C^d \times A \to A$ given by 
\begin{equation}\label{eq:Psi-i}
	\Psi_i(b,x,\zeta,z)=
	\phi^1_{\zeta_1 \xi_i(b,x)} 
	\circ \cdots  \circ \phi^d_{\zeta_{d} \xi_{i}(b,x)} (z),
\end{equation}
where $\zeta=(\zeta_1,\ldots,\zeta_d)\in\C^d$. 
Note that $\Psi_i(b,\cdotp,\cdotp,\cdotp)$ is $J_b$-holomorphic
in $x\in X$ and holomorphic on $\C^d\times A$, and it satisfies 
the following conditions.
\begin{enumerate}[\rm (i)]
\item $\Psi_i(b,x,0,z)=z$ for all $(b,x) \in B\times X$ and $z\in A$.
\item $\Psi_i(b,x_j,\zeta,z)=z$ for all $b\in B$, $\zeta\in\C^d$, $z\in A$,
and $j\in \{1,\ldots,r\}\setminus \{i\}$. 
\item The differential 
$
	\di_\zeta \Psi_i(b,x_i,\zeta,z) \big|_{\zeta=0}:\C^d\to T_z A 
$
is surjective at $z=f(b,x_i)$ for all $b\in B$. 
\end{enumerate}
Let $\zeta_j=(\zeta_{j,1},\ldots,\zeta_{j,d})\in\C^d$ for $j=1,\ldots,r$.
We define the maps $\Phi_i: B\times X \times (\C^d)^i \times A \to A$ 
for $i=1,\ldots,r$ by taking $\Phi_1=\Psi_1$ and
\begin{equation}\label{eq:Phi-i}
	\Phi_{i}(b,x,(\zeta_1,\ldots,\zeta_{i}),z) =
	\Psi_{i}(b,x,\zeta_i,\Phi_{i-1}(b,x,(\zeta_1,\ldots,\zeta_{i-1}),z))
\end{equation}
for $i=2,\ldots,r$. Set $\zeta=(\zeta_1,\ldots,\zeta_r)\in (\C^d)^r$.
The properties (i)--(iii) of $\Psi_i$ imply that the map
\begin{equation}\label{eq:Thetab}
	\Theta_{i}(b)=\di_{\zeta_i} \Phi_r(b,x_i,\zeta,f(b,x_i)) \big|_{\zeta=0}:
	\C^d\to T_{f(b,x_i)} A 
\end{equation}
is surjective for every $b\in B$ and $i=1,\ldots,r$.
Assuming that a map $\tilde g_i:B\to A$ is close enough 
to $g_i=f(\cdotp,x_i)$ for every $i=1,\ldots,r$, the implicit function 
theorem gives functions $\zeta_i:B\to\C^d$ close to the origin 
for $i=1,\ldots,r$ such that the map $\tilde f:B\times X\to A$ given by
\[
	\tilde f(b,x) = \Phi_r(b,x,(\zeta_1(b),\ldots,\zeta_r(b)),f(b,x)) \in A
\]
is $\Jscr$-holomorphic, close to $f$ on $B\times K$, 
and satisfies $\tilde f(b,x_i)=\tilde g_i(b)$ for $i=1,\ldots,r$.
(More precisely, for every $i=1,\ldots,r$ we split the the trivial bundle 
in the form $B\times \C^d=E_i\oplus E'_i$, where 
$E'_i\subset B\times \C^d$ 
is the complex vector subbundle with fibres $E'_{i,b}= \ker\Theta_{i}(b)$
(see \eqref{eq:Thetab}), and apply the implicit function theorem 
to find sections $\zeta_i:B\to E_i$ with the above property for $i=1,\ldots,r$.)

If the cone $A$ is not necessarily flexible, 
the same procedure can be used to
obtain a nondegenerate $\Jscr$-holomorphic map $\tilde f:B\times K\to A$
with the stated properties. Indeed, the maps $\Psi_i$ \eqref{eq:Psi-i}
and $\Phi_i$ \eqref{eq:Phi-i} are now defined for the values of 
$\zeta$ in a ball around the origin whose radius may be chosen 
independent of $x\in K$, and the arguments carry over verbatim. 
The proof is then completed by approximating this map 
on $B\times K$ by a $\Jscr$-holomorphic map $B\times X\to A$,
using Theorem \ref{th:Oka}.

The interpolation condition on a closed subset $Q\subset B$ as in 
the lemma can be fulfilled by using a cutoff function in the parameter $b$
in the above construction of the deformation.  
\end{proof}

In the sequel, $f_1:B\times X\to A$ is a nondegenerate 
$\Jscr$-holomorphic map of class $\Cscr^{l,k+1}$ given by 
Lemma \ref{lem:nondegenerate}. 
Fix a complex structure $J$ on $X$ and a strongly $J$-subharmonic 
Morse exhaustion function $\rho:X\to \R_+$. There is an exhaustion 
$\varnothing = K_0\subset K_1\subset\cdots \subset 
\bigcup_{i=0}^\infty K_i=X$ by smoothly bounded compact Runge sets 
$K_i=\{x\in X:\rho(x)\le c_i\}$ for a sequence of regular values
$0<c_1<c_2\cdots$ of $\rho$, with $\lim_i c_i=+\infty$, 
such that $K_{1}\ne \varnothing$ and for every $i\in \Z_+$ the 
open set  
$ 
	D_i=\mathring K_{i+1}\setminus K_i
$
contains at most one critical point of $\rho$. 
(See \cite[Sect.\ 1.4]{AlarconForstnericLopez2021}.)
In particular, $K_1$ is a disc. 
Recall that $\{\theta_b\}_{b\in B}$
is a family of nowhere vanishing $J_b$-holomorphic 1-forms
on $X$ furnished by Theorem \ref{th:thetab}.
We shall inductively construct a sequence of 
open neighbourhoods $U_i\subset B\times X$ of
$B\times K_i$, maps $f_i:U_i\to A$ of class $\Cscr^{l,k+1}$, 
and numbers $\epsilon_i>0$ such that the 
following conditions hold for every $i=1,2,\ldots$. 
\begin{enumerate}[\rm (a)]
\item The map $f_{i,b}:U_{i,b}\to A$ is $J_b$-holomorphic 
and nondegenerate for every $b\in B$.
\item $\int_C f_{i,b}\theta_{b}=0$ for every closed curve $C\subset K_i$.
\item  $f_i$ is homotopic to $f_1|_{U_i}$ through maps $U_i\to A$.
\item $\|f_{i+1}-f_i\|_{\Cscr^{l}(B\times K_i)} < \epsilon_i$.
\item $0<\epsilon_{i+1}<\epsilon_i/2$, and 
if $f:B\times X\to A\cup\{0\}$ is an $\Jscr$-holomorphic map of class
$\Cscr^{l}$ such that $\|f-f_i\|_{\Cscr^{l}(B\times K_i)} < 2\epsilon_i$
then $f$ is nondegenerate and $f(B\times K_{i-1})\subset A$.
\end{enumerate}
Under these conditions, the limit map 
$f=\lim_{i\to\infty}f_i:B\times X\to A$ exists and is of class 
$\Cscr^{l}(B\times X)$, it is $\Jscr$-holomorphic (hence 
of class $\Cscr^{l,k+1}(B\times X)$ by Lemma \ref{lem:Xholomorphic}),
nondegenerate (see Definition \ref{def:nondegenerate}), 
homotopic to $f_0$, and 
$\int_C f(b,\cdotp) \theta_{b}=0$ holds for every closed curve 
$C\subset X$ and $b\in B$. Fixing a point $x_0\in X$, 
the map $h:B\times X\to\C^n$ given by 
\[
	h(b,x)=\int_{x_0}^x f(b,\cdotp)\theta_b, \qquad b\in B,\ x\in X
\]
is well-defined and satisfies $dh(b,\cdotp)=f(b,\cdotp)\theta_b$
$(b\in B)$, so it is a nondegenerate $A$-immersion.

We now explain the induction.
The assumptions imply that $K_1$ is a smoothly bounded compact disc.
Let $U_1$ be an open disc containing $K_1$, and let
$f_1:B\times X\to A$ be the initial nondegenerate map.
Assume inductively that $i\in \N$ and we have already found maps $f_j$ 
with the required properties for $j=1,\ldots,i$, and let us explain
how to obtain the next map $f_{i+1}$. We distinguish two cases.

{\em The noncritical case:} The domain 
$D_i=\mathring K_{i+1}\setminus K_i$ does not contain any critical point 
of $\rho$. 

{\em The critical case:} $D_i$ contains a unique (Morse) critical point 
of $\rho$. 

We begin with the noncritical case. 
Then, $K_{i}$ is a strong deformation retract of $K_{i+1}$
and $D_i$ is a finite union of annuli. In particular, the inclusion 
$K_i\hra K_{i+1}$ induces an isomorphism of their homology groups
$H_1(K_i,\Z)\cong H_1(K_{i+1},\Z)$. Assume that $K_i$ is connected;
the procedure that we shall explain can be performed independently
on every connected component. Fix a point $x_0\in\mathring K_i$.
There are finitely many smooth Jordan curves 
$C_1,\ldots,C_m\subset K_i$ such that any two of them
only intersect at $x_0$, they form a basis of the homology group
$H_1(K_i,\Z)$, and their union $C=\bigcup_{j=1}^m C_j$ is Runge in $X$.
The same curves then form a basis of $H_1(K_{i+1},\Z)$.
Consider the period map 
$\Pcal:B\times \Cscr(B\times C,A)\to (\C^n)^m$ given for any 
$b\in B$ and $f\in \Cscr(B\times C,A)$ by
\begin{equation}\label{eq:periodmap}
	\Pcal(b,f) = \biggl( \oint_{C_j} f(b,\cdotp)\theta_b \biggr)_{j=1,\ldots,m}
	\in (\C^n)^m. 
\end{equation}convex 
By condition (b) we have that $\Pcal(b,f_i)=0$ for all $b\in B$.
Since the map $f_i:B\times K_i\to A$ is nondegenerate, 
we can apply \cite[Lemma 5.1]{AlarconForstneric2014IM} (see also 
\cite[Lemma 3.2.1]{AlarconForstnericLopez2021}) to find a 
{\em period dominating spray} of $J_b$-holomorphic  maps 
\[
	F_i(\zeta,b,\cdotp) : K_i \to A \quad 
	\text{for $b\in B$},
\]
of class $\Cscr^{l}(B\times K_i)$, 
depending holomorphically on $\zeta=(\zeta_1,\ldots,\zeta_N)$ 
in a ball $\B \subset \C^N$, such that 
$F_i(0,\cdotp,\cdotp)=f_i$. (Recall that a map is called holomorphic
on a compact set if it is holomorphic in an open neighbourhood
of the said set.) The period domination property means that the map 
\begin{equation}\label{eq:period}
	\B \ni \zeta \longmapsto 
	\Pcal(b,F_i(\zeta,b,\cdotp))=
	\bigg( \oint_{C_j} F_i(\zeta,b,\cdotp) \theta_b 
	\bigg)_{j=1,\ldots,m} \in (\C^n)^m
\end{equation}
is submersive at $\zeta=0$, i.e., its differential at $\zeta=0$ 
is surjective for every $b\in B$. A map $F_i$ with these properties
can be chosen to be of the form 
\begin{equation} \label{eq:Psi}
	\Psi(\zeta,b,x)=
	\phi^1_{\zeta_1 h_1(b,x)} 
	\circ \cdots  \circ \phi^N_{\zeta_{N} h_{N}(b,x)} (f(b,x))
\end{equation}
where $\zeta=(\zeta_1,\ldots,\zeta_N)\in \C^N$ for a suitable $N\in\N$, 
$(b,x)\in B\times K_i$, every $\phi^j$ is the flow of one of the 
holomorphic vector fields $V_1,\ldots, V_d$ on $A$ 
(possibly with repetitions) which span the tangent space of 
$A$ at every point, and the functions $h_j\in \Cscr^{l,k+1}(B\times X)$ 
are $\Jscr$-holomorphic. 
We begin by choosing functions $h_{j}:B\times C\to \C$ 
of class $\Cscr^{l}$ to ensure that the 
map $\Psi$ in \eqref{eq:Psi} is period dominating on the curves in $C$; 
see \cite[proof of Lemma 3.2.1]{AlarconForstnericLopez2021}
for the details. By the parametric Mergelyan theorem 
(see Theorem \ref{th:Mergelyan} when $l=0$ 
and Theorem \ref{th:Mergelyan-admissible} when $l>0$) 
we can approximate the functions $h_{j}$ in $\Cscr^{l}(B\times C)$ 
by $\Jscr$-holomorphic functions of class $\Cscr^{l}(B\times K_i)$  
depending holomorphically on $\zeta$. 
By Lemma \ref{lem:Xholomorphic} these approximants are 
then of class $\Cscr^{l,k+1}(B\times K_i)$.
If the approximation is close enough then the resulting map 
has the stated properties.

For each $b\in B$ let $V_b\subset \C^N$ denote the kernel 
of the differential of the period map \eqref{eq:period} at $\zeta=0$.
This is a complex vector subspace of $\C^N$ with 
$\dim V_b=N-mn$ which is of class $\Cscr^l$ in $b\in B$.
Let $W_b\subset \C^N$ denote the orthogonal complement
of $V_b$. Fix a number $0<r<1$. Since $A$ is an Oka manifold and $K_i$ is a
strong deformation retract of $K_{i+1}$, Theorem \ref{th:Okabis} allows us 
to approximate $F_i$ in the $\Cscr^l$ topology 
on $r\B\times B\times K_i$ by a family of $J_b$-holomorphic maps 
$g(\zeta,b,\cdotp): K_{i+1} \to A$ $(\zeta \in r\B, \ b\in B)$
which are holomorphic in $\zeta$ and of the same regularity
class as $F_i$. If the approximation is sufficiently 
close, the implicit function theorem gives  
a map $\zeta: B \to \B$ of class $\Cscr^l(B)$, close to the 
zero map, such that $\zeta(b)\in W_b$ for all $b\in B$ and 
the $J_b$-holomorphic map 
\[
	f_{i+1}(b,\cdotp) := g(\zeta(b),b,\cdotp) :  K_{i+1}\to A
\] 
satisfies the period vanishing conditions $\Pcal(b,f_{i+1})=0$ 
in \eqref{eq:period} for every $b\in B$. If the approximations 
were close enough then the map $f_{i+1}$ is nondegenerate. 
To complete the induction step, we choose a number $\epsilon_{i+1}$
satisfying condition (e). 

Next, we consider the critical case. 
Let $x_i\in D_i$ be the unique critical point of $\rho$ in $D_i$.
Since $\rho$ is strongly subharmonic, the Morse index of
$\rho$ at $x_i$ is either $0$ or $1$. 

If the Morse index is $0$, the point 
$x_i$ is a local minimum of $\rho$, 
and hence a new connected component of the sublevel set $\{\rho\le t\}$
appears when $t$ passes the value $\rho(x_i)$. On this new 
component of $K_{i+1}$ we can take $f_{i+1}$ to be any nondegenerate
$\Jscr$-holomorphic map to $A$. 
On the remaining connected components of $K_{i+1}$ we 
proceed as in the noncritical case explained above.

Assume now that $\rho$ has Morse index $1$ at $x_i$. 
In this case, there is a smooth embedded arc 
$x_i \in E_i\subset D_i\cup bK_i$,  which is transversely attached 
with both endpoints $bE_i=\{p_i,q_i\} \subset bK_i$ to $K_i$ 
and is otherwise disjoint from $K_i$, such that $S_i=K_i\cup E_i$ 
is a Runge admissible set in $X$ (see Definition \ref{def:admissible}),  
and $S_i$ is a strong deformation retract of $K_{i+1}$. 
(See \cite[pp.\ 21--22]{AlarconForstnericLopez2021} for the details.) 
We assume that the Runge admissible set $S_i=K_i\cup E_i$ 
is connected, since on the remaining components of $K_i$
we are faced with the noncritical case described above. 
We orient $E_i$ so that $p_i$ is the initial endpoint and $q_i$ is the final
endpoint. The inductive hypothesis implies that 
there is an $A$-immersion $h_i:B\times K_i\to \C^n$
satisfying 
\[
	2\di_{J_b}h_i(b,\cdotp)=f_i(b,\cdotp)\theta_b
	\quad\text{for all $b\in B$}.
\]  
We now extend $f_i$ from a small open neighbourhood $U_i$ of 
$B\times K_i$ to a map $f_i:U_i\cup (B\times S_i)\to A$
of class $\Cscr^l$ such that the extended map is homotopic to $f_1$ 
through a homotopy that is fixed on $U_i$, and for every $b\in B$ 
the map $f_i(b,\cdotp):E_i\to A$ is nondegenerate 
(see Definition \ref{def:nondegenerate}) and satisfies
\begin{equation}\label{eq:intE_i}
	\int_{E_i} f_i(b,\cdotp) \theta_b = h_i(b,q_i)-h_i(b,p_i),\quad b\in B.
\end{equation}
An extension with these properties is found as follows. 
(For more details, see 
\cite[proof of Theorem 4.1]{ForstnericLarusson2019CAG} 
and note that the choice of the complex structure on $X$
does not play any role in these arguments.)

We subdivide the arc $E_i$ in two subarcs such that 
$E_i=E_{i,1}\cup E_{i,2}$ and $E_{i,1}\cap E_{i,2}=\{x_i\}$ 
(this point lies in the relative interior of $E_i$). 
We begin by extending $f_i$ from $U_i$ to a map 
$f'_i:B\times E_i\to A$
satisfying the required homotopy condition and such that 
it is nondegenerate on $B\times E_{i,2}$; this is possible by 
the transversality argument in the proof of 
Lemma \ref{lem:nondegenerate}. Choose a number
$\epsilon>0$ to be specified later. Using 
\cite[Lemma 3.1]{ForstnericLarusson2019CAG}, 
we modify $f'_i$ on the relative interior of $B\times E_{i,1}$
to a map $f_i'':B\times E_{i}\to A$ satisfying 
\begin{equation}\label{eq:intE_ialmost}
	\Big| \int_{E_i} f''_i(b,\cdotp) \theta_b  
	- \big(h_i(b,q_i)-h_i(b,p_i)\big)\Big| <\epsilon, \quad b\in B.
\end{equation}
In the final step, $f''_i$ is modified on the relative interior of 
$B\times E_{i,2}$ to obtain a map $f_i$ satisfying 
\eqref{eq:intE_i}. To do this, we 
construct a spray of maps with the core $f''_i$, with the
deformation supported on the relative interior of $B\times E_{i,2}$,
which is period dominating on the arc $E_i$ 
(this is possible since $f''_i$ is nondegenerate on $B\times E_{i,2}$).
Assuming as we may that $\epsilon>0$ was 
chosen small enough, a suitable map in this spray corrects 
the small error in \eqref{eq:intE_ialmost} to zero. 

We can now proceed as in the noncritical case.
Let $\Ccal=\{C_1,\ldots,C_m\}$ be a homology basis of 
the admissible set $S_i$ 
such that $C=\bigcup_{j=1}^m C_j$ is a Runge set
(see \cite[Lemma 1.12.10]{AlarconForstnericLopez2021}).
If the arc $E_i$ connects two different connected components 
of $K_i$, we add it to the family $\Ccal$. 
(In the opposite case, $E_i$ is a part of a new closed curve
in the homology basis of $S_i$.) 
As in the noncritical case, we construct a $\Ccal$-period
dominating spray $F_i:\B\times B\times S_i\to A$ 
of the form \eqref{eq:Psi}, where $\B\subset \C^N$ is a ball, such that 
$F_i(0,\cdotp,\cdotp)=f_i$ and the map $F_i(\cdotp,b,\cdotp)$ is 
holomorphic in the complex structure $\Jst \times J_b$ on $\B\times K_i$
for every $b\in B$. By Theorem \ref{th:Mergelyan-manifold}
(if $l=0$) or Theorem \ref{th:Mergelyan-admissible} (if $l>0$)
we can approximate $f_i$ in $\Cscr^l(B\times S_i)$ by 
$\Jscr$-holomorphic maps $f'_i:B\times V_i\to A$, where $V_i$
is a neighbourhood of $S_i$. Likewise, 
we can approximate the $\Jscr$-holomorphic functions $\xi_j$ 
in \eqref{eq:Psi} in $\Cscr^l(B\times S_i)$ by functions $\xi'_j$ 
which are $\Jscr$-holomorphic on $B\times V_i$. Pick a number $0<r<1$. 
Inserting these approximants in the expression \eqref{eq:Psi} for 
the spray $F_i$ and shrinking the neighbourhood $V_i$ around 
$S_i$ if necessary gives a map $G_i:r\B \times B\times V_i\to A$
approximating $F_i$ in $\Cscr^l(r\B\times B\times S_i)$ 
such that $G_i(\cdotp,b,\cdotp)$ is holomorphic in the complex structure 
$\Jst \times J_b$ on $r\B\times X$ for every $b\in B$. 

The final step is exactly as in the noncritical case,
and we obtain a nondegenerate $\Jscr$-holomorphic 
map $\tilde f_{i}:B\times V_i\to A$ approximating
$f_i$ in $\Cscr^l(B\times K_i)$ such that for every $b\in B$,
the $J_b$-holomorphic map $\tilde f_{i}(b,\cdotp):V_i\to A$ 
satisfies $\Pcal(b,\tilde f_{i})=\Pcal(b,f_i)$ 
(see \eqref{eq:periodmap}) on all curves in $\Ccal$.
(Hence, the periods of $\tilde f_{i}$ vanish
on all closed curves in $\Ccal$.) 
Next, we extend $\tilde f_{i}$ by approximation in $\Cscr^l(B\times S_i)$ 
to a $\Jscr$-holomorphic map $f_{i+1}:B\times K_{i+1}\to A$, 
keeping the periods on the curves in $\Ccal$ unchanged. 
This is accomplished by the noncritical case since 
$S_i$ has a compact neighbourhood $S'_i\subset V_i$ such 
that $K_{i+1}\setminus S'_i$ is an annulus. We conclude the induction 
step by choosing $\epsilon_{i+1}$ satisfying condition (e).
\end{proof}

The following h-principle for families of conformal minimal immersions
is an immediate corollary to Theorem \ref{th:directed}
applied with the null quadric $\boldA$ \eqref{eq:nullq}.

%
%
\begin{corollary}\label{cor:minimal}
Let $X$, $\{J_b\}_{b\in B}$, and $\{\theta_b\}_{b\in B}$ 
be as in Theorem \ref{th:directed}.
Given a continuous map $f_0:B\times X\to \boldA$
to the punctured null quadric $\boldA\subset\C^n$ \eqref{eq:nullq}
for some $n\ge 3$, there is a map $u:B\times X\to \R^n$ 
of class $\Cscr^{l,k+1}$ such that $u_b=u(b,\cdotp):(X,J_b)\to \R^n$ 
is a nonflat conformal minimal immersion for every $b\in B$, 
and the $\Jscr$-holomorphic map $f:B\times X\to \boldA$, 
defined by $f(b,\cdotp)=\di_{J_b} u_b/\theta_b$ 
for all $b\in B$, is homotopic to $f_0$. 
\end{corollary}

This result can be improved by adding approximation and 
prescribing the flux homomorphisms. We invite the reader to supply 
the precise statements and proofs of these generalisation.

By adding various global conditions on the maps in Theorem \ref{th:directed}
and Corollary \ref{cor:minimal} such as properness, embeddedness, or completeness, the construction methods become more intricate,
and we do not know whether they can be made in families.
We pose the following problems.

%
%
\begin{problem}\label{prob:problems}
Let $\{(X,J_b)\}_{b\in B}$ be a family of open Riemann surfaces 
as in Theorem \ref{th:Oka}.
\begin{enumerate}[\rm (a)]
\item Is there a continuous or a smooth family of proper 
$J_b$-holomorphic immersions $X\to\C^2$ 
and embeddings $X\hra\C^3$?
(The basic case is classical, see \cite[Theorem 2.4.1]{Forstneric2017E}
and the references therein. Without the properness condition,
the affirmative result is given by Theorem \ref{th:directed}.)

\item Assuming that $A\subset \C^n$ is an Oka cone \eqref{eq:A}, 
is there a continuous or a smooth 
family of proper $J_b$-holomorphic $A$-immersions or $A$-embeddings
$X\to\C^n$? (For the basic case, see \cite{AlarconForstneric2014IM}.
Without the properness or embeddednes condition, the affirmative 
answer is given by Theorem \ref{th:directed}.)
\item Is there a continuous or a smooth 
family of proper conformal harmonic immersions
$(X,J_b) \to \R^n$ for $n\ge 3$?
(For the basic case, see \cite[Theorem 3.6.1]{AlarconForstnericLopez2021} 
and the references therein.)
\item
Assume that $X$ is a bordered Riemann surface. Is there a family 
of complete conformal minimal immersions $(X,J_b)\to\R^n$ $(n\ge 3)$
with bounded images, i.e., does the Calabi--Yau phenomenon
for minimal surfaces hold in families? (For the nonparametric case, see 
\cite[Chapter 7]{AlarconForstnericLopez2021} and \cite{Alarcon2024Cantor}.) 
The analogous question can be asked for holomorphic (directed) immersions
$(X,J_b)\to\C^n$, $n\ge 2$ in the context of the problem asked by 
Yang \cite{Yang1977JDG}; see the survey by Alarc\'on \cite{Alarcon2024Yang}.
\item
Let $\eta=dz+\sum_{j=1}^n x_j dy_j$ be the standard complex
contact form on $\C^{2n+1}$, $n\ge 1$. Is there a continuous
or a smooth family of proper $J_b$-holomorphic Legendrian immersions
$f_b:X\to\C^{2n+1}$, that is, such that $f_b^*\eta=0$ holds 
for all $b\in B$? (For the basic case, see 
\cite{AlarconForstnericLopez2017CM}. For the parametric case 
for maps from a single Riemann surface and without the properness
condition, see \cite{ForstnericLarusson2018MZ}.) 
\end{enumerate}
\end{problem}

%
Note that problem (a) is a special case of (b) with the cone $A=\C^n_*$
for $n=2$ and $n=3$. 

After the completion of this paper, 
a partial affirmative answer to Problem \ref{prob:problems} (a) was given by 
Drinovec Drnov{\v s}ek and Kali{\v s}nik \cite{DrinovecKalisnik2026}.
They proved that for any smooth open surface $X$,  
endowed with a family of complex structures $\{J_b\}_{b\in B}$ 
depending continuously on the parameter $b$ in a metrisable space $B$,  
there is a continuous family of proper holomorphic maps 
$F_{b}:(X,J_b)\to\mathbb C^{2}$, $b\in B$. 
%
%

It was recently observed by Alarc\'on 
and the author \cite[p.\ 3]{AlarconForstneric2026proper} (May 2026)
that, at least for parameter spaces $B$ as in Theorem \ref{th:Oka},
the proof in \cite{DrinovecKalisnik2026} can be upgraded,
using also Corollary \ref{cor:GunningNarasimhan} with approximation on admissible Runge subsets, to obtain a continuous family of proper 
$J_b$-holomorphic maps $F_b=(F_{b,1},F_{b,2}):X\to \C^2$ such that every 
component function $F_{b,i}:X\to\C$ $(b\in B,\ i=1,2)$ 
is an immersion. In particular, $F_b:X\to\C^2$ is a proper immersion
for every $b\in B$. See also \cite[Corollary 1.3]{AlarconForstneric2026proper}
for an explicit result in this direction.

The paper \cite{AlarconForstneric2026proper} 
also gives partial affirmative answers to problems (b) and (c) above.
It is shown that if $B$ is a compact Euclidean neighbourhood retract
and $\Jscr=\{J_b\}_{b\in B}$ is a continuous 
family of complex structures of local H\"older class 
$\Cscr^\alpha$ $(0<\alpha<1)$ on an open surface $X$, then there is 
a continuous family of proper $J_b$-conformal minimal
immersions $u_b:X\to\C^3$, $b\in B$, with an arbitrary given 
continuous family of flux homomorphisms 
$\mathrm{Flux}_{u_b}:H_1(X,\Z)\to\R^3$; see 
\cite[Theorem 1.1]{AlarconForstneric2026proper}.
In particular, there is a $\Jscr$-holomorphic map 
$F=(F_1,F_2,F_3):B\times X\to \C^3$
such that for every $b\in B$, the map 
$F_b=F(b,\cdotp)=(F_{b,1},F_{b,2},F_{b,3}):X\to\C^3$
is a $J_b$-holomorphic null immersion and 
the map $(\Re F_{b,1},\Re F_{b,2}):X\to\R^2$ is proper;
see \cite[Corollary 1.2]{AlarconForstneric2026proper}. 

Proper embeddings of smooth Cartan manifolds 
of type $(m,n)$ (with complex leaves of dimension $n$ 
and real codimension $m$) into $\R^{2m}\times \C^{m+2n+1}$ 
were constructed by Jurchescu \cite[Sect.\ 7]{Jurchescu1988RRMPA}.

%
%
%
%
\medskip 
\noindent {\bf Acknowledgements.} 
The author is supported by the European Union (ERC Advanced grant HPDR, 101053085) and grants P1-0291 and N1-0237 from ARIS, Republic of Slovenia. I wish to thank Oliver Dragi\v cevi\'c, Paul Gauthier, 
Jure Kali\v snik, Yuta Kusakabe, 
Finnur L\'arusson, Jaka Smrekar, Sa\v so Strle, and 
Tja\v sa Vrhovnik for remarks which helped me to improve the
presentation. I also thank Francisco J.\ L\'opez
for having asked whether complex analysis could be used to 
construct families of minimal surfaces of prescribed conformal types
in Euclidean spaces. This question was the principal motivation for the paper,
and it is answered in part by Theorem \ref{th:directed} 
and Corollary \ref{cor:minimal}. Last but not least, 
I thank an anonymous referee for very helpful and pointed 
remarks which helped me to improve the presentation, and
for drawing my attention to the works of Jurchescu and others
on Cartan manifolds.




\end{document}